\documentclass[11pt,letterpaper]{amsart}

\usepackage{amsfonts, amsthm, amsmath, amssymb}
\usepackage{hyperref}
\hypersetup{colorlinks=false}

\usepackage[margin=3.5cm]{geometry}

\RequirePackage{mathrsfs} \let\mathcal\mathscr
  
\renewcommand{\leq}{\leqslant}

\renewcommand{\geq}{\geqslant}

\renewcommand{\d}{\mathrm{d}}

\usepackage{comment}
\usepackage{graphics}
\usepackage{aliascnt}

\usepackage{enumerate}
\usepackage{amsmath}
\usepackage{amsfonts}
\usepackage{amssymb}
\usepackage{amsthm}
\usepackage{comment}
\usepackage{mathtools}

\newtheorem{theorem}{Theorem}[section]
\newtheorem{corollary}[theorem]{Corollary}

\newtheorem{lemma}[theorem]{Lemma}
\newtheorem{proposition}[theorem]{Proposition}
\theoremstyle{definition}
\newtheorem{definition}[theorem]{Definition}
\newtheorem{remark}[theorem]{Remark}

\numberwithin{equation}{section}
\newtheorem*{ack}{Acknowledgements}

\renewcommand{\Re}{\real}
\renewcommand{\Im}{\im}

\renewcommand\AA{\mathbb{A}}

\newcommand\FF{\mathbb{F}}

\newcommand\RR{\mathbb{R}}
\newcommand\Z{\mathbb{Z}}
\newcommand\ZZ{\mathbb{Z}}
\newcommand\N{\mathbb{N}}
\newcommand\NN{\mathbb{N}}

\newcommand\CC{\mathbb{C}}

\newcommand\QQ{\mathbb{Q}}

\newcommand{\asum}{\sideset{}{^{\ast}}\sum}

\newcommand{\x}{\mathbf{x}}
\newcommand{\y}{\mathbf{y}}
\newcommand{\s}{\mathbf{s}}
\renewcommand{\c}{\mathbf{c}}

\renewcommand{\u}{\mathbf{u}}

\newcommand{\g}{\mathbf{g}}

\newcommand{\cX}{\mathcal{X}}

\newcommand{\ve}{\varepsilon}

\DeclareMathOperator{\vol}{vol}

\DeclareMathOperator{\Br}{Br}

\DeclareMathOperator{\n}{N}
\DeclareMathOperator{\real}{Re}
\DeclareMathOperator{\im}{Im}

\DeclareMathOperator{\nf}{\mathbf{N}}

\newcommand{\fo}{\mathfrak{o}}
\newcommand{\fa}{\mathfrak{a}}
\newcommand{\fb}{\mathfrak{b}}
\newcommand{\fc}{\mathfrak{c}}

\newcommand{\fp}{\mathfrak{p}}
\newcommand{\fq}{\mathfrak{q}}

\newcommand{\0}{\mathbf{0}}

\makeatletter
\@namedef{subjclassname@2020}{%
  \textup{2020} Mathematics Subject Classification}
\makeatother

\title[Bateman--Horn,  polynomial Chowla and the Hasse principle]{Bateman--Horn,  polynomial Chowla and the Hasse principle with probability 1}

\author{Tim Browning}
\address{IST Austria\\
Am Campus 1\\
3400 Klosterneuburg\\
Austria}
\email{tdb@ist.ac.at}

\author{Efthymios Sofos} 
\address{Dipartimento di Matematica\\ Universit\`a di Roma Tor Vergata\\00133, Roma, Italy}
\email{sofos@mat.uniroma2.it}

\author{Joni Ter\"av\"ainen}
\address{Department of Pure Mathematics and Mathematical Statistics\\ University of Cambridge\\
Cambridge CB3 0WB\\ UK}
\email{joni.p.teravainen@gmail.com}

\subjclass[2020]{11N32 (11G35, 14G05)}

\begin{document}

\begin{abstract}
With probability $1$, 
we assess the average behaviour of various arithmetic functions at the values of degree $d$ polynomials $f\in \mathbb{Z}[t]$ that are ordered by  height.  This allows us to establish averaged versions of the
Bateman--Horn conjecture, the polynomial Chowla conjecture and to address a basic question
about the integral Hasse principle for norm form equations.
Moreover, we are able to  quantify the error term in the asymptotics and the size of the exceptional set of $f$, both with arbitrary logarithmic power savings.  
\end{abstract}

\maketitle

\thispagestyle{empty}
\setcounter{tocdepth}{1}
\tableofcontents

\section{Introduction}

The average behaviour of arithmetic functions evaluated at polynomials $f\in \mathbb{Z}[t]$ has long been a central pursuit in analytic number theory.  For the 
von Mangoldt function $\Lambda$, this encodes the fundamental question of when  polynomials capture their primes.  While Dirichlet's theorem allows us to handle linear polynomials, 
polynomials such as 
$f(t)=t^2+1$  remain out of reach. 
 For the Liouville function
 $\lambda$, cancellation is predicted by the  Chowla conjecture precisely when the polynomial is not proportional to the square of a polynomial. This also seems out of reach for irreducible 
 polynomials of degree at least two.
 For the arithmetic function $r$ that counts representations of an integer as a sum of two squares,
 this relates to deep questions about the solubility of 
  the equation
 $
 u^2+v^2=f(t)
 $
 in integers.
 Current techniques prevent us from treating such equations when $\deg(f)\geq 3$.

 The purpose of this paper is 
to resolve all of these questions when averaging over  the coefficients of  polynomials of a given  degree.

\subsection{Bateman--Horn for random polynomials}

Let $\Lambda$ denote the von Mangoldt  function, which we extend to all integers by setting 
$\Lambda(0)=0$ and $\Lambda(-n)=\Lambda(n)$ for $n>0$.  We say that a polynomial $f\in \mathbb{Z}[t]$ is 
an \emph{admissible polynomial} 
if $f$ is irreducible over $\mathbb{Q}$ and it has no fixed prime divisor. The latter means there exists no prime $p$ such that all of the integers $f(n)$ are divisible by $p$, as $n$ ranges over the integers.   In 1854
Bunyakovsky conjectured that  any 
admissible  polynomial takes on infinitely many (positive or negative) prime values. The Bateman--Horn conjecture is a far-reaching generalisation of Bunyakovsky's conjecture. It predicts that, for any irreducible polynomial $f\in \mathbb{Z}[t]$ of degree $d\geq 1$, we have
\begin{align}\label{eqq2}
\frac{1}{x}\sum_{
n \leq x 
}\Lambda(f(n))=\mathfrak{S}_f+o(1),
\end{align}
as $x\to\infty$, 
where
$$
\mathfrak{S}_f=\prod_{p}\left(1-\frac{1}{p}\right)^{-1}
\left(1-
\frac{\#\{u\in \mathbb{F}_p: f(u)=0\}}{p}\right). 
$$
Note that $\mathfrak{S}_f>0$  precisely when $f$ is 
an admissible 
polynomial, and so this assertion 
contains Bunyakovsky's conjecture. 

For \emph{fixed} polynomials $f$ of degree $\geq 2$, these conjectures seem very difficult. Our first result  studies the asymptotic~\eqref{eqq2} for \emph{almost all} polynomials, which involves fixing some of the coefficients of the polynomial and letting the others vary in a box.  This leads to the following definition.

\begin{definition}\label{def:cube}
Let $d\geq 1$ and let $H\geq 1$. For any index set $\mathcal{E}\subseteq \{0,\dots, d\}$ and any integers $\alpha_e\in [-H,H]$ for $e\in \mathcal{E}$, we call the set
\begin{align*}
\mathscr{C}=\mathscr{C}(H)= \left\{(a_0,\ldots, a_{d})\in (\mathbb{Z}\cap [-H,H])^{d+1}: a_e=\alpha_e\,\, \forall\,\, e\in \mathcal{E}\right\}  
\end{align*}
a \emph{combinatorial cube} of side length $H$ and dimension $d+1-\#\mathcal{E}$. We call $a_d$ the {\em constant coefficient} of an element  $(a_0,\ldots, a_{d})\in \mathscr{C}$.
Given a combinatorial cube $\mathscr{C}\subset \mathbb{Z}^{d+1}$, we say that a polynomial $a_0t^{d}+a_1t^{d-1}+\cdots +a_d$ \emph{has its coefficients in $\mathscr{C}$} if $(a_0,\ldots, a_d)\in \mathscr{C}$.
\end{definition}

We are now ready to reveal our first result. Under the assumption that the dimension of the combinatorial cube $\mathscr{C}$ is at least two, this shows that almost all 
admissible polynomials with coefficients in  $\mathscr{C}$ produce many prime values.

\begin{theorem}[Bateman--Horn  with probability $1$]\label{thm_main}
Let $d,A\geq 1$ and $0<c<5/(31d)$ be fixed. There exists a constant $H_0(d,A,c)$ such that the following holds for any  $H\geq H_0(d,A,c)$. Let $\mathscr{C}\subset \mathbb{Z}^{d+1}$ be a combinatorial cube of side length $H$ and dimension at least $2$, such that the constant coefficient is not fixed to be $0$. Then, for all but at most $\#\mathscr{C}/(\log H)^{A}$ degree $d$ polynomials $f\in \mathbb{Z}[t]$ with coefficients in $\mathscr{C}$, we have
\begin{align}\label{eqq1}
\sup_{x\in [H^{c},2H^{c}]}\left|\frac{1}{x}\sum_{
n \leq x 
}\Lambda(f(n))-\mathfrak{S}_f(x)\right|\leq (\log H)^{-A},    
\end{align}    
where 
$$
\mathfrak{S}_{f}(x)=\prod_{p\leq \exp(\sqrt{\log x})}
\left(1-\frac{1}{p}\right)^{-1}
\left(1-
\frac{\#\{u\in \mathbb{F}_p: f(u)=0\}}{p}\right). 
$$
\end{theorem}

The hypothesis that the constant coefficient of  $\mathcal{C}$ is not fixed to be $0$ is clearly  necessary, since 
 the polynomials parameterised by $\mathcal{C}$ will  be reducible otherwise. We proceed by putting  Theorem~\ref{thm_main} into context. 
First and foremost, Theorem~\ref{thm_main} (and its 
generalisation to tuples of polynomials in
 Theorem~\ref{thm_main_multi}) strengthens  work of Skorobogatov and Sofos~\cite[Thm.~1.9]{SS}, with
 the main improvement being the increased allowable size of $x$.
 In our work we can take the length of the sum to be of polynomial size $x\asymp H^c$ in terms of the height, whereas  only sums of logarithmic length $x\asymp (\log H)^{A}$ are allowed in 
 \cite{SS}. Moreover, since the dimension of $\mathcal{C}$ is allowed to be $2$, we can average over only two coefficients, rather than the full set of coefficients needed in~\cite{SS}. 
 
Under the assumption of GRH, it would be possible to prove~\eqref{eqq1} for $c=1/d-\delta$, for any fixed $\delta>0$. Thus, apart from the value of $c$, the range $x\ll H^c$ we obtain is as good as what GRH would give. 
Moreover, the size of the exceptional set of polynomials cannot be further improved\footnote{Indeed, if the exceptional set in Theorem~\ref{thm_main} was of size $\ll\# \mathscr{C}/(\log H)^{\psi(H)}$ for some $\psi(H)\to \infty$, then~\eqref{eqq1} would be true for every polynomial $f(t)=qt+a$ with $1\leq a,q\leq (\log H)^{\psi(H)/2}$} without improving the Siegel--Walfisz theorem. 
    Finally,  we mention some further results in the literature which use averaging over only one coefficient to capture prime values of polynomials. For polynomials of the shape $f(t)=t^d+c$, for varying $c$, the best results available are due to Balestrieri and Rome~\cite{rome} and Zhou~\cite{zhou} (generalising earlier works of Baier and Zhao~\cite{baier-zhao}, Foo and Zhao~\cite{foo-zhao}, and Granville and Mollin~\cite[Prop.~1]{granville-mollin}).

As an immediate corollary to Theorem~\ref{thm_main}, we obtain a polynomial lower bound on the number of prime values that a typical polynomial takes.

\begin{corollary}
\label{cor:H}
Let $d\geq 1$ and $\varepsilon>0$ be fixed.  Let $\mathscr{C}\subset \mathbb{Z}^{d+1}$ be a combinatorial cube of side length $H\geq H_0(d,\varepsilon)$ and dimension at least $2$.  Then, for 100\% of degree $d$ 
admissible polynomials $f\in \mathbb{Z}[t]$  with coefficients in $\mathscr{C}$, we have
$$
\#\left\{n\in \mathbb{N}:\,\, f(n) \textnormal{ is prime}\right\}\geq H^{5/(31d)-\varepsilon}.    
$$
\end{corollary}

Theorem~\ref{thm_main} is a special case of the following result that applies to the simultaneous primality of a tuple of polynomials. 

\begin{theorem}\label{thm_main_multi}
Let $d,A,r\geq 1$  and $0<c<5/(31d)$ be fixed. There exists a constant $H_0(d,A,c,r)$ such that the following holds for any  $H\geq H_0(d,A,c)$. Let $\mathscr{C}\subset \mathbb{Z}^{d+1}$ be a combinatorial cube of side length $H$ and dimension at least $2$, such that the constant coefficient is not fixed to be $0$. Then, for all but at most $(\#\mathscr{C})^r/(\log H)^{A}$ $r$-tuples of polynomials $f_1,\ldots, f_r\in \mathbb{Z}[t]$ having degrees $\leq d$ and coefficients in $\mathscr{C}$, we have
$$
\sup_{x\in [H^{c},2H^{c}]}\left|\frac{1}{x}\sum_{n\leq x}\Lambda(f_1(n))\cdots \Lambda(f_r(n))-\mathfrak{S}_{f_1,\ldots, f_r}(x)\right|\leq (\log H)^{-A},    
$$
where 
\begin{align}\label{eq:SS-fx'}
\mathfrak{S}_{f_1,\ldots, f_r}(x)
\hspace{-0.1cm}
=
\hspace{-0.2cm}
\prod_{p\leq \exp(\sqrt{\log x})}
\hspace{-0.1cm}
\left(1-\frac{1}{p}\right)^{-r}
\left(1-
\frac{\#\{u\in \mathbb{F}_p: f_1(u)\cdots f_r(u)=0\}}{p}\right).
\end{align}
\end{theorem}

One  further direction concerns  
an almost-all version of the Bateman--Horn conjecture for polynomials evaluated at primes. 
Although we will not give details, one can adapt the proof of 
Theorem \ref{thm_main} to prime arguments by taking  $\alpha_n$ to be the indicator function of the primes in Theorem \ref{prop_general}. In this way, it is possible to prove that 
\begin{align*}
\sup_{x\in [H^c,2H^c]}\left|\sum_{p\leq x}\Lambda(f(p))-\mathfrak{S}_{t,f}(x)\int_{2}^x\frac{\textnormal{d}u}{\log u}\right|\ll x(\log H)^{-A},
\end{align*}
for all but $\ll \#\mathscr{C}/(\log H)^{A}$ polynomials $f\in \mathbb{Z}[t]$ of degree $d$ with  coefficients in a combinatorial cube $\mathcal{C}$ of dimension at least $2$.

%Applying this with $f_1(t)=t$ and $f_2(t)=f(t)$, and using partial summation to replace $\Lambda(n)$ with the indicator of the primes, we see that (under the same conditions as in Theorem~\ref{thm_main_multi}) for all but $\ll \#\mathscr{C}/(\log H)^{A}$ polynomials $f\in \mathbb{Z}[t]$ of degree $d$ with  coefficients in $\mathcal{C}$, we have
%\begin{align*}
%\sup_{x\in [H^c,2H^c]}\left|\sum_{p\leq x}\Lambda(f(p))-\mathfrak{S}_{t,f}(x)\int_{2}^x\frac{\textnormal{d}u}{\log u}\right|\ll x(\log H)^{-A}.
%\end{align*}
%This can be seen as an almost-all version of the Bateman--Horn conjecture for polynomials evaluated at primes. 

\subsection{Chowla's conjecture for random polynomials}

Let $\lambda$ denote the Liouville function, which we extend to all integers by setting $\lambda(0)=0$ and $\lambda(-n)=\lambda(n)$ for $n>0$. 
The polynomial Chowla conjecture~\cite{chowla} states that, for any $f\in \mathbb{Z}[t]$ that is not of the form $cg(t)^2$ with $c\in \mathbb{R}$ and $g\in \mathbb{R}[t]$, we have
\begin{align*}
\sum_{n\leq x}\lambda(f(n))=o(x).    
\end{align*}
Despite recent progress due to Tao and 
Ter{\"a}v{\"a}inen
~\cite{tao-chowla, tt-ant, tt-duke},
concerning  the case of polynomials that factor into linear factors, 
this conjecture  remains wide open for any $f$ that is irreducible and of degree at least $2$. 
Matom\"aki, Radziwiłł and Tao \cite{kaisa} have addressed 
Chowla's conjecture for almost all polynomials of the form
$(x + h_1) \cdots (x + h_d)$,
with $|h_1|,\dots,|h_d| \leq H$, whenever $H=H(x)\leq x$ goes to infinity as $x\to \infty$.
In the spirit of Theorem~\ref{thm_main}, we can prove the polynomial Chowla conjecture for almost all polynomials, improving on a qualitative result due to Ter{\"a}v{\"a}inen
~\cite[Thm.~2.11]{tera-polynomial}. 

\begin{theorem}[Polynomial Chowla conjecture with probability $1$]\label{thm_chowla}
Let $d\geq 1$, $A\geq 1$ and $0<c<5/(31d)$ be fixed. There exists a constant $H_0(d,A,c)$ such that the following holds for any $H\geq H_0(d,A,c)$. Let $\mathscr{C}\subset \mathbb{Z}^{d+1}$ be a combinatorial cube of side length $H$ and dimension at least $2$. Then, for all but at most $\#\mathscr{C}/(\log H)^{A}$ degree $d$ polynomials $f\in \mathbb{Z}[t]$ with coefficients in $\mathscr{C}$, we have
\begin{align*}
\sup_{x\in [H^{c},2H^{c}]}\frac{1}{x}\left|\sum_{n\leq x}\lambda(f(n))\right|\leq (\log H)^{-A}.    
\end{align*}    
\end{theorem}

\subsection{Norm forms at random polynomials}

Let $K/\QQ$ be a finite  extension of number fields of degree $e\geq 2$ 
and fix a $\ZZ$-basis $\{\omega_1,\dots,\omega_e\}$ for the ring of integers of $K$.
We will denote the corresponding {\em norm form} by
\begin{align}\label{eq:norms}
\nf_K(\x)=
N_{K/\QQ}(x_1\omega_1+\dots+x_e\omega_e),
\end{align}
where
$\x=(x_1,\dots,x_e)$ and 
 $N_{K/\QQ}$ is the field norm.
This is a homogeneous
polynomial of degree $e$ with integer coefficients.
Fix a positive integer  $d$. 
Our final result concerns the solubility of the equation
\begin{equation}\label{eq:n-f}
\nf_K(\x)=f(t)
\end{equation}
over $\ZZ$, for a typical   degree $d$ polynomial
$f\in \ZZ[t]$. These varieties are sometimes called \emph{generalised Ch\^atelet varieties}.
Over $\QQ$, their arithmetic has been the subject of intense investigation, encompassing work of 
Colliot-Th\'el\`ene, Sansuc and Swinnerton-Dyer \cite{crelle1,crelle2} when $e=2$ and $d\in \{3,4\}$, 
work of 
 Browning and Matthiesen \cite{nf} when $e\geq 2$ and $f$ is a product of linear polynomials over $\QQ$, and 
 conditional work of 
Colliot-Th\'el\`ene,  Skorobogatov and Swinnerton-Dyer \cite{lead} when 
 $e=2$ and $f$ is irreducible. The situation over $\ZZ$ continues to 
present a much more  significant challenge.

The equation~\eqref{eq:n-f} defines an affine variety $X\subset \AA^{e+1}$ which
is equipped with a dominant morphism $\pi:X\to \AA^1$. We denote by $\mathcal{X}$  its integral model over $\ZZ$, which is given by the same equation. Let $\mathcal{X}(\ZZ)$ denote the set of integer points on $\mathcal{X}$ and let  
$\mathcal{X}(\mathbf{A}_\ZZ)=X(\RR)\times \prod_p \mathcal{X}(\ZZ_p)$ denote the set of ad\`eles on $\mathcal{X}$.  Since $\mathcal{X}(\ZZ)$ embeds diagonally in 
$\mathcal{X}(\mathbf{A}_\ZZ)$, an obvious 
 necessary condition for the solubility of~\eqref{eq:n-f}  over $\ZZ$ is the condition
$\mathcal{X}(\mathbf{A}_\ZZ)\neq \emptyset$. In other words,
the   equation  should be soluble over $\RR$ and over $\ZZ_p$ for all primes $p$. 
If $\cX(\ZZ)=\emptyset$ but 
$\mathcal{X}(\mathbf{A}_\ZZ)\neq \emptyset$, then we have a {\em counter-example to the integral Hasse principle}. Colliot-Th\'el\`ene and Xu~\cite{xu} have established that the {\em Brauer--Manin obstruction} plays an important role in  the existence of integral points on $\mathcal{X}$, leading to an inclusion of sets
$
\mathcal{X}(\ZZ)\hookrightarrow \mathcal{X}(\mathbf{A}_\ZZ)^{\Br X}\subseteq \mathcal{X}(\mathbf{A}_\ZZ),
$
where $\Br X=\mathrm{H}_{\text{\'et}}^2(X,\mathbb{G}_m)$ is the Brauer group of $X$.
It remains a substantial  challenge to find classes of varieties for which the 
Brauer--Manin obstruction can be shown to be the only obstruction. 

In~\cite[Thm.~1.1]{DW},
Derenthal and Wei have studied the arithmetic of
\eqref{eq:n-f} when $e=4$ and $f(t)=c(t^2-a)$, for non-zero $c\in \ZZ$ and square-free $a\in \ZZ$. 
Under the assumption that $\sqrt{a}\in K$ and $\pi(X(\RR))$ is unbounded, it is shown that the Brauer--Manin obstruction controls the integral Hasse principle for $X$.
The equation~\eqref{eq:n-f} is most familiar in the case $e=2$ of quadratic extensions $K=\QQ(\sqrt{a})$, for square-free $a\in \ZZ$.
As explained in~\cite[\S 4]{harari}, 
  Harari~\cite[Thm.~4]{harari} uses class field theory to prove that the Brauer--Manin obstruction 
is the only obstruction to the existence of integral points
when $f$ is  the constant polynomial.
When $f(t)=bt^2+c$ is quadratic such that  $\nf_{K}(x_1,x_2)-bt^2$ is an indefinite ternary quadratic form, then it follows from 
the work of Colliot-Th\'el\`ene and Xu~\cite[Thm.~6.3]{xu} that the Brauer--Manin obstruction  is the only one. 
In~\cite[p.~175]{CTH}, Colliot-Th\'el\`ene and Harari have highlighted the difficulty of tackling the arithmetic of $X$ when $f$ is a separable polynomial of degree $d\geq 3$, noting that the quotient $\Br X_t/\Br \QQ$ is infinite for any smooth $\QQ$-fibre $X_t=\pi^{-1}(t)$ of the morphism $\pi:X\to \AA^1$.
In 1979, Colliot-Th\'el\`ene and Sansuc~\cite{sansuc} discovered 
that  Schinzel's Hypothesis~\cite{schinzel}, 
which is a generalisation of Bunyakovsky's conjecture to multiple polynomials, 
has a  key role to play in the study of local--global principles for rational points on 
conic bundle surfaces.
This was adapted to integral points by 
Gundlach~\cite[Thm.~2]{gundlach}, who assumed 
Schinzel's Hypothesis 
to prove that the 
Brauer--Manin obstruction is the only obstruction to the existence of integral points
on~\eqref{eq:n-f} when $K=\QQ(i)$ and  
$f(t)=-t^d+c$ for odd $d\geq 3$ and $c\in \ZZ$.
This work has been generalised to other quadratic fields 
$K=\QQ(\sqrt{a})$ by Mitankin~\cite{vlad}.  
Suppose that $f(t)=cf_1(t)\cdots f_n(t)$ is non-constant and separable, and that the discriminant of the splitting field of $f_1$ is coprime to the discriminant of $K$. Then, under 
Schinzel's Hypothesis  and 
certain further technical assumptions,   Mitankin  confirms 
that the Brauer--Manin obstruction controls the integral Hasse principle for $X$.
Given  Corollary~\ref{cor:H}, it is natural to suspect that our methods might have something to say about the integral Hasse principle on average
for equations of the form~\eqref{eq:n-f}. However, as discussed below, using prime values of polynomials is not enough to generate solutions to $\nf_K(\mathbf{x})=f(t)$ for almost all 
polynomials $f\in \mathbb{Z}[t]$, so we offer a different approach to the problem,  which measures  the density of integer solutions directly, at least on average.

 Problems can occur 
when $\pi(X(\RR))$ is bounded, which can be viewed  as an additional {\em obstruction at infinity}.
Consider the case $K=\QQ(i)$ and $f(t)=-t^4+22$,  for example, corresponding to the equation
$
x_1^2+x_2^2+t^4=22.
$
This equation has no solution over $\ZZ$, despite being soluble over $\QQ$ and over $\ZZ_p$ for every prime $p$.  
Berg~\cite{berg} has shown that there is no Brauer--Manin obstruction to the existence of integer points and so this counter-example to the integral Hasse principle comes purely from  an obstruction at infinity. 

For  $H\geq 1$, 
we shall  parameterise degree $d$ polynomials 
$f(t)=c_0t^d+\cdots+c_d\in \ZZ[t]$ by coefficient vectors belonging to the set
\begin{equation}\label{eq:stripe} 
S_d(H) =\left\{\c=(c_0,\dots,c_d)\in \ZZ^{d+1}: |\c|\leq H \right\},
\end{equation}
where $|\cdot|$ is the sup-norm. 
As $H\to \infty$, we would like  to  understand the solubility of 
$\nf_{K}(\mathbf{x})=f_{\mathbf{c}}(t)$ 
over $\ZZ$
 for a randomly chosen vector $\c\in S_d(H)$,
where we adopt the notation $f_{\mathbf{c}}(t)=c_0t^d+\cdots +c_d$. At the non-archimedean places we shall want our vector $\c$ to be such that $\nf_{K}(\mathbf{x})=f_{\mathbf{c}}(t)$
is  soluble over  $\ZZ_p$ for all primes $p$. At the infinite place we wish to capture the information that the 
 obstruction at infinity  is void, which we  achieve  by simply 
 demanding that the leading coefficient $c_0$ of $f_\c$ is positive. Henceforth,  
 a  vector $\c\in \ZZ^{d+1}$ is said to be {\em admissible} if $c_0>0$ and 
$\nf_{K}(\mathbf{x})=f_{\mathbf{c}}(t)$
is  soluble over $\ZZ_p$ for all primes $p$. 
Let 
\begin{equation}\label{eq:def-Sd-l}
S_d^{\text{loc}}(H)=\left\{\c\in S_d(H): \text{$\c$ is admissible}
\right\}.
\end{equation}
It is easily confirmed that 
$
\#S_d(H) =2^{d+1}H^{d+1} +O_d(H^d).
$
Moreover, 
it follows from ~\cite[Thm.~1.4]{wa} that 
\begin{equation}\label{eq:local-wa}
\#S_d^{\text{loc}}(H)=\left(\frac{1}{2}+o(1)\right)   \prod_p \sigma_p \cdot \#S_d(H)  ,
\end{equation}
as $H\to \infty$, 
where $\sigma_p$ is the probability that 
the equation $\nf_K(\x)=f_\c(t)$ is soluble over $\ZZ_p$, for a 
coefficient vector  $\c\in \ZZ_p^{d+1}$.
In particular, the equation 
\eqref{eq:n-f} is everywhere locally soluble for a positive proportion of coefficient vectors. 
We shall prove that, with probability 1, 
the equation 
\eqref{eq:n-f}
admits a solution  $(\x,t)\in \ZZ^{e+1}$, for any  degree $d$ 
polynomial $f\in \ZZ[t]$ with 
admissible coefficient vector.
 In fact we shall be able to do so with an exceptional set that saves an arbitrary power of $\log H$.

\begin{theorem}[Hasse principle for 100\% of generalised Ch\^atelet surfaces]\label{t:1}
Let $d\geq 1$ and $A>0$. Then, for any finite extension $K/\mathbb{Q}$, we have 
$$
\frac{\# \left\{ \c\in S_d^{\textnormal{loc}}(H):\,\, \nf_{K}(\mathbf{x})=f_{\mathbf{c}}(t) \textnormal{ soluble over } \ZZ \right\}}{\#S_d^{\textnormal{loc}}(H)}= 1+O_{d,K,A}\left(\frac{1}{(\log H)^A}\right),
$$
where the implied constant depends only on $d$, on  $K$ and on the choice of $A$.
\end{theorem}

It is worth  comparing this result with recent work of Skorobogatov and Sofos~\cite{SS}, 
who are also interested in the equation 
\eqref{eq:n-f} for random polynomials $f\in \ZZ[t]$, although
it should be emphasised that our results concern solubility over $\ZZ$, rather than over $\QQ$. Putting this key difference to one side, the scope of Theorem~\ref{t:1} 
improves on that of~\cite[Thm.~1.3]{SS} in two major aspects. Firstly, we can handle arbitrary number fields $K/\QQ$, rather than restricting to cyclic extensions.  Secondly, we are able to prove that $100\%$ of admissible polynomials 
yield integer solutions (with an error term that saves an arbitrary power of $\log H$), 
whereas only a positive proportion is achieved in~\cite{SS}. 
In the subsequent paragraph to \cite[Thm.~1.3]{SS},   Skorobogatov and Sofos
indicate how one might extend their work to handle totally imaginary abelian extensions $K/\QQ$ of class number $1$.  The first step
is an application of the 
Kronecker--Weber theorem, which implies that $K\subset \mathbb{Q}(\zeta_m)$, 
for some  $m\in \NN$.
But then the rational primes $p\equiv 1\bmod m$ split in $K/\mathbb{Q}$, and so, in order to find an integer solution to $\nf_{K}(\mathbf{x})=f(t)$ it suffices to find $t\in \NN$ such that $f(t)$ is a prime congruent to $1$ modulo $m$. It is clear that this  approach cannot be made to work 
for 100\% of polynomials, since $\nf_{K}(\mathbf{x})=f(t)$ may be soluble even if $f$ does not take prime values in a given residue class. In contrast to this, our proof of Theorem~\ref{t:1} is  based on estimates for the average number of representations of $f(n)$ as a value of the norm form, as described in  Theorem~\ref{t:1'}.

Combining Theorem~\ref{t:1}  with~\eqref{eq:local-wa} leads to the  following immediate consequence. 

\begin{corollary}
Let $d\geq 1$ and let $S_d^+(H)=\{\c\in S_d(H): c_0>0\}$. 
Then, for any  finite extension $K/\QQ$, we have
$$
\frac{1}{\#S_d^+(H) }\# \left\{ \c\in S_d^{+}(H):\,\, \nf_{K}(\mathbf{x})=f_{\mathbf{c}}(t) \textnormal{ soluble over } \ZZ\right\}
=(1+o(1)) \prod_p \sigma_p,
$$
as $H\to \infty$,
where $\sigma_p$ is the probability that 
the equation $\nf_K(\x)=f_\c(t)$ is soluble over $\ZZ_p$, for a 
coefficient vector  $\c\in \ZZ_p^{d+1}$.
\end{corollary}

This directly improves on 
~\cite[Cor.~7.2]{SS}, which only deals with the special case $K=\QQ(i)$ and proves that the left-hand side exceeds 
$\frac{56}{100}$,
for sufficiently large $H$.

In Theorem~\ref{t:1'}, we shall prove a stronger version of Theorem~\ref{t:1}, in which we can almost always produce at least $H^\Delta$ integer solutions to the equation 
$\nf_{K}(\mathbf{x})=f_{\mathbf{c}}(t)$, for an appropriate constant $\Delta>0$ depending only on $d$ and $e$. 
This quantification will permit us to 
establish a weak form of Zariski density, for $100\%$ of admissible polynomials.

\begin{corollary}\label{cor:zariski}
Let $K/\QQ$ be a number field of degree $e\geq 2$
and  let $D,L\geq 1$. Then, 
as admissible   degree $d$ coefficient vectors are ordered by height, 
with probability 1, the integer solutions to the equation 
$\nf_K(\x)=f(t)$  are not contained in a union of  $\leq L$ irreducible curves of degree $k\in [d(2e+7),D]$.
\end{corollary}

This result is most interesting in the case $e=2$, since then the only possible non-trivial subvarieties have dimension $1$. 

\begin{remark}
In many cases it is possible to remove the restriction on the sign of the leading coefficient in the set 
$S_d^{\text{loc}}(H)$  of admissible vectors. 
Let  $S_d^\pm(H) $ be  the set of  vectors $\c\in S_d(H)$ for which $\pm c_0>0$ and let 
$S_d^{\text{loc},\pm}(H)$  be the 
set of $\c\in S_d^\pm(H)$ such that 
\begin{equation}\label{eq:n-f-c}
\nf_{K}(\mathbf{x})=f_{\mathbf{c}}(t)
\end{equation}
is  soluble over $\ZZ_p$ for all $p$.
When $d$ is odd, the vector $\c$ leads to solubility over $\ZZ$ in ~\eqref{eq:n-f-c} 
if and only if the vector $(-c_0,c_1,-c_2,\dots,-c_{d-1},c_d)$ does, whence
\begin{align*}
&\# \{ \c\in S_d^{\text{loc},+}(H):  \text{\eqref{eq:n-f-c} soluble over $\ZZ$}\}\\
&\qquad\qquad=
\# \{ \c\in S_d^{\text{loc},-}(H): \text{\eqref{eq:n-f-c}  soluble over $\ZZ$}\}.
\end{align*}
Thus, when $d$ is odd,
one trivially obtains an 
analogue of Theorem~\ref{t:1} in which there is  no restriction on the sign of the leading coefficient. 
When $d$ is even, we may have an obstruction at infinity if $c_0<0$ and the norm form  $\nf_K$ is positive definite. Thus, when $d$ is even,  we should only consider 
$S_d^{\text{loc},-}(H)$ when $\nf_K$ is not positive definite. In this case, on multiplying through by $-1$, we can equate 
$\# \{ \c\in S_d^{\text{loc},-}(H): \text{\eqref{eq:n-f-c} is soluble over $\ZZ$}\}$
with the number of $\c\in S_d^{\text{loc},+}(H)$ for which the equation 
$-\nf_K(\x)=f_\c(t)$ is soluble over $\ZZ$. Although we omit details, it is possible to adjust the  argument behind the proof of Theorem~\ref{t:1} to apply in this setting. (Note that this is trivial when the unit group $U_K$ contains an element of norm $-1$, since then 
$-\nf_K(\x)=f(t)$ is soluble over $\ZZ$ if and only if 
$\nf_K(\x)=f(t)$ is soluble over $\ZZ$.)
\end{remark}

\subsection{Structure of the paper}

In Section~\ref{sec:equidistribution}, we introduce a key tool that allows us to reduce the problem of averages along random polynomials to the problem of equidistribution in short intervals and arithmetic progressions. In Section~\ref{sec:pntaps}, we show that the Liouville and von Mangoldt functions are  equidistributed in the desired manner. In Section \ref{sec:chowla} 
we deduce Theorem~\ref{thm_chowla}  from this, and in 
Section \ref{sec:bateman-horn} we deduce Theorem~\ref{thm_main_multi}.

The rest of the paper is devoted to proving Theorem~\ref{t:1} and its corollaries (and is independent from Sections~\ref{sec:pntaps} to~\ref{sec:bateman-horn}). Section~\ref{sec:polynomial} contains some auxiliary results about norm forms modulo prime powers. In Section~\ref{sec:normforms}, we construct a local model for the representation function of a norm form and evaluate its averages over short intervals and arithmetic progressions; this is achieved in Proposition~\ref{prop:1moment}. In Section~\ref{sec:localised}, and Proposition \ref{p:STEP1} specifically, we prove that the asymptotics of the representation function of the norm form match those of its localised counting function for almost all polynomials. Finally, Section~\ref{sec:rarely-small} deals with the remaining task of showing that the localised counting function is almost always large.

\subsection*{Notation}
We will always use 
$I$ to denote an interval. If $I=[a,b]$, then for any $c_1,c_2\in \mathbb{R}$ we may define the affinely translated interval $c_1I+c_2=[c_1a+c_2,c_1b+c_2]$.
We denote the $k$-fold divisor function by $\tau_k$, and abbreviate $\tau_2$ as $\tau$. We will often use the submultiplicativity property $\tau_k(mn)\leq \tau_k(m)\tau_k(n)$ for all $m,n\in \mathbb{N}$. We denote by $\varphi$ the Euler 
totient 
function, by $\mu$ the M\"obius function, by $\lambda$ the Liouville function, and by $\Lambda$ the von Mangoldt function. We sometimes use the abbreviation $\varphi^{*}(n)=n/\varphi(n)$.

All implied constants in the Vinogradov notation $O(\cdot),o(\cdot),\ll$ are  only allowed to depend on fixed parameters, unless otherwise indicated by  a subscript. Thus, for instance, $o_{n\to \infty}(1)$ denotes a quantity tending to $0$ as $n\to \infty$.

We always regard $d$ and $A$ fixed, with the latter being large; thus, our implied constants may depend on these parameters. 

\begin{ack} 
The authors are extremely grateful to the anonymous referee 
and to Yijie Diao for useful comments. 
TB was supported
by a FWF grant (DOI \texttt{10.55776/P32428}),
ES was supported by EPSRC
 New Horizons  
grant \texttt{EP/V048236/1}, 
 and JT was supported by Academy of Finland grant no.\ 340098, a von Neumann Fellowship (NSF grant \texttt{DMS-1926686}), and funding from European Union's Horizon
Europe research and innovation programme under Marie Sk\l{}odowska-Curie grant agreement No
101058904.
 This material is based upon work supported by a grant from the Institute
for Advanced Study School of Mathematics. 
\end{ack}

\section{Reduction to equidistribution in progressions}\label{sec:equidistribution}

This section contains the key new technical tool in this paper. It shows that 
controlling averages of a general arithmetic function over values of  random polynomials can be done successfully 
if  the function is well distributed in arithmetic progressions. 
We begin by giving a simpler version of the main technical result, which better conveys  the structure. 

\begin{corollary} \label{cor:simpler} 
Let  $d\in \mathbb N$ be fixed, let $x,H\geq 2$ with  
%$x\leq H^{1/(100d)}$, 
$\exp((\log H)^{1/100})\leq x\leq H^{1/(100d)}$, 
and let $F:\Z\to\mathbb{C}$ be a $1$-bounded sequence satisfying 
  \begin{align}\label{eq:F(n)}
  \max_{ 1\leq u\leq q\leq x^d} 
\sup_{\substack{
I \subset [-2H^{1.01}, 2H^{1.01}] \textnormal{ interval }   
\\
|I|> H^{0.99}
}}
\frac{q}{|I|}
\left|\sum_{\substack{n\in I\\n\equiv u\bmod q}}F(n)\right|\ll (\log H)^{-10000d}.
\end{align}
Then,  one has 
\begin{align}\label{eq:F(n)2}
 \left|\sum_{n \leq x} F(f(n ) )\right|   \ll  x(\log x)^{-1000}
\end{align}
for almost all $f\in \Z[t]$ of degree $d$ having  coefficients  bounded by $H$ in   modulus.
\end{corollary}

An assumption of the form~\eqref{eq:F(n)} is very natural, since if $d=1$, the conclusion~\eqref{eq:F(n)2} implies that $F$ is equidistributed in almost all arithmetic progressions $u\bmod q$ with $q\leq x^d$. 
It turns out that the simple hypothesis~\eqref{eq:F(n)} is too demanding to be useful in proving Theorems~\ref{thm_main},~\ref{thm_chowla} and~\ref{t:1}, which will require 
a stronger (but more technical) version, given in Theorem~\ref{prop_general}. On the other hand, 
Corollary~\ref{cor:simpler} is a relatively simple application of the second moment method; see footnotes~\ref{foot1},~\ref{foot2},~\ref{foot3} for various simplifications compared to the proof of  Theorem~\ref{prop_general}.  

In order to prove Theorems~\ref{thm_chowla} and~\ref{thm_main}, we wish to obtain~\eqref{eq:F(n)2} (with $1000$ replaced by an arbitrary constant) for the functions $F(n)=\lambda(n)$ and $F(n)=(\Lambda(n)-\widetilde{\Lambda}(n))/(\log x)$, where $\widetilde{\Lambda}$ is a simpler model function for the von Mangoldt function. (The function 
$\widetilde \Lambda$ shares a similar  distribution to $\Lambda$ on arithmetic progressions, so that $F$ has mean $o(1)$.) 
If $x=(\log H)^{A}$ for some fixed $A>0$, then the hypothesis~\eqref{eq:F(n)} 
follows directly from the prime number theorem in short intervals and arithmetic progressions; in particular, this would already lead to a simpler\footnote{Andrew Granville recently informed us that he has developed a similar argument in an unpublished note.} proof of~\cite[Thm.~1.9]{SS} and~\cite[Thm.~2.12]{tera-polynomial}. However,  since we have $x=H^{c/d}$ for some $c>0$ in Theorems~\ref{thm_chowla} and~\ref{thm_main}, we cannot hope to prove~\eqref{eq:F(n)}; the hypothesis would actually be false if $L(s,\chi)$ had a zero too close to the $1$-line for some character $\chi\bmod q$. We therefore need a refinement of Corollary~\ref{cor:simpler} that only requires equidistribution   for  moduli  $q$ that lie outside a sparse set of integers. Moreover, the values of the residue class $u\bmod q$ also need to be constrained somewhat in~\eqref{eq:F(n)}, since if $u$ and $q$ share a  prime divisor larger than $\exp(\sqrt{\log x})$, then the mean of $\Lambda(n)-\widetilde{\Lambda}(n)$ is not small in the progression $n\equiv u\bmod q$. To approach Theorem~\ref{t:1}, we need to allow $F$ to be divisor-bounded, since we will take $F$ to be a suitably normalised version of the representation function of a norm form. Finally, for Theorems~\ref{thm_main_multi} and~\ref{thm_chowla} we need to be able to average over $f\in \mathbb{Z}[t]$ in any combinatorial cube. 

Taken together,
these requirements lead us to the following formulation of our key technical result.
For the statement of the result, we shall denote by $\tau_B(n)$ the number of distinct representations of $n$ as a product of $B$ positive integers, with $B,n \in \mathbb{N}$.

 \begin{theorem}
[Equidistribution controls  sums for  almost all polynomials]\label{prop_general}  Let $A\geq 1$, $\varepsilon>0$  and $0\leq k< \ell \leq d$ be fixed. Let $H\geq H_0(A,d)$ and 
assume that $x$ lies in the interval
$$
\exp\left((\log H)^{1/100}\right)\leq  x\leq H^{1/(2d)}.
$$ 
Let $F:\mathbb{Z}\to \mathbb{C}$ be a sequence such that 
\begin{enumerate}

\item $|F(n)|\leq \tau_{B}(n)$ 
for all $n \in \mathbb Z$ and for
some $B\geq 1$ satisfying 
$400dB^{4d}\leq A$;   
\item There exists a set $\mathcal{Q}\subset [1,x^{d}]\cap \mathbb{Z}$ with $\sum_{q\in \mathcal{Q}}
q^{-1/(8d)}\ll (\log x)^{-3A}$,
 such that for any $q\leq x^{d}$ that is not a multiple of any element of $\mathcal{Q}$ we have
\begin{align*}
\max_{\substack{1\leq u\leq q\\ \gcd (u,q)\leq  
 \exp\left(\frac{ \sqrt{\log x}  }{    \log \log x }\right  )}}  
\sup_{\substack{
I  \textnormal{ interval } 
\\
|I|> H^{1-\varepsilon}
\\
I\subset [-2H x^d, 2H x^d]
}}
\frac{q}{|I|}
\left|\sum_{\substack{n\in I\\n\equiv u\bmod q}}F(n)\right|\ll (\log H)^{-10Ad}.
\end{align*}
\end{enumerate} 
Then, for any polynomial $g\in \mathbb{Z}[t]$ of degree $\leq d$ with coefficients in $[-H,H]$ such that   $g(0)\neq 0$, and for any coefficients  $\alpha_n\in  \CC$ such that $|\alpha_n|\leq 1$, 
we have 
\begin{align}\label{equ1}
\sup_{x'\in [x/2,x]}\sum_{|a|,|b|\leq H}\left|\sum_{1\leq n\leq x'}\alpha_nF(an^k+bn^{\ell}+g(n))\right|^2\ll H^2x^2(\log x)^{-A}.
\end{align}
\end{theorem}

We remark that we have not tried to optimise the dependencies between $A,B$ and $d$ in Theorem~\ref{prop_general}, since in applications $A$ is always taken to be an arbitrarily large constant.

By expanding out the square in~\eqref{equ1}, the proof of 
Theorem 
\ref{prop_general} boils down to obtaining cancellation in 
\begin{align*}
\sum_{n_1,n_2\leq x'}\alpha_{n_1}\overline\alpha_{n_2}\sum_{|a|,|b|\leq H}F(an_1^k+bn_1^{\ell}+g(n_1))\overline{F}(an_2^k+bn_2^{\ell}+g(n_2)).\end{align*}
This will eventually  lead us to study the size of $F$ along arithmetic progressions to large moduli of the form $n_1^k(n_1^{\ell-k} - n_2^{\ell-k})$ for ``typical'' $n_1,n_2\in \NN$.
Before embarking on this endeavour, we will need\footnote{\label{foot1}For proving the simpler Corollary~\ref{cor:simpler}, we would not need the results in Section~\ref{sub:preparatory}.} several elementary results about divisors that we proceed to collect together.

\subsection{Preparatory results on the anatomy of integers}\label{sub:preparatory}

\begin{lemma}\label{le_polydiv001} 
 Fix any  $ b,c\in \N$ and let $q_1,q_2\in \N$. Then for all $x\geq 1 $ we have 
$$ 
\# \left\{ \mathbf n \in (\N\cap[1,x])^2:
\begin{array}{l}
\gcd(n_1,n_2)=1 \\ 
q_1 \mid   n_1^b-  n_2^b, q_2 \mid n_1^c
\end{array}
\right\} \ll
\frac{x^2}{(q_1q_2)^{1/(2\max\{b,c\}})}. $$ 
 \end{lemma}

\begin{proof} 
If there is a prime dividing both $q_1$ and $n_1$ then, by $q_1 \mid   n_1^b-  n_2^b$, this prime would also divide $n_2$, which would contradict the coprimality of $n_1$ and $n_2$.
Thus we can proceed under the assumption that $n_1n_2$ is coprime to $q_1$.
We deduce that  $(n_2/n_1)^b\equiv 1\bmod{q_1}$. 
Thus, the cardinality is at most   
\begin{align}\label{eq:nsum}
\mathbf 1_{q_1\leq x^b} 
\sum_{\substack{\lambda \bmod{q_1} \\\lambda ^b\equiv 1\bmod{q_1 }}}
\sum_{\substack{n_1  \leq x \\  q_2 \mid n_1^c }}
\sum_{\substack{n_2 \leq x \\ n_2 \equiv \lambda n_1\bmod{q_1}}}1
 & \ll  \left( \frac{x}{q_1}+ \mathbf 1_{q_1\leq x^b}\right)
\sum_{\substack{\lambda \bmod{q_1} \\\lambda ^b\equiv 1\bmod{q_1 }}}
\sum_{\substack{ n_1  \leq x \\  q_2 \mid n_1^c }}1\nonumber
\\ & \ll  (2b)^{\omega(q_1) }
\left( \frac{x}{q_1}+ \mathbf 1_{q_1\leq x^b}\right)
 \sum_{\substack{n_1  \leq x \\  q_2 \mid n_1^c }}1
,\end{align} 
since the  congruence $\lambda^b\equiv 1\bmod{p^{\alpha}}$ has at most $2b$ solutions for any prime power $p^{\alpha}$.
Indeed,  if $p$ is odd, there are at  most $b$ solutions by the existence of primitive roots modulo $p^{\alpha}$, 
and if $p=2$, then the claim follows from the fact that $(\mathbb{Z}/2^{\alpha}
\mathbb{Z})^{\times}$ is generated by $-1$ and $5$. 
The standard  estimate  $\omega(q)=o( \log q ) $ gives $(2b)^{\omega(q_1)} \ll q_1^{0.01}$. If 
$q_1 \leq x^b$, then  it gives   $(2b)^{\omega(q_1)} \ll  x^{0.009} $. Hence, 
\begin{align*}
(2b)^{\omega(q_1) }
\left( \frac{x}{q_1}+ \mathbf 1_{q_1\leq x^b}\right) \ll \frac{x}{q_1^{0.99}}+x^{0.009} \left(\frac{x^b}{q_1}\right)^{0.99/b}
&\leq  \frac{x}{q_1^{0.99}}+\frac{x^{0.999}}{q_1^{0.99/b}}\\ &\leq  \frac{2x}{q_1^{1/(2b)}}
.
\end{align*}
It remains to bound the sum over $n_1$ in~\eqref{eq:nsum}. To do so 
we let $q_2'=\prod_{p \mid q_2} p^{ \alpha_p }$, where $\alpha_p=\lceil v_p(q_2)/c\rceil$. 
 This means that 
$q_2\mid n_1^c$ if and only if $ q_2' \mid n_1$. Thus, the sum over $n_1 $  is at most  $x/q_2'$.   By definition we have  $q_2' \geq q_2^{1/c}$, 
which exceeds $q_2^{1/(2\max\{b,c\}) }$.  \end{proof}

\begin{lemma}\label{l:polydplnit} 
 Fix any  $ b,c\in \N$. Then for all $x\geq 1 $  and $q\in \N$ 
we have 
\[
\# \left\{ \mathbf n\in (\N\cap[1,x])^2: q\mid (n_1^b-n_2^b)n_1^c   \right\}\ll 
x^2 q^{-\frac{1}{8\max\{b,c\}}}.
\]
\end{lemma}
\begin{proof} We classify all $(n_1,n_2)$ according to the value of $r=\gcd(n_1,n_2)$. Letting $m_1=n_1/r, m_2=n_2/r$ 
we see that the divisibility condition is equivalent to $q(r)= q/\gcd(q,r^{b+c})$ dividing $(m_1^b-m_2^b)m_1^c $. 
Let $\lambda=2\max\{b,c\}$.
The coprimality of $m_1,m_2$
implies that $m_1^b-m_2^b$ and $m_1^c $ are coprime.  Hence, the count is 
\begin{align*}
&\leq \sum_{r\leq x  }\sum_{\substack{ (q_1,q_2) \in \N^2 \\ q_1 q_2 =q(r) \\ \gcd(q_1,q_2)=1  }} 
\# \left\{  \mathbf m \in (\N\cap[1,x/r])^2
:
\begin{array}{l}
\gcd(m_1,m_2)=1\\   q_1 \mid   m_1^b-  m_2^b, q_2 \mid m_1^c
\end{array}
\right\}.
\end{align*}
It follows from  Lemma~\ref{le_polydiv001}  that the summand is 
$O((x/r)^2 (q_1q_2)^{-1/\lambda })$.
Thus the count is 
\begin{align*}
&\ll 
\sum_{r\in \N } \frac{x^2\tau(q(r))}{r^2q(r)^{1/\lambda}}
\ll  x^2 \sum_{r\in \N } r^{-2}q(r)^{-1/(2\lambda)},
\end{align*}
by the divisor bound.

Now define the arithmetic function $f_\lambda$ through $n^{1/(2\lambda)}= \sum_{k\mid n } f_\lambda (k ).$
In particular,  $$ f_\lambda(n) = \sum_{k\mid n } \mu(k) (n/k)^{1/(2\lambda)}=n^{1/(2\lambda)} \prod_{p\mid n } (1-p^{-1/(2\lambda)})
\in \left[0,   n^{1/(2\lambda)} \right].$$ This gives 
\begin{align*} \sum_{r\in \N } r^{-2}q(r)^{-1/(2\lambda)}
&= q^{-1/(2\lambda)}\sum_{r\in \N } r^{-2}\gcd(q,r^{b+c})^{1/(2\lambda)}\\ &=
q^{-1/(2\lambda)}
\sum_{r\in \N } r^{-2}\sum_{\substack{k\mid q\\ k\mid  r^{b+c}} } f_\lambda (k ).
\end{align*}
 Define $k'=\prod_{p\mid k}p^{\beta_p}$, where 
$\beta_p=\lceil v_p(k)/(b+c)\rceil$. In particular, $k\mid  r^{b+c}$ if and only if 
$k'\mid r$. Thus, the sum over $r$ turns    into  $$  
 \sum_{k\mid q } f_\lambda (k )\sum_{\substack{r\in \N\\ k'\mid  r}    } r^{-2}
\ll 
 \sum_{k\mid q } \frac{f_\lambda (k )}{k'^2}\leq  \sum_{k\mid q }  k^{\frac{1}{2\lambda}   -\frac{2}{ b+c }  }
  ,$$ owing to the bound $k' \geq k^{1/(b+c)}$. We clearly have $\frac{1}{2\lambda}   \leq \frac{2}{ b+c } $, whence
$$ \sum_{r\in \N} r^{-2} q(r)^{-1/(2\lambda) } \ll q^{-1/(2\lambda)} \tau(q) \ll q^{-1/(4\lambda)} ,$$
which is sufficient.
\end{proof}

We will also need the following simple lemma on sums of reciprocals of large divisors of a given integer.

\begin{lemma}\label{le_divisors}  
For all $q\in \N$ and $z\geq 1 $ we have 
$$
\sum_{\substack{d\mid q\\d\geq z}}\frac{1}{d}\ll 
z^{-1/2}\exp\left(\frac{4  (   \log (3q))^{1/2}  }{ (  \log \log (3q))^{3/2} } \right)
.$$ 
\end{lemma}

\begin{proof}  
Using Rankin's trick, we have 
\[
\sum_{\substack{d\mid q\\d\geq z}}\frac{1}{d}\leq 
\frac{1}{\sqrt{z}}
 \prod_{p\mid q}\left(1+\frac{1}{p^{1/2}}+\frac{1}{p}+\frac{1}{p^{3/2}}+\cdots\right)
= \frac{1}{\sqrt{z}} \prod_{p\mid q}\left(1-\frac{1}{p^{1/2}}  \right)^{-1} .\]
The inequality  $ (1-t)^{-1}   \leq 1+t+4t^2$ holds for   $0< t \leq 1/\sqrt 2$. 
Moreover, we have $1+t \leq \mathrm e^t $ for any $t\geq 0$. Hence
\begin{align*}
\sum_{\substack{d\mid q\\d\geq z}}\frac{1}{d}&\leq 
\frac{1}{\sqrt{z}} \prod_{p\mid q}\left(1+\frac{1}{p^{1/2}}+\frac{4}{p} \right)\\
  &\leq \frac{1}{\sqrt{z}}\exp\left(\sum_{p\mid q}(p^{-1/2}+4p^{-1})\right)
\\  &\leq \frac{1}{\sqrt{z}}\exp\left(\sum_{p\leq \omega( q) }(p^{-1/2}+4p^{-1})\right).
\end{align*}  
By partial summation and the prime number theorem, there exists an absolute constant  $C>2$ such that 
$$
\sum_{p\leq u}(p^{-1/2}+4p^{-1}) \leq   \frac{3 u^{1/2}}{\log u},
$$ 
for all $u> C$. 
Moreover, it follows from~\cite[Thm.~2.10]{montgomery-vaughan} that
$\omega(q)\leq 1.1(\log q)/(\log \log q)$ whenever $\omega(q)\geq C'$, provided $C'$ is a large enough constant. This concludes the proof when $\omega(q)
>\max\{C,C'\}$. In the opposite case, we have 
$$
\sum_{\substack{d\mid q\\d\geq z}}\frac{1}{d} 
\leq \frac{1}{\sqrt{z}}\exp\left(\sum_{p\leq \omega( q) }(p^{-1/2}+4p^{-1})\right)
=O_{C,C'} \left(z^{-1/2}\right),
$$
which is sufficient.
\end{proof}

\begin{lemma}\label{lem:corlry} For all  $ a, c \in \Z \setminus \{0\}$,  $b\in \Z$ and  $x_1 ,  x _2\geq   1 $ we have 
\begin{align*} \#\left\{n \in \Z\cap[-x_1,x_1] : \gcd(an +b,c)>x_2\right\}&\\
&\hspace{-5cm} \ll 
\frac{x_1\gcd(a,c)}{x_2^{1/2}}
 \exp\left(\frac{4  (   \log (3|c|))^{1/2}  }{ (  \log \log (3|c| ))^{3/2} } \right)
+ \tau(c).
\end{align*}
\end{lemma}

\begin{proof}Letting $t=\gcd(an +b,c)$ we can bound the cardinality in the lemma by 
\begin{align*}
& \sum_{\substack{ x_2<t \leq |a|x_1+|b|  \\ t \mid c} } \#\{n \in \Z\cap[-x_1,x_1] : t \mid an +b \}
\\
&\quad=
\hspace{-0.3cm}
\sum_{\substack{ x_2<t \leq |a|x_1+|b| \\ t \mid c     } }
 \#\left\{n \in \Z\cap[-x_1,x_1] :
\frac{t}{ \gcd(a,t)} \mid \frac{a}{ \gcd(a,t)}n +
\frac{b}{ \gcd(a,t)}  \right\}.
\end{align*} 
By the coprimality of $t/ \gcd(a,t) $ and $a/ \gcd(a,t)$, we can bound this by  
$$
\sum_{\substack{x_2<t \leq |a|x_1+|b|\\t \mid c} }
\left(\frac{2x_1\gcd(a,t)}{t}+1\right)\leq 
2x_1 \gcd(a,c ) 
\sum_{\substack{ t \mid c  \\ t>x_2   } } \frac{1}{t}
+
\tau(c).
$$ 
The proof follows on invoking Lemma~\ref{le_divisors}.
\end{proof} 

In the proof of 
Theorem 
\ref{prop_general}, when we open the
square in~\eqref{equ1} to obtain a double sum over $n_1,n_2\leq x'$,  we shall find it convenient to  restrict to pairs $(n_1,n_2)$ which have certain properties that almost all pairs $(n_1,n_2)$ satisfy. 
To this end, we will need the following lemma.

\begin{lemma}\label{le_typical} Let $A>0 $, $d\in \N$ be fixed and let $x> 1 $. 
  Let $g\in \mathbb{Z}[t]$ be a polynomial of degree $\leq d$ all of whose coefficients lie in $[-x^A,x^{A}]$ and such that $g(0)\neq 0$.  Define $\mathcal{M}_A$ to be the set
\begin{equation}\label{eq:define_M}
\left\{\mathbf n \in (\N\cap [1,x])^2: 
\begin{array}{l}
\text{$n_1,n_2>x/(\log x)^{A}$,  $|n_1-n_2|>x/(\log x)^{A}$}\\
\gcd(n_1,n_2)<(\log x)^{A}\\
\gcd(g(n_1),n_1^d)\leq \exp((\log x)^{1/2}/ (2\log \log x ))
\end{array}
\right\}.
\end{equation}
Then $\#((\N\cap[1,x])^2\setminus \mathcal{M}_A)\ll x^2/(\log x)^{A}$.
\end{lemma}

\begin{proof} 

We will show that the set of $(n_1,n_2)$ failing an individual property in~\eqref{eq:define_M} has size $\ll x^2/(\log x)^{A}$. The desired upper bound will then follow on taking the maximum of all these bounds. 

Dealing with the first property  is trivial, since  there are $\ll x/(\log x)^{A}$ pairs   $(n_1,n_2)$ with $|n_1-n_2|\leq x/(\log x)^{A}$ and $n_1 $ fixed.
Furthermore,  there are $\ll x^2/(\log x)^{A}$ pairs $(n_1,n_2)$ with $\min\{n_1, n_2\}\leq x/(\log x)^{A}$.
The number of pairs failing the second property is  at most 
$$
\sum_{(\log x)^{A} < m\leq x }\sum_{\substack{n_1,n_2\leq x \\ \gcd(n_1,n_2)=m}}1\leq \sum_{m> (\log x)^{A}}\frac{x^2}{m^2}\ll x^2/(\log x)^{A}.    
$$
To handle the third property, we let  $M=\exp((\log x)^{1/2}/(2\log \log x ))$. 
Note that if $\gcd(n_1^d,g(n_1))>M$, then $\gcd(n_1^d,g(n_1)^d)> M$, so $\gcd(n_1,g(n_1))> M^{1/d}$. Since $\gcd(n_1,g(n_1))=\gcd(n_1,g(0))$, 
the number of pairs failing the third property is  at most 
$$
\sum_{\substack{m\mid g(0)\\  M^{1/d}<m \leq x }}\sum_{\substack{n_1,n_2\leq x\\m\mid n_1}}1
 \ll x^2 \sum_{\substack{m\mid g(0)\\m\geq M^{1/d}}}\frac{1}{m} 
 \ll \frac{x^2}{M^{1/(2d)}}   \exp\left(\frac{4  (  \log(3|g(0)|))^{1/2}  }{ (  \log \log (3|g(0)| ))^{3/2} } \right) ,
 $$
by Lemma~\ref{le_divisors}. The proof  is  concluded by using the assumption    $ |g(0)| \leq  x^A$.
\end{proof}

The  remaining results in this section concern upper bounds for moments of the generalised divisor function.
We begin by recalling
 from~\cite[Eq~(1.80)]{iw-kow}
the 
upper bound 
\begin{equation}\label{eq:standard}
\sum_{r\leq R}\tau_{B}(r)^{k}\ll R(\log R)^{B^{k}-1},
\end{equation}
for any $B,k\in \NN$.
The following result gives a generalisation involving  general linear  polynomials.

\begin{lemma}\label{le_shiu} 
 Let $A\geq 1$ and $B,k\in \NN$ be fixed. 
Let $x\geq 3$, $y\leq x^{A}$ and $1\leq a,r\leq x^{A}$. Then we have
$$
\sum_{y\leq n\leq x+y}\tau_{B}(rn+a)^k\ll \tau_{B}(\gcd(a,r))^kx(\log x)^{B^k-1}\log \log x.
$$
\end{lemma}

\begin{proof}
Let $r'=r/\gcd(a,r)$ and  $a'=a/\gcd(a,r)$. Note that we have  $\tau_{B}(mn)\leq \tau_B(m)\tau_B(n)$. We conclude that 
\begin{align*}
\sum_{y\leq n\leq x+y}\tau_{B}(rn+a)^k&\leq \tau_B(\gcd(a,r))^k \sum_{y\leq n\leq x+y}\tau_{B}(r'n+a')^k \\
&=\tau_B(\gcd(a,r))^k \sum_{\substack{r'y+a'\leq n'\leq r'(x+y)+a'\\n'\equiv a'\bmod{r'}}}\tau_{B}(n')^k.
\end{align*}
By Shiu's bound~\cite{shiu}, this is
$O(\tau_B(\gcd(a,r))^kx (\log x)^{B^k-1}\log \log x)$,  on 
using the simple bound $r'/\varphi(r')\ll \log \log(3r')$.
\end{proof}

Next, the following result is a direct consequence of~\cite[Thm.~1]{BB} and requires little comment.

\begin{lemma}\label{lem:NT}
Let $r\geq 1$ and $B,k\in \NN$ be fixed. Let $x\geq 3$. Then we have 
\begin{align*}
\sum_{\substack{n_1,n_2\leq x\\n_1\neq n_2}}\tau_B(|n_1^r-n_2^r|)^k\ll x^2(\log x)^{(B^{k}-1)r}.    
\end{align*}
\end{lemma}

\subsection{Proof of Theorem ~\ref{prop_general}}

We  have now gathered  the technical apparatus with which to prove 
Theorem ~\ref{prop_general}.
Let $x'\in [x/2,x]$ and $g\in \mathbb{Z}[t]$ be chosen so that the left-hand side of~\eqref{equ1} is maximised. Opening up the square turns it into 
 \begin{align*}
\sum_{1\leq n_1,n_2\leq x'} \alpha_{n_1}\overline\alpha_{n_2}\sum_{|a|,|b|\leq H}F(an_1^k+bn_1^{\ell}+g(n_1))\overline{F}(an_2^k+bn_2^{\ell}+g(n_2)).   
\end{align*}

Let $\mathscr{M}=\mathscr{M}_{3A}$, where $\mathcal{M}_{3A}$ is defined in ~\eqref{eq:define_M}. 
For reasons that will become clear, we consider separately the following  contributions:
\begin{enumerate}
\item  $n_1=n_2$;
\item   $(n_1,n_2)\in \mathcal{N}$; and 
\item $(n_1,n_2)\in \mathcal{N}^{c}=\{(n_1,n_2)\in \N^2:n_1,n_2\leq x, n_1\neq n_2\}\setminus \mathcal{N}$, 
\end{enumerate} 
where\footnote{\label{foot2}In the proof of Corollary~\ref{cor:simpler}, one can take $\mathcal{Q}$ to be the empty set, so that $\mathscr{N}$ is the set of $(n_1,n_2)\in \mathcal{M}$ such that $n_1\neq n_2$.}
\begin{align*}
\mathcal{N}=\mathcal{M}\cap \{(n_1, n_2)\in \N^2:\,\, q\nmid n_1^d( n_1^{\ell-k}-n_2^{\ell-k})\,\forall\,q\in \mathcal{Q}\}.
\end{align*}
Lemma~\ref{le_typical} implies that  
$\#((\N\cap [1,x'])^2\setminus \mathcal{M})\ll x^2(\log x)^{-3A}$. Moreover, Lemma~\ref{l:polydplnit} yields
$$
\#\{(n_1, n_2)\in (\N\cap[1,x'])^2:
\exists q\in \mathcal{Q}  \textrm{ with } 
 q\mid n_1^d( n_1^{\ell-k}-n_2^{\ell-k}) \}|
\ll \sum_{q\in \mathcal{Q}}
\frac{x^2}{q^{1/(8d)}},$$  which is $\ll x^2 (\log x)^{-3A}$ by the assumption of 
Theorem ~\ref{prop_general}.
Therefore, 
\begin{align}\label{nbound}
\#\mathcal{N}^c\ll x^2(\log x)^{-3A}.
\end{align}

\subsubsection*{Contribution of $n_1=n_2$} We first bound the contribution of pairs $n_1=n_2$, noting that 
$|\alpha_{n_1}|\leq 1$. By assumption (1) of Theorem~\ref{prop_general}, Lemma~\ref{le_shiu}, and the assumption $x\leq H^{1/2d}$, their contribution is
\begin{align*}
&\ll \sum_{n\leq x} \sum_{|a|,|b|\leq H}|F(an^k+bn^{\ell}+g(n))|^2\\
&\ll \sum_{n\leq x}\sum_{|a|\leq H}\sum_{|b|\leq H}\tau_B(|an^k+bn^{\ell}+g(n)|)^2\\  
&\ll \sum_{n\leq x}\tau_B(n^{\ell})^2H^2(\log H)^{B^2-1}\log \log H\\
&\ll xH^2(\log H)^{B^2-1+B^{2\ell}-1}(\log \log H)\\
&\ll x^2H^2/(\log H)^{A},
\end{align*}
since $x\geq \exp((\log H)^{1/100})$, by assumption. 

\subsubsection*{Contribution of $(n_1,n_2)\in \mathcal{N}$}
We shall prove that
\begin{equation}\label{claim1}
\begin{split}
\sum_{\mathbf n \in \mathcal{N}} \alpha_{n_1}\overline\alpha_{n_2}
\hspace{-0.2cm}
\sum_{|a|,|b|\leq H}F(an_1^k+bn_1^{\ell}+g(n_1))&\overline{F}(an_2^k+bn_2^{\ell}+g(n_2))
\\
&\ll x^2H^2(\log x)^{-A}.
\end{split}
\end{equation}
 Let us write the left-hand side as
\begin{align}\label{eqq8}
\sum_{\mathbf n \in \mathcal{N}}  \alpha_{n_1}\overline\alpha_{n_2}
 \sum_{m_1,m_2\in \mathbb{Z} } \hspace{-0.2cm}
F(m_1)\overline{F(m_2)} \gamma(\mathbf m )
 , \end{align}
 where $$ 
 \gamma(\mathbf m ) =\#\left\{(a,b) \in (\Z\cap [-H,H] )^2: 
m_i=an_i^k+bn_i^{\ell}+g(n_i) \ \forall  i=1,2 \right\}
.$$
By  the triangle inequality,~\eqref{eqq8} is at most 
\begin{align*}
 \sum_{\mathbf n \in \mathcal{N}} \sum_{m_2\in \mathbb{Z}}\tau_{B}(m_2)
\left|\sum_{m_1\in \mathbb{Z}}F(m_1) \gamma(\mathbf m ) 
\right| ,
\end{align*} 
since 
$|\alpha_{n_1}|,|\alpha_{n_2}|\leq 1$.

We make the change of variables $m_i'=(m_i-g(n_i))/n_i^k$ to write this as 
\begin{align}\label{eq:9equ10} 
 \sum_{\mathbf n \in \mathcal{N}}
 \sum_{m_2'\in \mathbb{Z}}\tau_{B}(n_2^km_2'+g(n_2))\left|\sum_{m_1'\in \mathbb{Z}}
F(n_1^{k}m_1'+g(n_1))
 \gamma'(\mathbf m ') \right|, 
 \end{align} 
 where 
$$ \gamma'(\mathbf m ') =\#\left\{(a,b) \in (\Z\cap [-H,H] )^2: 
m'_i=a +bn_i^{\ell-k}  \ \forall  i=1,2 \right\}.
$$
For brevity, in what follows we put
$$
\Delta =n_1^{\ell-k}-n_2^{\ell-k}.
$$
Recall that when  $\mathbf n \in \mathcal{N}$ one has    $\min\{n_1,|n_1-n_2|\}>x/(\log x)^{3A}$, whence
\begin{align}\label{eqq17}
|\Delta|\geq |n_1 -n_2| n_1^{\ell-k-1}
 \gg x^{\ell-k}/(\log x)^{3A(\ell-k)}.    
\end{align} Straightforward linear algebra shows that $ \gamma'(\mathbf m ')$ is the indicator function of the simultaneous statements 
\begin{equation}\label{eq:event02}
\Delta \mid (m_1'-m_2'),\quad
|m_1'-m_2'| \leq | \Delta| H,\quad 
|m_2'n_1^{\ell-k}-m_1'n_2^{\ell-k}| \leq  |\Delta| H.
\end{equation}
The  latter two conditions can be merged as $m_1'\in J(\mathbf n,m_2')$ for some interval $J(\mathbf n,m_2')$ whose length  is  
\begin{align}\label{eqq26}
|J(\mathbf n ,m_2')|\leq 2 |\Delta| H/ n_2^{\ell-k } 
\leq   2H x^{\ell-k}/n_2^{\ell-k}
\leq   2  H (\log x)^{3A(\ell-k)}
, 
\end{align} 
thanks to the bound
$n_2>x/(\log x)^{3A}$.  The conditions in~\eqref{eq:event02} imply that 
$$ | m'_2| =|\Delta|^{-1}
 |(m_2' n_1^{\ell-k} -m_1'n_2^{\ell-k} )   +(m_1'- m'_2) n_2^{\ell-k}  | \leq H +n_2^{\ell-k}  H\leq 2 n_2^{\ell-k}H.
$$

We may now  conclude that~\eqref{eq:9equ10}  is at most 
\begin{equation}\label{eq:bbq}
 \sum_{\mathbf n \in \mathcal{N}} 
\sum_{\substack{|m_2'|\leq 2n_2^{\ell-k}H
}}\tau_{B}(n_2^km_2'+g(n_2))\left|\sum_{\substack{m_1'\in 
 J(\mathbf n,m_2') 
\\m_1'\equiv m_2'\bmod{\Delta}}}F(n_1^{k}m_1'+g(n_1))\right| .
\end{equation}
To prove\footnote{\label{foot3}Note that for  Corollary~\ref{cor:simpler},~
the fact that 
\eqref{eq:bbq} 
is $\ll x^2H^2(\log x)^{-A}$
follows directly from~\eqref{eq:F(n)}.}
 that this is $\ll x^2H^2(\log x)^{-A}$,
 we will use assumption (2) of Theorem \ref{prop_general}.  It will first be necessary 
to deal first with the terms for which 
\begin{equation}\label{eq:skp}
\gcd(n_1^{k}m_2'+g(n_1),\Delta)>\exp\left(\frac{   (   \log x)^{1/2}  }{    2\log \log x } \right)=x_2.
\end{equation}
Since $\Delta\leq x^{d}\leq H^{1/2}$, Lemma~\ref{le_shiu},~\eqref{eqq26} and the assumption $|F(n)|\leq \tau_B(n)$ allow us to bound this contribution  by
\begin{align*}
&\ll  
\sum_{\mathbf n \in \mathcal{N}}
\hspace{-0.4cm}
\sum_{\substack{ |m_2'|\leq 2 n_2^{\ell-k}H \\\gcd(n_1^{k}m_2'+g(n_1),\Delta)>x_2}}\hspace{-0,5cm} \tau_{B}(n_2^km_2'+g(n_2))\tau_B(n_1^k\Delta)
\frac{H(\log x)^{3A(\ell-k)}(\log H)^B}{|\Delta|}\\
& \ll    Hx^{k-\ell} (\log x)^{6A(\ell-k)}(\log H)^B \\
&\qquad \times
\sum_{\mathbf n \in \mathcal{N}}\sum_{\substack{|m_2'|\leq 2 n_2^{\ell-k } H
\\\gcd(n_1^{k}m_2'+g(n_1),\Delta)>x_2}}\hspace{-0,5cm}\tau_B(n_1^k\Delta) \tau_{B}(n_2^km_2'+g(n_2)) ,
\end{align*} 
thanks to~\eqref{eqq17}. 
Since $\log H\leq (\log x)^{100}$ by assumption, an application of   
Cauchy--Schwarz yields
\begin{align*}
&\ll   Hx^{k-\ell} (\log x)^{6A(\ell-k)+100B} S_1^{1/2}S_2^{1/2},
\end{align*}
where
\begin{align*}
S_1&=\sum_{\mathbf n \in \mathcal{N}}\sum_{|m_2'|\leq 2 x^{\ell-k } H}
\tau_B(n_1^k\Delta)^2\tau_{B}(n_2^km_2'+g(n_2))^2,\\
S_2&=\sum_{\mathbf n \in \mathcal{N}}\sum_{\substack{|m_2'|\leq 2 x^{\ell-k } H\\\gcd(n_1^{k}m_2'+g(n_1),\Delta)>x_2}} 1.
\end{align*}

The second property of the set $\mathcal M$ in~\eqref{eq:define_M} implies that 
\begin{equation}\label{eq:snow}
    \gcd(n_1^k,\Delta)\leq \gcd(n_1^d,n_2^d)=\gcd(n_1,n_2)^d\leq (\log x)^{3Ad}.
\end{equation}
Hence, by Lemma~\ref{lem:corlry}, we have 
$$ 
S_2\ll 
x^2\left(\frac{x^{\ell-k } H (\log x)^{3Ad 
}}{x_2^{1/2}}
 \exp\left(\frac{4  (   \log(3|\Delta|))^{1/2}  }{ (  \log \log(3|\Delta|) )^{3/2} } \right)
+ \tau(\Delta)\right) 
 .$$ 
Using the bounds   $|\Delta| \leq x^{\ell-k}$
and $\tau(\Delta) \ll |\Delta|^{0.1/(\ell-k)} \ll x^{0.1}$, this  leads to 
\begin{align*}
S_2&\ll
\frac{x^{2+\ell-k } H (\log x ) ^{3Ad
 } }{x_2^{1/2}}
 \exp\left(\frac{   4d^{1/2}(   \log(3x))^{1/2}  }{    (\log \log x)^{3/2} } \right)
+ x^{2.1}\\  &\ll  \frac{x^{ 2+\ell -k }H }{ \exp((\log x)^{0.49})}.
\end{align*}Next, Lemma~\ref{le_shiu} 
implies that 
$$ 
S_1
\ll 
x^{\ell-k } H 
(\log H)^{B^2-1}\log\log H \cdot 
\Sigma(x),
$$ 
with 
$$\Sigma(x)=   
\sum_{\substack{n_1,n_2\leq x\\n_1\neq n_2}}
\tau_B(n_1^k)^2
\tau_{B}(n_2^k)^2
\tau_B(\Delta)^2.$$   
Cauchy--Schwarz yields
$$\Sigma(x)
\ll \left(\sum_{n_1,n_2\leq x}
\tau_B(n_1)^{4k}
\tau_{B}(n_2)^{4k}\right)^{1/2}
\left(\sum_{\substack{n_1,n_2\leq x\\n_1\neq n_2}}
\tau_B(\Delta)^4\right)^{1/2}.
$$ 
Combining~\eqref{eq:standard} and  Lemma~\ref{lem:NT}, it now follows that 
\begin{align*}
S_1
&\ll 
x^{2+\ell-k}H(\log H)^{(B^2-1)+(B^{4k}-1)+(B^4-1)d/2}\log\log H   \\
&\ll 
x^{2+\ell-k}H(\log x)^{300dB^{4k}},
\end{align*}
since we are assuming that $\log H\leq (\log x)^{100}$.
Combining the bounds for $S_1$ and $S_2$, we see that the contribution to~\eqref{eq:bbq}
from the terms with~\eqref{eq:skp}  is crudely bounded by $\ll x^2H^2 (\log x)^{-A}$.

The proof of~\eqref{claim1} has now been reduced to proving  the estimate
\begin{equation}
\begin{split}\label{eqq14}
\sum_{\mathbf n \in \mathcal{N}}\sum_{\substack{|m_2'|\leq 2 n_2^{\ell-k}H
\\\gcd(n_1^{k}m_2'+g(n_1),\Delta)\leq x_2}}\hspace{-0.5cm}\tau_{B}(n_2^km_2'+g(n_2))
\left|S_3\right|
\ll x^2H^2(\log x)^{-A},
\end{split}\end{equation} 
where
$$
S_3=\sum_{\substack{m_1'\in 
 J(\mathbf n,m_2')  
\\m_1'\equiv m_2'\bmod{\Delta}}}F(n_1^{k}m_1'+g(n_1)).
$$
Making a change of variables $m=n_1^{k}m_1'+g(n_1)$, we see that 
$$
S_3=\sum_{\substack{m\in J'(\mathbf n,m_2')
\\
m\equiv u \bmod{q}}}F(m)    ,
$$ 
where $u$ is the unique solution to $u\equiv n_1^{k}m_2'+g(n_1) \bmod{\Delta}, u\equiv g(n_1)\bmod{n_1^k}$ modulo $q=  \textnormal{lcm}(n_1^k, |\Delta|)\leq x^{d}$, and 
$J'(\mathbf n,m_2')=n_1^{k}J(\mathbf n,m_2')+g(n_1)$  
 is an interval     with
 \begin{align}\label{eq:jbound}
 |J'(\mathbf n,m_2') | =  n_1^{k} |J(\mathbf n ,m_2')|\leq 2n_1^k|\Delta|H/n_2^{\ell-k},
\end{align}
  by~\eqref{eqq26}. Note that
$\gcd(n_1^k,\Delta)\leq  (\log x)^{3Ad}$, by \eqref{eq:snow},   
  so 
  $$q\geq n_1^k\Delta/(\log x)^{3Ad}.$$
  
Using the inequality   $\gcd(a_1,a_2a_3) \leq \gcd(a_1,a_2) \gcd(a_1,a_3)$,  
and   recalling 
the third property of $\mathcal M_{3A}$ in~\eqref{eq:define_M}, 
we infer that  
\begin{align*}
\gcd(u,q)
&\leq \gcd(g(n_1),n_1^{k})\gcd(n_1^{k}m_2'+g(n_1),\Delta) \\ 
&\leq  \exp((\log x)^{1/2}/ (2\log \log x )) x_2 \\
&=\exp((\log x)^{1/2}/ (\log \log x )) ,
\end{align*} 
where $x_2$ is defined in \eqref{eq:skp}.

By the construction of $\mathcal N$, the integer $q$ is not a multiple of any element of $\mathcal Q$. We may  therefore use   assumption (2) of Theorem ~\ref{prop_general} and~\eqref{eq:jbound}
 to deduce that 
  \begin{align*}
  S_3
  \ll \frac{ |J'(\mathbf n,m_2')|}{q (\log H )^{10Ad}}  
\ll  \frac{ n_1^k \Delta H/n_2^{\ell-k}}{q (\log H )^{10Ad}} 
&\ll \frac{   H (\log x)^{3Ad}/n_2^{\ell-k}}{ (\log H )^{10Ad}} \\
&\ll  \frac{   H }{x^{\ell-k}  (\log H )^{4Ad}} 
 ,  
 \end{align*}
 except when $|J'(\mathbf n,m_2')|\leq H^{1-\varepsilon}$. In the latter case, we apply  the crude bound 
 $|F(n)|\leq \tau_B(n)\ll (Hx^d)^{\varepsilon/10} $. This leads to the crude estimate
$$
S_3\ll (Hx^d)^{\varepsilon/10}  \left(\frac{|J'(\mathbf n,m_2')|}{q}+1\right) \ll \frac{n_1^k H^{1-0.7\varepsilon}}{q}
\ll \frac{ H}{ x^{\ell-k}  (\log H)^{ 100Ad}}
,$$
since $x^{d}\leq H^{1/2}$.
Thus, by Lemma~\ref{le_shiu} and~\eqref{eq:standard}, the left-hand side of~\eqref{eqq14} is 
\begin{align*}  
&\ll \sum_{n_1\leq x}\sum_{n_2\leq x} \sum_{|m_2'|\leq 2n_2^{\ell-k}H}\tau_{B}(n_2^km_2'+g(n_2))
   \frac{ H}{x^{\ell-k}  (\log H)^{4Ad}}\\
   &\ll \sum_{n_1\leq x}\sum_{n_2\leq x}\tau_B(n_2^k)(\log H)^{B-1}\frac{H^2}{(\log H)^{4Ad}}\\
&\ll x^2H^2   \frac{ (\log H)^{B^k-1+B-1}   }{     (\log H)^{ 4Ad}}. \end{align*}
 As $B^{2k}+B\leq 2dB^{2d}\leq A$, this completes the proof of~\eqref{claim1}. 
 
 \subsubsection*{Contribution of $(n_1,n_2) \in \mathcal{N}^c$}
 We shall prove that
\begin{equation}
\begin{split}
\label{claim2}
\sum_{\mathbf n \in \mathcal{N}^c} 
\alpha_{n_1}\overline\alpha_{n_2}
\hspace{-0.3cm}
\sum_{|a|,|b|\leq H}F(an_1^k+bn_1^{\ell}+g(n_1))&\overline{F}(an_2^k+bn_2^{\ell}+g(n_2))
\\
&\ll x^2H^2(\log x)^{-A}.
 \end{split}\end{equation}
 By arguing as in the case of $\mathbf n\in \mathcal{N}$ up to~\eqref{eq:bbq}, and using assumption (1) of Theorem~\ref{prop_general},  the left-hand side of~\eqref{claim2} is 
 \begin{align*}
&\leq \sum_{\mathbf n \in \mathcal{N}^{c}} 
\sum_{\substack{|m_2'|\leq 2 n_2^{\ell-k}H}}
\tau_{B}(n_2^km_2'+g(n_2))\sum_{\substack{m_1'\in 
J(\mathbf n,m_2') 
\\m_1'\equiv m_2'\bmod{\Delta}}}\tau_B(n_1^{k}m_1'+g(n_1)).
\end{align*} 
By Lemma~\ref{le_shiu} and~\eqref{eqq26},~\eqref{eqq17} we can bound this by
\begin{align*}
&\ll \sum_{\mathbf n \in \mathcal{N}^{c}} 
\sum_{\substack{|m_2'|\leq 2 n_2^{\ell-k}H}}
\tau_{B}(n_2^km_2'+g(n_2))\tau_{B}(n_1^k\Delta)H(\log H)^{B}/ n_2^{\ell-k}. 
\end{align*}
Our assumption $\exp((\log H)^{1/100})\leq x$ implies that $\log H\leq (\log x)^{100}$. 
Therefore, on
applying first Lemma~\ref{le_shiu} and then H\"older's inequality, the  previous expression is seen to be
\begin{align*}
&\ll 
H^2(\log x)^{101B}
\sum_{\mathbf n \in \mathcal{N}^{c}} \tau_{B}(n_2)^k\tau_{B}(n_1^k\Delta)\\
&\ll  H^2(\log x)^{101B}
(\#\mathcal{N}^{c})^{1/2}\left(\sum_{\substack{n_1,n_2\leq x\\n_1\neq n_2}}
\tau_{B}(|n_1^{\ell-k}-n_2^{\ell-k}|)^4\right)^{1/4}\\
&\quad \times
\left(\sum_{n_1,n_2\leq x}
\tau_{B}(n_2)^{4k}\tau_{B}(n_1)^{4k}\right)^{1/4}
\end{align*}

Recall from~\eqref{nbound} that  
$\#\mathcal{N}^{c}\ll x^2(\log x)^{-3A}$.
Appealing to Lemma~\ref{lem:NT} and the standard moment bound~\eqref{eq:standard}, we see  the 
 left-hand side of~\eqref{claim2} is seen to be
$$
\ll x^2H^2(\log x)^{-3A/2+101B+B^{4k}/2+ (B^4-1)d/4} 
\ll x^2H^2(\log x)^{-A},
$$
using 
$101B+B^{4k}/2+(B^4-1)d/4\leq 200dB^{4d}\leq A/2$.
This proves~\eqref{claim2}, which thereby completes the proof of 
Theorem~\ref{prop_general}.

\section{Equidistribution of the von Mangoldt and Liouville functions}\label{sec:pntaps}

Theorem~\ref{prop_general} essentially reduces the proofs of Theorems~\ref{thm_chowla} and~\ref{thm_main} to the problem of controlling sums of the
Liouville  function and the von Mangoldt function in short intervals and to polynomially large moduli, for all but a few moduli. This is the object of the following two results.

\begin{proposition}[Liouville function in  short intervals and progressions]\label{prop_PNTAPs}
Let $\varepsilon>0$ be small but fixed and let $A\geq 1$ be fixed. Let $x\geq x_0(A,\varepsilon)$. Let $c_0=5/36$. Then there exists a set $\mathcal{Q}\subset [(\log x)^{A},x^{c_0}]$ of integers satisfying
    \begin{align}\label{eq:Q-size}
    \sum_{q\in \mathcal{Q}}\frac{1}{q^{\varepsilon}}\ll (\log x)^{-30A/\varepsilon},
    \end{align}
 such that for any integer $1\leq q\leq x^{c_0}$ that is not a multiple of an element of $\mathcal{Q}$, we  have
    \begin{align}\label{eqq3}
   \max_{\substack{a\bmod q\\\gcd(a,q)\leq x^{\varepsilon}}}\sup_{\substack{I\subset [1,x]\\|I|\geq x^{1-c_0+2\varepsilon}}} \left|\sum_{\substack{n\in I\\n\equiv a\bmod q}}\lambda(n)\right|\ll \frac{|I|}{q(\log x)^{A}}.    
    \end{align}
\end{proposition}

\begin{proposition}[Von Mangoldt function in short intervals and progressions]\label{prop_PNTAPs'}
Let $\varepsilon>0$ be small but fixed and let $A\geq 1$ be fixed. Let $x\geq x_0(A,\varepsilon)$. Let $c_0=5/36$. Then there exists a set $\mathcal{Q}\subset [(\log x)^{A},x]$ of integers satisfying~\eqref{eq:Q-size}
 such that for any integer $1\leq q\leq x^{c_0}$ that is not a multiple of an element of $\mathcal{Q}$, we  have
    \begin{align}\label{eqq3b}
   \max_{\substack{a\bmod q\\ \gcd(a,q)\leq x^\ve}}\sup_{\substack{I\subset [1,x]\\|I|\geq x^{1-c_0+2\varepsilon}}} \left|\sum_{\substack{n\in I\\n\equiv a\bmod q}}\Lambda(n)-\mathbf 1_{\gcd(a,q)=1}\frac{x}{\varphi(q)}\right|\ll \frac{|I|}{q(\log x)^{A}}.    
    \end{align}
\end{proposition}

We note that for $q=1$ (or $q$ being a power of $\log x$), Propositions~\ref{prop_PNTAPs} and~\ref{prop_PNTAPs'} could be deduced from classical works on the Liouville function in short intervals (due to Motohashi~\cite{moto} and Ramachandra~\cite{rama}) and on primes in short intervals (due to Huxley~\cite{huxley-inventiones}). For the case where $q$ is a power of $x$, we can for the most part adapt the arguments of Ramachandra~\cite{rama}, but we need to exclude those moduli for which the $L$-functions modulo $q$ do not have a good zero-free region, and we then use zero density estimates to show that the set of such bad $q$ is sparse in the sense of~\eqref{eq:Q-size}.

Let $N(\sigma,T,\chi)$ denote the number of zeros of $L(s,\chi)$ in the region $\real(s)\geq \sigma$ and $|\textnormal{Im(s)}|\leq T$ (counted with multiplicities). Let $\sum^*_{\chi\bmod{q}}$ stand for a sum over the primitive characters modulo $q$, and let $\prod^*_{\chi\bmod{q}}$ stand for a product over the primitive characters modulo $q$. The first ingredient for proving Propositions~\ref{prop_PNTAPs} and~\ref{prop_PNTAPs'} is a suitable zero density estimate. 

\begin{lemma}[Zero density estimates] \label{le:zeros} Let $Q,T\geq 1$, $q\in \mathbb{N}$, $\varepsilon>0$ and $\sigma\in [1/2,1]$. Then we have
\begin{align*}
 \sum_{\chi\bmod q} N(\sigma,T,\chi)\ll_{\varepsilon} (qT)^{(12/5+\varepsilon)(1-\sigma)}   
\end{align*}
and
\begin{align*}
 \sum_{q\leq Q}\,\,\asum_{\chi\bmod q} N(\sigma,T,\chi)\ll_{\varepsilon} (Q^2T)^{(12/5+\varepsilon)(1-\sigma)}.      
\end{align*}
\end{lemma}

\begin{proof}
By~\cite[Grand Density Theorem 10.4 and the remark after it]{iw-kow} with $Q=1$, we have
\begin{align*}
 \sum_{\chi\bmod q} N(\sigma,T,\chi)\ll_{\varepsilon}~& (qT)^{(2+\varepsilon)(1-\sigma)}
\\ 
&  +(qT)^{(1-\sigma)\cdot\min\{3/(2-\sigma),3/(3\sigma-1)\}}(\log(qT))^{A}    ,
\end{align*}
for some constant $A$. This implies the first estimate of the lemma for $\sigma\leq 4/5$, say. For $\sigma\in (4/5,1]$, the first estimate instead follows from a result of Jutila~\cite{jutila}. The second estimate of the lemma follows similarly from~\cite[Grand density theorem 10.4]{iw-kow} (with $k=1$) and~\cite{jutila}. 
\end{proof}

In what follows, let $\chi_0$ denote the principal Dirichlet character modulo $q$.

\begin{lemma} 
\label{lem:oast}
Let $\varepsilon>0$ and $A\geq 2$ be fixed. Let $x\geq 2$, let $1\leq q,h\leq x$, and suppose that $q^2x/h\leq x^{5/12-\varepsilon/2}$. Let $\mathcal{X}$ be the collection of all Dirichlet characters $\chi$ of modulus $q$ for which $L(s,\chi)$ possesses the zero-free region
\begin{align}\label{eq:zerofree}
 \left\{s\in \mathbb{C}:\quad \Re(s)\geq 1-\frac{2A\log \log x}{\varepsilon\log x},\quad |\Im(s)|\leq \frac{2x(\log x)^{3A}}{h}\right\}.
\end{align}
Then we have
\begin{align*}
\sum_{\chi\in \mathcal{X}}\left|\sum_{x\leq n\leq x+h}\lambda(n)\chi(n)\right|\ll \frac{h}{(\log x)^{A}}
\end{align*}
and
\begin{align*}
\sum_{\chi\in \mathcal{X}}\left|\sum_{x\leq n\leq x+h}\Lambda(n)\chi(n)-h\mathbf{1}_{\chi=\chi_0}\right|\ll \frac{h}{(\log x)^{A}}.    
\end{align*}
\end{lemma}

%Before proving Lemma~\ref{lem:oast}, let us note the following quick corollary.

%\begin{corollary}\label{le:rama}
%Let $A\geq 2$ and $\varepsilon>0$ be fixed. Then for any $x^{7/12+\varepsilon}\leq h\leq x$ and $1\leq a\leq q\leq (\log x)^{A}$, we have 
%\begin{align}\label{eq:lambdabound1}
%\left|\sum_{\substack{x\leq n\leq x+h\\n\equiv a\bmod q}}\lambda(n)\right|\ll \frac{h}{q(\log x)^A}
%\end{align}
%and 
%\begin{align}\label{eq:lambdabound2}
%\sum_{\substack{x\leq n\leq x+h\\n\equiv a\bmod q}}\Lambda(n)=\frac{h}{\varphi(q)}\mathbf{1}_{\gcd(a,q)=1}+O\left(\frac{h}{\varphi(q)(\log x)^A}\right).    
%\end{align}
%\end{corollary}

%\begin{proof} We may assume that $\gcd(a,q)=1$ in~\eqref{eq:lambdabound1} by the complete multiplicativity of $\lambda$, and we may make the same assumption in~\eqref{eq:lambdabound2} as otherwise that bound is trivial.

%Let 
%$$a_n=\lambda(n)\quad \textnormal{or}\quad a_n=\Lambda(n).$$
%Since $q\leq (\log x)^{A}$, all the $L$-functions $L(s,\chi)$ for $\chi\bmod q$ a character have the Vinogradov--Korobov zero-free region
%$$
%\left\{s\in \mathbb{C}\colon \textnormal{Re}(s)\geq 1-\frac{C_{\varepsilon}}{(\log x)^{2/3+\varepsilon}},\quad |\textnormal{Im}(s)|\leq x^2\right\}
%$$
%for some constant $C_{\varepsilon}>0$. Now the claim follows by using orthogonality of characters to write
%\begin{align*}
%\sum_{\substack{x\leq n\leq x+h\\n\equiv a\bmod q}}a_n=\frac{1}{\varphi(q)}\sum_{\chi\bmod q}\overline{\chi}(a)\sum_{x\leq n\leq x+h}a_n\chi(n)    
%\end{align*}
%and applying Lemma~\ref{lem:oast}.
%\end{proof}

\begin{proof}
%[Proof of Lemma~\ref{lem:oast}] 
Let 
$T=qx(\log x)^{3A}/h$ 
 and    
\begin{equation}\label{eq:beta-defn}
\beta=\frac{2A\log \log x}{\varepsilon\log x}.
\end{equation}
Let 
$$F(s,\chi)=\frac{L(2s,\chi)}{L(s,\chi)}\quad \textnormal{or} \quad F(s,\chi)=\frac{L'(s,\chi)}{L(s,\chi)}.$$
Then, writing
\begin{align*}
a_{n,\chi}=\lambda(n)\chi(n)\quad \textnormal{or}\quad  a_{n,\chi}=\Lambda(n)\chi(n),
\end{align*}
respectively, we have $F(s,\chi)=\sum_{n\geq 1}a_{n,\chi}n^{-s}$.  By Perron's formula, for $u=1+1/\log x$ we can write
\begin{align}\label{eq:Perron}
\sum_{x\leq n\leq x+h}
a_{n,\chi}=\frac{1}{2\pi i}\int_{u-iT}^{u+iT}F(s,\chi)
\frac{(x+h)^s-x^s}{s}\, \d s+O\left(\frac{x(\log x)^2}{T}\right).
\end{align}
Let \begin{align}
\label{eq:M}
M_{\chi}=\begin{cases}h\mathbf{1}_{\chi=\chi_0} 
&\,\text{ if $a_{n,\chi}=\Lambda(n)\chi(n)$\,
for all\, $n\in \mathbb N$,
}\\
0 &\,\text{ if $a_{n,\chi}=\lambda(n)\chi(n)$\,
for all\, $n\in \mathbb N$.}
\end{cases}
\end{align}
Recalling the choice of $T$, the error term in~\eqref{eq:Perron} is $\ll h(\log x)^{-3A+2}/q$,
whence
\begin{align}\begin{split}\label{eq:Fintegral}
&\sum_{\chi\in \mathcal{X}}\left|\sum_{x\leq n\leq x+h}
a_{n,\chi}-M_{\chi}\right| \\
&\qquad = \sum_{\chi\in \mathcal{X}}\left|\frac{1}{2\pi i} \int_{u-iT}^{u+iT}F(s,\chi)
\frac{(x+h)^s-x^s}{s}\, \d s -M_{\chi}\right|+O\left(\frac{h}{(\log x)^{A}}\right).
\end{split}
\end{align}

We now replace the integral over $[u-iT,u+iT]$ with an integral over a polygonal contour $C$, which we proceed to define, and which is based on the Hooley--Huxley contour from Ramachandra’s paper~\cite{rama}.
%Let $\rho_1,\dots , \rho_k$ be all the zeros of $L(s,\chi)$ with real part $\geq 1/2-1/\log T$ and imaginary part in 
%$[-T,T]$, arranged in order of their  increasing imaginary part. 
The contour $C$ consists of vertical lines 
$$
[\sigma_1+iy_1,\sigma_1+iy_2], [\sigma_2+iy_2, \sigma_2+iy_3], \dots ,[\sigma_{J-1}+iy_{J-1},\sigma_{J-1}+iy_J],
$$ 
with $y_1=-T, y_J=T$, and of horizontal lines that connect them together to form one polygonal line. 
The $\sigma_j, y_j$ are determined as follows:
\begin{itemize}
\item[(i)]  $y_j=-T+2(j-1)/\log T$, for $1\leq  j<J,$ where $J$ is the largest integer $j$ such that $y_{j}<T$.
\item[(ii)]  $\sigma_j$ is $1/\log x$ plus the largest real part of those zeros of $\prod_{\chi\in \mathcal{X}}L(s,\chi)$ that have imaginary part in $[y_j-1/\log T,y_{j+1}+1/\log T]$. If there are no zeros with imaginary part in this range, we  set $\sigma_j=1/2$.
\end{itemize}
Since $\prod_{\chi\in \mathcal{X}}L(s,\chi)$ has the zero-free region 
$\real(s)\geq 1-\beta$ and $|\im(s)|\leq 2T$,
where $\beta$ is given by~\eqref{eq:beta-defn},
it follows  that  $\sigma_j\leq 1-\beta+1/\log x$ for $1\leq j<J$.

We now replace the integral over $[u-iT,u+iT]$ in~\eqref{eq:Fintegral} with an integral over $C$, noting that all the zeros of $L(s,\chi)$ are to the left of $C$ (and we cross a pole of $F(s,\chi)((x+h)^s-x^s)/s$ at $s=1$, giving a residue $M_{\chi}$, if and only if $\chi$ is principal). 
The mean value theorem gives
$|(x+h)^s-x^s|\ll h|s| x^{\real(s)-1}$ for $x\geq 2$ and $s\in \mathbb{C}$. 
Hence, on moving the line of integration in~\eqref{eq:Fintegral}, we are led to the bound 
\begin{equation}
\begin{split}
\label{eq:integrate}
\frac{1}{\varphi(q)}\sum_{\chi \in \mathcal{X}}\left|\sum_{x\leq n\leq x+h}
a_n-M_{\chi}\right|
\ll~& h \frac{1}{\varphi(q)}\sum_{\chi \in \mathcal{X}} \int_C |F(s,\chi)|x^{\real(s)-1} \d s\\
 &+h(\log x)^{-A}\\
\ll~&  h (\log x) \int_C x^{\real(s)-1} \d s
+h(\log x)^{-A},
\end{split}
\end{equation}
where we have observed that $|L(2s,\chi)|\ll \log x$ 
since $\real(2s)\geq 1$ and $|\Im(2s)|\leq 4T$ throughout $C$, and
\begin{align*}
\frac{1}{|L(s,\chi)|}\ll \log x,\quad \frac{|L'(s,\chi)|}{|L(s,\chi)|}\ll \log x    
\end{align*}
by~\cite[Thm.~11.4]{montgomery-vaughan}, since $C$ stays distance $\geq 1/\log x$ to the right of any zeros of $L(s,\chi)$.

The contribution to the integral in~\eqref{eq:integrate}
from the vertical line  $[\sigma_{j}+iy_{j},\sigma_{j}+iy_{j+1}]$ and the horizontal line $[\sigma_j+iy_{j+1},\sigma_{j+1}+iy_{j+1}]$
is
$$
\ll h(x^{\sigma_j-1}+x^{\sigma_{j+1}-1})\log x,
$$
so it now suffices to show that
\begin{align}\label{eq:Sigma}
\Sigma\coloneqq \sum_{j< J} x^{\sigma_j-1}\ll (\log x)^{-A-1}.
\end{align}
The contribution to $\Sigma$ from those $j$ with $\sigma_j\leq 1/2$ is $\ll x^{-1/2}qT(\log qT)$.
Since $qT= q^2x(\log x)^{3A}/h\ll x^{5/12-\varepsilon/2+o(1)}$, this is clearly
$\ll (\log x)^{-A-1}$. Next we  consider the contribution to $\Sigma$ from those $j$ for which $
\sigma_j\in [\sigma-1/\log T,\sigma]$, for a given $\sigma\in [1/2,1-\beta+1/\log x]$. Their contribution to $\Sigma$ is
$$
\ll x^{\sigma-1}\sum_{\chi\in \mathcal{X}}N(\sigma-2/\log T,T, \chi).
$$
By Lemma~\ref{le:zeros}, this is 
$$
\ll  x^{\sigma-1}  (qT)^{(12/5+\varepsilon/2)(1-\sigma)},
$$
for any fixed $\varepsilon>0$.
If $x>(qT)^{12/5+\varepsilon/2}x^{\varepsilon/10}$, then this is
$
\ll x^{-\varepsilon(1-\sigma)/10}
$
and the claimed bound~\eqref{eq:Sigma} follows by summing over $1/2\leq \sigma\leq 1-\beta+2/\log T$ with 
$\sigma \in \frac{1}{\log T}\mathbb{Z}$. Since $qT\ll x^{5/12-\varepsilon/2+o(1)}$, as above, the claimed inequality for $x$ holds. We now conclude that
\begin{align*}
\sum_{\chi \in \mathcal{X}}\left|\sum_{x\leq n\leq x+h}a_n-M_{\chi}\right|\ll \frac{h}{(\log x)^A},    
\end{align*}
as desired.
\end{proof}

We are now ready to prove Propositions~\ref{prop_PNTAPs} and~\ref{prop_PNTAPs'}.

\begin{proof}[Proof of Propositions~\ref{prop_PNTAPs} and~\ref{prop_PNTAPs'}] We may assume that $\varepsilon>0$ is sufficiently small.
Let $$
 T=2x^2$$ and    
\begin{equation*}
\beta=10\frac{A}{\varepsilon}\frac{\log \log x}{\log x}.
\end{equation*}
We choose $\mathcal{Q}$ to be the set of $q\in \NN$ such that  
$(\log x)^{100A/\varepsilon^2}\leq q\leq x^{c_0}$ and 
$$\prod^*_{\chi\bmod q}L(s,\chi)= 0$$ for some $s\in \mathbb{C}$ with $\real(s)\geq 1-\beta$ and $|\textnormal{Im(s)}|\leq 2T$.

Using Lemma~\ref{le:zeros}, for any $(\log x)^{100A/\varepsilon^2}\leq Q\leq x^{c_0}$ we can  estimate
 \begin{align*}
\#(\mathcal{Q}\cap [Q,2Q])&\leq \sum_{Q\leq q\leq 2Q}\,\,\asum_{\chi\bmod q}N(1-\beta,2T,\chi)\\
&\ll_{\varepsilon} (Q^2T)^{(12/5+\varepsilon)\beta}\\
&\ll x^{(2c_0+2)\cdot (12/5+\varepsilon)\beta}\\
&\ll (\log x)^{30A/\varepsilon}\\
&\ll Q^{\varepsilon/2}.    
\end{align*}
 Now, by splitting the sum into dyadic intervals, we can estimate
 \begin{align*}
    \sum_{q\in \mathcal{Q}}\frac{1}{q^{\varepsilon}}\ll \sum_{2^{k+1}\geq (\log x)^{100A/\varepsilon^2}} \sum_{q\in \mathcal{Q}\cap [2^k,2^{k+1}]}\frac{1}{2^{k\varepsilon}}\ll \sum_{2^{k+1}\geq (\log x)^{100A/\varepsilon^2}} 
    \frac{1}{2^{k\varepsilon/2}}.
 \end{align*}
This is clearly $O((\log x)^{-30A/\varepsilon})$, which  therefore confirms~\eqref{eq:Q-size}.

 We now turn to the verification of~\eqref{eqq3} and~\eqref{eqq3b}. 
Let $\mathcal{R}$ be the set of integers that are not multiples of any element of $\mathcal{Q}$. Let 
$$a_n=\lambda(n)\quad \textnormal{or}\quad a_n=\Lambda(n)-\frac{x}{\varphi(q)}\mathbf{1}_{\gcd(n,q)=1}.$$ 
Choose an interval $I\subset [1,x]$ with $|I|\geq x^{1-c_0+2\varepsilon}$ and integers $1\leq a\leq q\leq x^{c_0}$, with $q\in \mathcal{R}$ and $\gcd(a,q)\leq x^{\varepsilon}$ such that 
\begin{align*}
\left|\sum_{\substack{n\in I\\n\equiv a\bmod q}}a_n\right| \end{align*}
is maximal. Let $d=\gcd(a,q)$. Then, depending on the case, using either the complete multiplicativity of $\lambda$ or the fact that $\gcd(n,q)>1$ and $\Lambda(n)\neq 0$ implies that $n=p^k$ with $p\mid q$, we 
are left with
proving that for all coprime $a,q$ with $1\leq a\leq q\leq x^{c_0}$ we have
\begin{align}\label{eqq4}
\sup_{\substack{I'\subset [1,x]\\|I'|\geq x^{1-c_0+\varepsilon}}} 
\left|\sum_{\substack{n\in I'\\n\equiv a\bmod q}}a_n\right|\ll \frac{|I'|}{q(\log x)^{A}}. \end{align}
Let $I'=[x',x'+h]$ be the interval that maximises the left-hand side of~\eqref{eqq4}, with $x'\leq x$ and 
$h\geq x^{1-c_0+\varepsilon}$. By shortening $I$ by a negligible amount if necessary, we may assume that $x'\geq x^{1/2}$, say.
By the orthogonality of characters, we have
\begin{align*}
\sum_{\substack{x'\leq n\leq x'+h\\n\equiv a\bmod q}}a_n&=\frac{1}{\varphi(q)}\sum_{\chi\bmod q}\overline{\chi(a)}\sum_{x'\leq n\leq x'+h}a_n\chi(n).
\end{align*}
Since $q\in \mathcal{R}$, we have  $r\not \in \mathcal{Q}$ for all $r\mid q$.
Thus  it follows that $L(s,\chi)$ obeys the zero-free region~\eqref{eq:zerofree} for all characters $\chi \pmod q$, with $2A$ in place of $A$. 
Noting that 
$h\geq  x^{1-c_0+\varepsilon}$ and $3c_0=\frac{5}{12}$, we may apply   Lemma~\ref{lem:oast} 
to deduce that 
\begin{align*}
\frac{1}{\varphi(q)}\sum_{\chi \bmod q}\left|\sum_{x'\leq n\leq x'+h}a_n\chi(n)\right|
\ll \frac{h}{\varphi(q)(\log x)^{2A}}.   
\end{align*}
The claimed estimate follows on noting that  $1/\varphi(q)\ll (\log \log q)/q$.
\end{proof}

\section{Polynomial Chowla on average}\label{sec:chowla}

In this section we apply Theorem~\ref{prop_general} and Proposition~\ref{prop_PNTAPs} to prove Theorem~\ref{thm_chowla}.
In what follows, we think of $d$ and $A\geq 1$ as being fixed. Without loss of generality, we may assume that $A$ is large. Let $\mathscr{C}\subset \mathbb{Z}^{d+1}$ be a combinatorial cube of side length $H$. We may assume that the dimension of $\mathscr{C}$ is two, since that	 case directly implies the cases of larger dimension.  Our goal in this section is to show that
\begin{equation} \label{eqq20}
\sum_{f\in \mathscr{C}}\left|\sum_{n\leq x}\lambda(f(n))\right|^2\ll x^2H^2(\log x)^{-A},
\end{equation}
uniformly for $x\in [H^{c},2H^{c}]$, where $c$ is as in Theorem~\ref{thm_chowla}. Theorem~\ref{thm_chowla} will then follow directly 
from Chebyshev's inequality.

Suppose that  $\mathscr{C}$ has the $k$th and $\ell$th coordinates as its two free coordinates, with $k<\ell$. 
In order
to prove~\eqref{eqq20}, it suffices to prove that 
\begin{equation}\label{eq:clothy}
\sum_{|a|,|b|\leq H}\left|\sum_{n\leq x}\lambda(an^k+bn^{\ell}+g(n))\right|^2\ll x^2H^2(\log x)^{-A},
\end{equation}
uniformly for polynomials $g\in \mathbb{Z}[t]$ of degree $\leq d$ which have  coefficients in $[-H,H]$ and for which the degree $k$ and $\ell$ coefficients are zero.
We claim that it suffices to prove that, for any arithmetic 
function $\alpha_n: \NN\to  [-1,1]$, we have 
\begin{align}\label{eqq23}
\sum_{|a|,|b|\leq H}\left|\sum_{n\leq x}\alpha_n\lambda(an^k+bn^{\ell}+g(n))\right|^2\ll x^2H^2(\log x)^{-A},
\end{align}
uniformly for polynomials $g\in \mathbb{Z}[t]$ of degree $\leq d$ which have  coefficients in $[-H,H]$ and degree $k$ and $\ell$ coefficients equal to zero, and  which satisfy the additional condition $g(0)\neq 0$. To see this we write $S(H)$ for the sum on the left-hand side of~\eqref{eq:clothy} and we suppose that  $g(0)=0$. Then 
$g(n)=g_0n^d+\cdots +g_{d-1}n$, for appropriate integers $g_0,\dots,g_{d-1}$, with $g_{d-k}=g_{d-\ell}=0$.
Let $j\geq 1$ be the minimal integer such that  $g_{d-j}\neq 0$, taking 
$j=d+1$ if $g$ vanishes identically.  Then it follows that
$\lambda(an^{k}+bn^{\ell}+g(n))$ is equal to
$$
\begin{cases}
\lambda(n)^j\lambda(an^{k-j}+bn^{\ell-j}+g_0n^{d-j}+\cdots+g_{d-j}) & \text{ if $j\leq k$,}\\
\lambda(n)^k\lambda(a+bn^{\ell-k}+g_0n^{d-k}+\cdots+g_{d-j}n^{j-k}) & \text{ if $j> k$.}
\end{cases}
$$
If $j\leq k$,  we see that 
$$
S(H)=
\sum_{|a|,|b|\leq H}\left|\sum_{n\leq x}\alpha_n 
\lambda(an^{k-j}+bn^{\ell-j}+g^\dag(n))\right|^2,
$$
where $\alpha_n=\lambda(n)^j$ and $g^\dag(n)= g_0n^{d-j}+\cdots+g_{d-j}$.
This is exactly of the form 
\eqref{eqq23}, 
with $g(0)= g^\dag(0)\neq 0$.
Alternatively, if $j>k$ we have
\begin{align*}
S(H)
&=
\sum_{|a|,|b|\leq H}\left|\sum_{n\leq x}\alpha_n 
\lambda(a+bn^{\ell-k}+g^\ddag(n))\right|^2\\&=
\sum_{|a|,|b|\leq H}\left|\sum_{n\leq x}\alpha_n 
\lambda(a+1+bn^{\ell-k}+ g^\ddag(n))\right|^2 +O(Hx^2),
\end{align*}
on shifting $a$ by $1$, 
where $\alpha_n=\lambda(n)^k$ and $g^\ddag(n)= g_0n^{d-k}+\cdots+g_{d-j}n^{j-k}$.
The first term 
is exactly of the form 
\eqref{eqq23}, 
with $g(0)=1+ g^\ddag(0)=1$,
and the error term of $O(Hx^2)$ is satisfactory.

\begin{proof}[Proof of Theorem~\ref{thm_chowla}] By the discussion above, it suffices to prove~\eqref{eqq23}. We apply Theorem 
~\ref{prop_general} (with $B=1$) to the function $F(n)=\lambda(n)$. This directly implies~\eqref{eqq23}  once we verify the hypotheses of that theorem. Assumption (1) is clear. Assumption (2) in turn follows from Proposition~\ref{prop_PNTAPs} (with $\varepsilon=1/(8d)$ there), provided that $x^d\leq (2Hx^d)^{c_0}$ and $(2Hx^d)^{1-c_0}\leq H^{1-\varepsilon}$, where $c_0=5/36$ is as in Proposition~\ref{prop_PNTAPs}. Since by assumption $x\leq 2H^{c}$ with $c<5/(31d)$, these conditions indeed hold for some small enough $\varepsilon>0$.
\end{proof}

\section{Bateman--Horn on average}\label{sec:bateman-horn}

In this section we apply Theorem~\ref{prop_general} to prove Theorem~\ref{thm_main_multi} (and hence Theorem~\ref{thm_main} and Corollary~\ref{cor:H}).
Let $d$ and $A\geq 1$ be fixed. We first prove a localised version of the theorem for all tuples of polynomials.

\subsection{An analogue of Bateman--Horn for rough numbers}

To deal with the von Mangoldt function, we need to apply the $W$-trick. Let
\begin{align}\label{eq:w}
  w=\exp(\sqrt{\log x}),\quad W=\prod_{p\leq w}p.   
\end{align}
Define a $W$-tricked model for the von Mangoldt function by
\begin{align*}
\Lambda_{w}(n)=\frac{W}{\varphi(W)}\mathbf 1_{\gcd(n,W)=1}.    
\end{align*}
Note that $\Lambda_w(n)\leq W/\varphi(W)\ll \log x$, for any $n\in \NN$. 
Recall the definition~\eqref{eq:SS-fx'} of $\mathfrak{S}_{f_1,\dots,f_r}(x)$, for an  $r$-tuple 
$f_1,\dots,f_r\in \ZZ[t]$.

\begin{lemma}\label{le:lambdaw}
Let $A\geq 1$ and $d,r\in \mathbb{N}$ be fixed, let $x\geq 2$, and let $w$ be as in \eqref{eq:w}. Let $f_1,\ldots, f_r\in \mathbb{Z}[t]$ be polynomials of degree $\leq d$. Then \begin{align}\label{eq:wtrickbound}\begin{split}
\sum_{n\leq x}\Lambda_w(f_1(n))\cdots \Lambda_w(f_r(n))&= \mathfrak{S}_{f_1,\ldots, f_r}(x) x+O( x  (\log x)^{-A}).
\end{split}
\end{align}
\end{lemma}

This lemma will follow relatively easily from the fundamental lemma of sieve theory. The following result is~\cite[Fundamental lemma~6.3]{iw-kow}.
\begin{lemma} 
\label{lemiwakow}
Let $\kappa>0$ and $D>1$. There exist two sequences of real numbers $(\lambda_k^+)$ and $(\lambda_k^-)$, supported on square-free numbers and depending 
only on $\kappa $ and $D$,
with the following properties: \[ \lambda^\pm_1=1, \ \ 
|\lambda^\pm_k|\leq 1 \ \ \ \forall k<D, 
\
\
\lambda^\pm_k=0\ \ \ \forall k\geq D ,\] and for any integer $n\geq1$,
$$ \sum_{k\mid n } \lambda^-_k\leq \mathbf{1}_{n=1} \leq\sum_{k\mid n }  \lambda^+_k. 
$$ 
Let $P(w)=\prod_{p\leq w}p$ and let $s=\log D/\log w$.
Then, 
for any multiplicative function $h(k)$ with $0\leq h(k)<1$,    satisfying 
$$
\prod_{a\leq p < b}(1-h(p))^{-1} \leq \left(\frac{\log b}{\log a}\right)^\kappa
\left(1+\frac{K}{\log a}\right),
$$
for all $2\leq a<b \leq D$, we have
$$
\sum_{k\mid P(w)} 
\lambda^\pm_k h(k)=
\left(1+O\left(\mathrm e^{-s} \left(1+\frac{K}{\log w} \right)^{10}\right)\right)\prod_{p\leq w } (1-h(p) ),
$$
the implied constant depending only on $\kappa$. 
\end{lemma} 

\begin{proof}[Proof of Lemma~\ref{le:lambdaw}]
Let 
\begin{align*}
h(p)&= \mathbf{1}_{p\leq w} 
\frac{\#\{u\in \mathbb{F}_p: f_1(u)\cdots f_r(u)=0\}}{p},
\end{align*}
and extend $h$ multiplicatively to all square-free integers. 
In proving~\eqref{eq:wtrickbound} we can assume that 
$h(p)$ is strictly less than $1$  for all     primes $p\leq w$.
This is because in the opposite case both  $\Lambda_w(f_1(n))\cdots \Lambda_w(f_r(n))$  and $\mathfrak{S}_{f_1,\ldots, f_r}(x)$ would be zero, 
for $n\leq x$.
This implies that $f_1\cdots f_r$ does not have a fixed prime divisor $p\leq w$.
Hence, we  have $h(p) \leq dr/p$ whenever $p\leq w$, by Lagrange's theorem. It follows from  Mertens' theorem that
  \[
  \prod_{a\leq p < b}(1-h(p))^{-1} \leq
\prod_{\substack{a\leq p < b\\p\leq w}}\left(1-\frac{dr}{p} \right)^{-1} 
\leq 
\left(\frac{\log b}{\log a}\right)^{dr}
\left(1+\frac{K_{d,r}}{\log a}\right)
,\] where $K_{d,r}$ depends at most on $d$ and $r$.
Applying  Lemma~\ref{lemiwakow} with $\kappa=dr$ yields 
\begin{align*}
\sum_{n\leq x}\Lambda_w(f_1(n))\cdots \Lambda_w(f_r(n)) 
&\leq 
\left(\frac{W}{\varphi(W)}\right)^r \sum_{k\mid W} \lambda_k^+ \sum_{\substack{n\leq x\\  f_1\cdots f_r(n)\equiv 0 \bmod k}} 1
\\
&=\left(\frac{W}{\varphi(W)}\right)^r \sum_{k\mid W} \lambda_k^+ 
\hspace{-0.2cm}
\sum_{\substack{m\bmod k\\  f_1\cdots f_r(m)\equiv 0 \bmod k}}
\sum_{\substack{n\leq x\\\ n\equiv m \bmod k}} 1\\
&=\left(\frac{W}{\varphi(W)}\right)^r \sum_{k\mid W} \lambda_k^+ \left(h(k)x+O(kh(k))\right)
\\
&=
 (1+O(\mathrm e^{-s})) \mathfrak{S}_{f_1,\dots,f_r}(x) x
 \\
&\quad  + O\left(
 \left(\frac{W}{\varphi(W)}\right)^r
D^2\right),
\end{align*} 
where $s=(\log D)/(\log w)$ and we used the trivial bound $h(k) \leq 1$ in the error term.
Taking $D=x^{1/3}$ and using the trivial bound $ \mathfrak{S}_{f_1,\ldots, f_r}(x) \ll (W/\varphi(W))^r \ll (\log x)^r$, 
we deduce that 
\[\sum_{n\leq x}\Lambda_w(f_1(n))\cdots \Lambda_w(f_r(n)) 
\leq \mathfrak{S}_{f_1,\ldots, f_r}(x) x
 + O\left(\frac{x(\log x)^r}{\mathrm e^{s}}+x^{2/3} (\log x)^r\right).\]
Recalling that $w=\exp( \sqrt{ \log x})$, we obtain 
$\mathrm e^{s} =\exp((1/3)\sqrt{\log x } ) \gg_A (\log x)^{A+r},$ whence
\[
\sum_{n\leq x}\Lambda_w(f_1(n))\cdots \Lambda_w(f_r(n)) \leq \mathfrak{S}_{f_1,\dots,f_r}(x) x
 + O\left(x(\log x)^{-A}\right).
 \]  
 A symmetric argument proves the corresponding lower bound, thereby concluding the proof of~\eqref{eq:wtrickbound}. 
\end{proof}

\subsection{Concluding the proof}

We are now ready to prove Theorem~\ref{thm_main_multi} using Lemma~\ref{le:lambdaw}.

\begin{proof}[Proof of Theorem~\ref{thm_main_multi}] Without loss of generality, we may assume that the combinatorial cube $\mathcal{C}$ defined in Definition \ref{def:cube} has  
dimension two and, furthermore, that the  constant coefficient is not fixed to be $0$. 
In particular it follows that $\#\mathcal{C}$ has order of magnitude $H^2$.

By Lemma~\ref{le:lambdaw}, the triangle inequality and Chebyshev's inequality, it suffices to prove
\begin{equation}\label{eq:multpol}
\begin{split}
\sum_{f_1,\dots,f_r\in \mathscr{C}}
&\left|\sum_{n\leq x}
\alpha_n\left(\Lambda(f_1(n))\cdots \Lambda(f_r(n))-\Lambda_{w}(f_1(n))\cdots
\Lambda_{w}(f_r(n))\right)
\right|^2\\
&\qquad\qquad\qquad\qquad\qquad \qquad \ll
\frac{x^2H^{2r}}{(\log x)^{A}},
\end{split}\end{equation}
uniformly for $x\in [H^{c},2H^{c}]$, for any 
coefficients $|\alpha_n|\leq 1$. In our application we shall only be interested in the sequence $\alpha_n=1$, but  the proof requires us to allow general coefficients. 

We first claim that it suffices to prove \eqref{eq:multpol} for $r=1$, that is, for any coefficients $|\alpha_n|\leq 1$, we have
\begin{equation}\label{eqq21}
\sum_{f\in \mathscr{C}}\left|\sum_{n\leq x}
\alpha_n(\Lambda-\Lambda_{w})(f(n))\right|^2\ll \frac{x^2H^2}{(\log x)^{A}},
\end{equation}
uniformly for $x\in [H^{c},2H^{c}]$. To check this, we write  
$S_r(x,H;\alpha_n)$ for the sum in~\eqref{eq:multpol} that is to be estimated, and we argue by induction on $r$. 
When $r=1$ the desired estimate~\eqref{eq:multpol} is just~\eqref{eqq21}.
Suppose now that $r>1$ and write
$\Lambda(f_r(n))=\Lambda_w(f_r(n))+(\Lambda-\Lambda_w)(f_r(n))$.
Then, on applying the triangle inequality, it follows that $S_r(x,H;\alpha_n)$ is at most
$$
(\log x)^2 \sum_{f_r\in \mathscr{C}} 
 S_{r-1}(x,H;\beta_n)
+(\log x)^{2(r-1)}\sum_{f_1,\dots,f_{r-1}\in \mathscr{C}} S_1(x,H;\gamma_n),
$$
where 
$$
\beta_n = \alpha_n\Lambda_w(f_{r}(n))/(\log x),\quad 
\gamma_n=\alpha_n\Lambda(f_1(n))\cdots \Lambda(f_{r-1}(n))/(\log x)^{r-1}.
$$
The desired bound now follows from the induction hypothesis,  on adjusting the choice of $A$.

It remains to prove~\eqref{eqq21}.
Suppose  $\mathscr{C}$ has the $k$th and $\ell$th coordinates as its two free coordinates, with $k<\ell$. It then suffices to prove that 
\begin{equation}
\sum_{|a|,|b|\leq H}\left|\sum_{n\leq x}\alpha_n(\Lambda-\Lambda_w)(an^k+bn^{\ell}+g(n))\right|^2\ll \frac{x^2H^2}{(\log x)^{A}},\label{eqq24}
\end{equation}
uniformly for 
sequences $|\alpha_n|\leq 1$, and for 
polynomials $g\in \mathbb{Z}[t]$ of degree $\leq d$ having coefficients in $[-H,H]$, and with $g(0)\neq 0$. 
The condition $g(0)\neq 0$ can indeed be imposed, since if $k=0$, we can change the value of $g(0)$ by $O(1)$ without changing the validity of~\eqref{eqq24}. Alternatively, if $k\geq 1$ and $g(0)=0$, then $\mathscr{C}$ consists of polynomials whose constant coefficient is fixed to be $0$, a case that was excluded.

We apply Theorem ~\ref{prop_general} to the 
 function $F(n)=(\Lambda(n)-\Lambda_w(n))/(\log x)$, 
which is $1$-bounded.
The desired conclusion now follows, provided that the hypotheses of that theorem hold. Assumption (1) clearly holds. For assumption (2), we apply Lemma~\ref{lemiwakow} to obtain that, for any $h\geq x^{2/3}$ and $1\leq a\leq q\leq x^{1/3}$,
\begin{align*}
\sum_{\substack{x\leq n\leq x+h\\n\equiv a\bmod q}} \Lambda_w(n)&=(1+O((\log x)^{-100Ad}))\mathbf 1_{\gcd(a,q,W)=1}\frac{W}{\varphi(W)}\frac{h}{q}\prod_{\substack{p\mid W\\p\nmid q}}\left(1-\frac{1}{p}\right)\\
&=(1+O((\log x)^{-100Ad}))\mathbf 1_{\gcd(a,q,W)=1}\frac{h}{\varphi(q)}.
\end{align*}
 Note also that $\gcd(a,q,W)=\gcd(a,q)$ unless $\gcd(a,q)$ has a prime factor $>w$, in which case $\gcd(a,q)>w=\exp(\sqrt{\log x})$. 

Let $c_0=5/36$. 
We now conclude from  Proposition~\ref{prop_PNTAPs'} that
\begin{align*}
\left|\sum_{\substack{x\leq n\leq x+h\\n\equiv a\bmod q}}(\Lambda(n)-\Lambda_{w}(n))\right|\ll \frac{h}{q(\log x)^{A}}    
\end{align*}
provided that  $x\geq h\geq x^{1-c_0+\varepsilon}$, $1\leq a\leq q\leq x^{c_0}$,  $\gcd(a,q)\leq\exp(\sqrt{\log x})$ and $q$ is not a multiple of any element of the set $\mathcal{Q}$ present in Proposition~\ref{prop_PNTAPs'}. Assumption~(2) of Theorem 
~\ref{prop_general} directly follows, provided that $(2Hx^d)^{1-c_0}\leq H^{1-\varepsilon}$ for some small $\varepsilon>0$, which indeed holds by the assumption that $x\leq 2H^{c}$ for some fixed $c<5/(31d)$. Now~\eqref{eqq21} follows.
\end{proof}

\section{Polynomials and norm forms modulo prime powers}\label{sec:polynomial}

\subsection{Polynomials in one variable}

To begin with, let $f\in \ZZ[t]$ be a polynomial of degree $d$ and let 
\begin{equation}\label{eq:lambda}
\lambda_f(q)=\#\left\{u\in \ZZ/q\ZZ: f(u)\equiv 0\bmod{q}\right\},
\end{equation}
for any integer $q\geq 1$. 
The {\em content} of $f$ is defined to be the greatest common divisor of the coefficients of $f$.
We shall need some upper bounds for $\lambda_f(p^k)$ for prime power moduli. 

\begin{lemma}\label{lem:stewart}
Let $p$ be a prime and let $k\in \NN$. Let $f\in \mathbb{Z}[t]$ be a polynomial of degree $d\geq 1$. Assume that $f$ is separable  and let $\sigma$ be the $p$-adic valuation of the  content of $f$. 
If $\sigma\geq k$ then $\lambda_f(p^k)=p^k$. If $\sigma<k$ then
$$
\lambda_f(p^k)\leq d \min\left\{p^{k(1-\frac{1}{d})+\frac{\sigma}{d}},  p^{k-1}\right\}.
$$
Additionally, if $p$ does not divide the discriminant of $f$, then
\begin{align*}
\lambda_f(p^k)\leq d.     
\end{align*}
\end{lemma}

\begin{proof}
The final bound follows  from an application of Lagrange's theorem and Hensel's lemma.
For the remaining bounds, we  first proceed under the assumption   $\sigma=0$, 
so that the content  is coprime to $p$. But then the first bound  follows from Corollary 2 and Equation (44) of Stewart~\cite{stewart} and  the second bound is a consequence of Lagrange's theorem. Now suppose that 
$\sigma\geq 1$. If $\sigma\geq k$ then $\lambda_f(p^k)=p^k$. If $\sigma<k$ then $\lambda_f(p^k)=p^{\sigma}\lambda_{g}(p^{k-\sigma})$, where $g=p^{-\sigma}f$ has content coprime to $p$, and we can apply our existing bounds. 
\end{proof}

Next, let $\tau$ denote the  divisor function.
We will make frequent use of a bound on the   average size of 
$\tau^C$ at polynomial arguments, for a given constant $C>0.$ In principle, it should be possible to deduce the following result from work of Nair and Tenenbaum \cite{NT}. However, given that 
extra care would need to be taken to get sufficient uniformity in the polynomial $f$, as well as  to 
handle any fixed prime divisors of $f$,  we have chosen to present a self-contained proof.

\begin{lemma}\label{lem:henriot}
Let $\delta>0$ and $A,C>0$. Let $f\in \ZZ[t]$ be a 
separable 
polynomial of degree $d\geq 1$. Let $\|f\|$ be the maximum modulus of the coefficients of $f$ and assume that the content of $f$ is at most $(\log x)^A$. Then there exist a positive constant $K$, depending only on $\delta$, $d$ and $C$, such that for $x\geq \|f\|^{\delta}$ we have
$$
\sum_{\substack{n\leq x}} \tau(|f(n)|)^C \ll x (\log x)^{K}.
$$
\end{lemma}

\begin{proof} 
Let $h$ be the content of $f$, so that $f(t)=hg(t)$, where $g\in \ZZ[t]$ has content $1$.
Then  $\tau(|f(n)|)^C \leq  \tau(h)^C\tau(|g(n)|)^C \ll (\log x) \tau(|g(n)|)^C$, 
since $h\leq (\log x)^A$, 
on applying the trivial bound for the divisor function. 
Landreau's inequality~\cite{landreau} states that for any $C>0$ and $k\in \NN$, we have 
\begin{align*} 
\tau(n)^{C}\leq k^{(k-1)C} \sum_{\substack{m\mid n\\m\leq n^{1/k}}}\tau(m)^{kC}.    
\end{align*}
Applying this together with the assumption that $\|g\|\leq \|f\|\leq x^{1/\delta}$, we have
\begin{align*}
\tau(|g(n)|)^C\ll  \sum_{\substack{m\mid g(n)\\m\leq x^{1/2}}} \tau(m)^{C'}  
\end{align*}
for some constant $C'>0$ depending only on $\delta$, $d$ and $C$. Hence, for any $\ve>0$ we have 
\begin{align*}
 \sum_{\substack{n\leq x}} \tau(|f(n)|)^C &\ll \log x\sum_{m\leq x^{1/2}}\tau(m)^{C'} \sum_{\substack{n\leq x\\g(n)\equiv 0\bmod m}}1\nonumber\\
 &=x\log x \sum_{m\leq x^{1/2}}\frac{\tau(m)^{C'}\lambda_g(m)}{m}+O(x^{1/2+\ve}), 
\end{align*}
by the divisor bound and 
where $\lambda_g(m)$ is given by
\eqref{eq:lambda}.

We may write
$$
\sum_{m\leq x^{1/2}}\frac{\tau(m)^{C'}\lambda_g(m)}{m}
 \leq \sum_{\ell \mid D^{\infty}}\frac{\tau(\ell)^{C'}\lambda_g(\ell)}{\ell}\sum_{\substack{m'\leq x^{1/2}/\ell\\\gcd(m',D)=1}}\frac{\tau(m')^{C'}\lambda_g(m')}{m'},
$$
where $D$ is the 
non-zero discriminant of $g$. 
Lemma~\ref{lem:stewart} implies that 
$$
\lambda_g(m')\leq d^{\omega(m')}\leq \tau(m')^d, 
$$
when $\gcd(m',D)=1$.
We conclude that 
\begin{align*}
\sum_{m\leq x^{1/2}}\frac{\tau(m)^{C'}\lambda_g(m)}{m}
&\ll (\log x)^{O_{C',d}(1)} \sum_{\ell \mid D^{\infty}}\frac{\tau(\ell)^{C'}\lambda_g(\ell)}{\ell}\\
 &=(\log x)^{O_{C',d}(1)} \prod_{p\mid D}\left(1+\sum_{j\geq 1}\frac{\tau(p^{j})^{C'}\lambda_g(p^j)}{p^j}\right)\\
 &\ll (\log x)^{O_{C',d}(1)}\prod_{p\mid D}\left(1+O_{C',d}\left(\frac{1}{p}\right)\right),
\end{align*}
by Lemma~\ref{lem:stewart}. 
Since $|D|\ll \|g\|^{d(d-1)}\leq x^{d(d-1)/\delta}$, this leads to the  claimed bound.
\end{proof}

\subsection{Norm forms modulo prime powers}\label{sec:local}

In this section we summarise some facts about the function
\begin{equation}\label{eq:def_rho}
\rho(p^k,a)=\#\left\{\x\in (\ZZ/p^k\ZZ)^e : \nf_K(\x)\equiv a\bmod{p^k}\right\},
\end{equation}
for $a\in \ZZ$ and  a prime power $p^k$, 
a detailed study of which has been  carried out in ~\cite[\S 4]{nf}.
Intimately related to $\rho(p^k,a)$ are the coefficients of the  Dedekind zeta function associated to $K$.
Let  $\n\fa=\#\fo_K/\fa$ be the  norm of an integral ideal $\fa\subset\fo_K$.
Then we  define
$$
\zeta_K(s)=\sum_{\substack{\fa\subset \fo_K\\\fa\neq (0)}} \frac{1}{(\n \fa)^s} =
\sum_{m=1}^\infty \frac{r_K(m)}{m^{s}},
$$
for $s\in \CC$ such that $\real(s)>1$, with 
$$
r_K(m)=\#\{\fa \subset \fo_K: \n \fa=m\}.
$$
This function  is multiplicative and 
the Dedekind zeta function admits a meromorphic continuation to all of $\CC$ with a simple 
pole at $s=1$. 
To describe its behaviour at prime powers, let $p$ be any rational prime
and recall that the principal ideal $(p)$ factorises into a product of
prime ideals in $\fo_K$. That is,
\begin{equation}\label{eq:factor}
 (p)=\fp_1^{e_1}\cdots \fp_{r}^{e_r},
 \end{equation}
where $e_i=e_{\fp_i}(p), r=r(p)\in \NN$ and 
each $\fp_i \subset \fo_K$ is a prime ideal satisfying $\n\fp_i=p^{f_i}$,
for some $f_i=f_{\fp_i}(p) \in \NN$. 
The prime $p$ ramifies if and only if $p$ divides the discriminant $D_K$ of $K$. 
Hence $e_1=\cdots=e_r=1$ unless $p\mid D_K$.
Moreover, on taking norms in~\eqref{eq:factor}, we have
$\sum_{i=1}^re_if_i=\deg(K/\mathbb{Q})=e.$
Thus
$$
r_K(p^k)=\#\left\{\fp_1^{m_1}\cdots \fp_r^{m_r}\subset \fo_K:
f_1m_1+\cdots+f_r m_r=k\right\},
$$
for any $k\in \NN$. In particular, we have
\begin{equation}\label{eq:moon}
r_K(p^k)\leq (1+k)^e.
\end{equation}

We begin by studying the function~\eqref{eq:def_rho} when $k=1$.

\begin{lemma}\label{lem:io}
Let $p$ be a prime and let  $a\in \ZZ$, with  $p\nmid a$. 
Then 
$$
\rho(p,a)\leq 
p^{e-1}\left(1-\frac{1}{p}\right)^{-1}\prod_{\fp\mid p} \left(1-\frac{1}{\n \fp}\right),
$$
with equality if $p\nmid D_K$.
%\item[(iii)]
%If we further assume that $p\nmid D_K$, then
%$$
%\rho(p,a)=p^{e-1}\left(1-\frac{1}{p}\right)^{-1}\prod_{\fp\mid p} \left(1-\frac{1}{\n \fp}\right).
%$$
%\end{itemize}
\end{lemma}

\begin{proof}
We shall prove this result using basic properties of finite fields. 
Let $\FF_{p^\ell}$ be the finite field with $p^\ell$ elements and 
$\FF_p$-basis $\{\tau_1,\dots,\tau_\ell\}$.
Write 
$$N_\ell(\y)=N_{\FF_{p^\ell}/\FF_p}(y_1\tau_1+\cdots +y_\ell\tau_\ell)$$ 
for the associated norm form.
Then, for any $a\in \FF_p$, we claim that
\begin{equation}\label{eq:ff}
\#\{\y\in \FF_p^\ell:  N_\ell(\y)=a\}=\begin{cases}
1 & \mbox{if $a=0$,}\\
\frac{p^\ell-1}{p-1} & \mbox{if $a\neq 0$.}
\end{cases}
\end{equation}
The case $a=0$ follows since $N_{\ell}(\y)=0$ if and
only if $y_1=\cdots=y_\ell=0$. To handle the case $a\neq 0$ we first prove that the norm form $N_\ell$ assumes every non-zero value of $\FF_p$. 
To see this, let $\alpha$ be a generator of $\FF_{p^\ell}^*$,
so that  $\alpha$ has order $p^{\ell}-1$.
The Galois group of $\FF_{p^\ell}/\FF_p$ is cyclic and generated by
$x\mapsto x^p$.
Thus
$N_{\FF_{p^\ell}/\FF_p}(\alpha) = \alpha^{(p^{\ell}-1)/(p-1)}$. 
This  is a generator of $\FF_p^*$ and therefore establishes the
claim.
Since multiplication by a fixed element in $\FF_{p^{\ell}}$ is injective,
this allows us to construct a bijection between any two sets
$\{\y\in \FF_p^\ell:  N_\ell(\y)=b\}$ and
$\{\y\in \FF_p^\ell:  N_\ell(\y)=b'\}$ for $b,b' \in \FF_p^*$.
Thus, we conclude the proof of~\eqref{eq:ff} for $a\neq 0$ by noting that 
\begin{align*} 
p^\ell-1 
&=\sum_{b\in \FF_p^*}
\#\{\y\in \FF_p^\ell:  N_\ell(\y)=b\} 
=(p-1)
\#\{\y\in \FF_p^\ell:  N_\ell(\y)=a\}.
\end{align*}

Suppose that  $p$ factorises in $\fo_K$ as 
\eqref{eq:factor}. Then the reduction modulo $p$ of the minimal  
polynomial has the same factorisation modulo $p$, by the Dedekind--Kummer  
theorem.  This implies that there is a  
non-singular $\FF_p$-linear transformation 
$\x\mapsto (\y^{(1)},\cdots, \y^{(r)})$, 
with $\y^{(i)}=(y_1^{(i)},\dots, y_{f_i}^{(i)})$, such that 
$$
\nf(\x)\equiv \prod_{i=1}^r N_{f_i}(\y^{(i)})^{e_i} \bmod{p}.
$$
For any $n\in \NN$, 
let $G_n$ denote the set of $n$th powers in $\FF_p^*$.
Then 
\begin{align}\label{eq:gn}
\#G_n=\frac{p-1}{\gcd(n,p-1)}.
\end{align}
Moreover, there are precisely $\gcd(n,p-1)$ elements of $\FF_p$ with order dividing $n$.
Since $p\nmid a$, it follows that
\begin{align*}
\rho(p,a)
&= \#\{  \x\in \FF_p^e: \nf(\x)=a\}\\
&= \sum_{\substack{(a_1,\dots,a_r)\in  G_{e_1}\times \cdots \times G_{e_r}\\
a_1\cdots a_r=a\\
}}
\prod_{i=1}^r 
 \#\{  \y\in \FF_p^{f_i}: N_{f_i}(\y)^{e_i}=a_i\}.
\end{align*}
On appealing to 
~\eqref{eq:ff}, we deduce that 
 \begin{equation}\label{eq:io}
\rho(p,a)
= M(a)
\prod_{i=1}^r 
\gcd(e_i,p-1)
\left(\frac{p^{f_i}-1}{p-1}\right),
\end{equation}
where $M(a)$ is the number of 
$(a_1,\dots,a_r)\in  G_{e_1}\times \cdots \times G_{e_r}$ such that 
$a_1\cdots a_r=a.$
Recalling~\eqref{eq:gn}, an argument involving 
primitive roots readily reveals that
$$
M(a)=
\begin{cases} 
0 &  \text{ if $a\not\in G_d$,}\\
 \frac{\gcd(d,p-1)}{p-1}\prod_{1\leq i\leq r} \frac{p-1}{\gcd(e_i,p-1)} &
 \text{ if $a\in G_d$,}
\end{cases}
$$
where $d=\gcd(e_1,\dots,e_r)$. Once inserted into~\eqref{eq:io}, this 
establishes that 
$$
\rho(p,a) =\begin{cases}
0 & \text{ if $a\not\in G_{\gcd(e_1,\dots,e_r)}$,}\\
\frac{\gcd(e_1,\dots,e_r,p-1)}{p-1}(p^{f_1}-1)\cdots(p^{f_r}-1) & \text{ if $a\in G_{\gcd(e_1,\dots,e_r)}$.}
\end{cases}
$$

If $p\nmid D_K$ then $e_1=\cdots=e_r=1$ and it follows that 
 \begin{align*}
\rho(p,a)
&= 
\frac{(p^{f_1}-1)\cdots (p^{f_r}-1)}{(p-1)}= p^{e-1} \left(1-\frac{1}{p}\right)^{-1}
(1-p^{-f_1})\cdots (1-p^{-f_r}),
\end{align*}
as desired.  On the other hand, if $p\mid D_K$ then the upper bound follows on noting that by the above 
$$
(p-1)\rho(p,a)\leq p^{f_1+\cdots+f_r}\min\{e_1,\dots,e_r\}\leq p^{e_1f_1+\cdots+e_rf_r}=p^e.
$$
This completes the proof of the lemma.
\end{proof}

We claim that 
\begin{equation}\label{eq:clip}
\begin{split}
\left(1-\frac{1}{p}\right)^{-1}\prod_{\fp\mid p} \left(1-\frac{1}{\n \fp}\right)
&=1-\frac{r_K(p)-1}{p}+O\left(\frac{1}{p^2}\right)\\
&\geq
1-\frac{e-1}{p}+O\left(\frac{1}{p^2}\right)
\end{split}
\end{equation}
in the statement of the previous lemma.  To see this, on  multiplying out $\prod_{\fp\mid p }(1-\n \fp^{-1})$, we get a summand 
$-p^{-1}$ precisely when $\fp$ has norm $p$, which happens $r_K(p)\leq e$ times. All other terms contribute $O(p^{-2})$, by~\eqref{eq:moon}.
Moreover, we get the term $p^{-1}$ coming from expanding $(1-1/p)^{-1}$, and so the claim follows.

We  may now record some facts about $\rho(p^k,a)$ for an arbitrary prime power $p^k$ and arbitrary $a\in \ZZ$.  The  following result draws on  the proof of~\cite[Lemma~4.2]{nf}.

\begin{lemma}\label{lem:LM}
Let $p$ be a prime, 
let $a\in \ZZ$ and let  $k\in \NN$.
 Suppose that $\alpha=v_p(a)$.
Then  the following hold:
\begin{itemize}
\item[(i)] If $\alpha\geq k$ then 
$$
\rho(p^k,a)=p^{ke} \left(\sum_{j\geq k}\frac{r_K(p^j)}{p^j}\right)
\prod_{\fp\mid p} \left(1-\frac{1}{\n \fp}\right).
$$
\item[(ii)] If $v_p(e)=\ve\geq 0$ and $\alpha<k$  then there exists $b\in \ZZ$ coprime to $p$ such that 
$$
\rho(p^k,a) = \frac{r_K(p^\alpha) \rho(p^{\min(k-\alpha,2\ve+1)},b) p^{k(e-1)} }{p^{\min(k-\alpha,2\ve+1)(e-1)}}.$$
\item[(iii)] 
If $p\nmid eD_K$ and $\alpha<k$ then
$$
\rho(p^k,a)=r_K(p^{\alpha})p^{k(e-1)}\left(1-\frac{1}{p}\right)^{-1}\prod_{\fp\mid p} \left(1-\frac{1}{\n \fp}\right).
$$
 \end{itemize}
\end{lemma}

\begin{proof}
Suppose that $\alpha\geq k$. 
Then 
 $\rho(p^k,a)=\rho(p^k,0)$ is equal to the number of 
$\alpha\in \fo_K$, modulo $p^k$, for which $p^k\mid N_{K/\QQ}(\alpha)$.
For some $j\geq k$, there is a unique ideal factorisation 
$(\alpha)=\fq\fb$ where $\n\fq=p^j$ and where $p$ is coprime to $\fb$.
The number of such ideals $\fq$ is precisely $r_K(p^j)$.
Associated to each $\fq$ it remains to count the number of $\alpha \in \fo_K$, modulo $p^k$, 
with $\fq\mid (\alpha)$  and for which $\fq^{-1}(\alpha)$ is coprime to the ideal $(p)$.
Given any integral ideal $\fc$, let $\Lambda_{\fc}$ be the set of 
$\alpha \in \fc$, modulo $p^k$.  
Assuming
that  $\n\fc=p^\ell$ for some $\ell\geq k$, it follows that 
$$
\#\Lambda_{\fc}=[\fc: (p^k)]=\frac{[\fo_K: (p^k)]}{[\fo_K:\fc]}=
\frac{p^{ke}}{\n\fc},
$$
since $p^k\in \fc$.  Hence
\begin{align*}
\rho(p^k,a)
&=\sum_{j\geq k} r_K(p^j) 
\sum_{\fb \mid (p)} \mu_K(\fb) \#\Lambda_{\fb\fq} \\
&= p^{ke}\sum_{j\geq k} \frac{r_K(p^j)}{p^j}  
\sum_{\fb \mid (p)} \frac{\mu_K(\fb)}{\n\fb},
\end{align*}
where $\mu_K$ is the number field analogue of the M\"obius function.
Part (i) then follows on writing the inner sum as an Euler product.

Suppose now that $\alpha<k$. 
It follows from ~\cite[Eq.~(4.6)]{nf} 
that 
\begin{equation}\label{eq:pluto}
\rho(p^k,a)=  r_K(p^\alpha)p^{\alpha(e-1)} 
\rho(p^{k-\alpha},b),
\end{equation}
where $b=p^{-\alpha}a\n \mathfrak{t}$, for a certain ideal $\mathfrak{t}$ that is coprime to $p$. 
Thus we must study 
$\rho(p^{m},b)$ when $p\nmid b$ and $m\geq 1$. It follows from Hensel's lemma, in the 
form~\cite[Lemma~3.4]{nf} with 
$G=\nf$ and $\ell=0$, that 
$$
\frac{\rho(p^m,b)}{p^{m(e-1)}}=\frac{\rho(p^{2v_p(e)+1},b)}{p^{(2v_p(e)+1)(e-1)}},
$$
for any $m\geq 2v_p(e)+1$.
Part (ii) now follows on combining this with~\eqref{eq:pluto}. Finally, in  part (iii) we have $p\nmid e$. Hence
we may apply part (ii) with $\ve=0$, which then allows us to apply Lemma~\ref{lem:io} and easily leads to equality in  part (iii).
\end{proof}

We will also need a general upper bound for $\rho(p^k,a)$. This is achieved in the following result. 

\begin{lemma}\label{lem:sun'}
Let $p$ be a prime, 
let $a\in \ZZ$ and let  $k\in \NN$.
 Suppose that $\alpha=v_p(a)$ and $\ve=v_p(e)$. Then the following hold:
 \begin{itemize}
 \item[(i)]
 If $\alpha\geq k$ then  
$
\rho(p^k,a)\leq (1+k)^e p^{k(e-1)}.
$
\item[(ii)]
 If $\alpha< k$ then  
$$
\rho(p^k,a)\leq 
r_K(p^\alpha) \varphi^*(p)
p^{2\ve +k(e-1)},
$$
where $\varphi^*(n)=n/\varphi(n)$.
 \end{itemize}
\end{lemma}

\begin{proof}
Part (i) follows from  ~\cite[Eq.~(4.5)]{nf}. 
Turning to part (ii), we suppose that $\alpha<k$.
Noting that $\rho(p^\ell,a)\leq p^{(\ell-1) e}\rho(p,a)$, for any $\ell\in \NN$, from part (ii) of Lemma~\ref{lem:LM}  it follows that
$$
\rho(p^k,a) \leq r_K(p^\alpha) \cdot 
p^{\min(k-\alpha,2\ve+1)-e} \cdot
\rho(p,b) p^{k(e-1)},
$$
for an appropriate $b\in \ZZ$ which is coprime to $p$.
But then Lemma~\ref{lem:io} yields
 $
\rho(p^k,a) \leq 
r_K(p^\alpha) \cdot 
p^{2\ve +k(e-1)}\varphi^*(p),
$
as desired. 
\end{proof}

The following result follows as 
an easy consequence of the previous results. 

\begin{lemma}\label{lem:sun}
Let $q\in \NN$, let $a\in \ZZ$ and let $\gamma(q,a)=
q^{-(e-1)} \rho(q,a).$
Then 
$$
\gamma(q,a)\leq 
e^2 \varphi^*(q) \tau(\gcd(a,q))^e,
$$ 
where  $\varphi^*(q)=q/\varphi(q)$.
\end{lemma}

\begin{proof}
It follows from the Chinese remainder theorem that 
$$
\gamma(q,a)=
\prod_{\substack{p^k\|q}} 
\frac{\rho(p^k,a)}{p^{k(e-1)}}.
$$
Let $\alpha=v_p(a)$. If 
 $k\leq \alpha$ then part  (i) of Lemma~\ref{lem:sun'} yields
$$
 \frac{\rho(p^k,a)}{p^{k(e-1)}} \leq (1+k)^e=\tau(p^k)^e.
$$
Now suppose that $k>\alpha$.  Then 
part (ii) of Lemma~\ref{lem:sun'} gives
$$
 \frac{\rho(p^k,a)}{p^{k(e-1)}} \leq  p^{2v_p(e)}
  r_K(p^{\alpha})
\varphi^*(p).
$$
But $r_K(p^\alpha)\leq (1+\alpha)^e=\tau(p^\alpha)^e$, by~\eqref{eq:moon}.
The statement of the lemma easily follows.
\end{proof}

\section{Norm forms: a localised counting function}\label{sec:normforms}

\subsection{Preliminary steps}

We shall approach Theorem~\ref{t:1} by studying the average size of a counting function that measures the density of integer solutions to 
\eqref{eq:n-f}. 
In order to describe 
this counting function, we first introduce a function $R_K$ 
that assigns to each 
integer its number of representations by the norm form $\nf_K$. 
We may henceforth assume\footnote{Indeed, if $1=\sum_{i=1}^{e}c_i\omega_i$ with $c_i\in \mathbb{Z}$, then $g=\gcd(c_1,\ldots, c_e)=1$, since $1/g$ is an algebraic integer. Now, there is a matrix $M$ with determinant $\pm 1$ and first row $(c_1,\ldots, c_{e})$. Hence $M(\omega_1,\ldots, \omega_{e})^{\textnormal{T}}$ is a basis of $\mathfrak{o}_K$ with first element equal to $1$.} 
without loss of generality that $\omega_1=1$ in 
the definition \eqref{eq:norms} of the norm form. 

For a small parameter  $\kappa>0$, that will depend only on the field $K$, we define the region
$$
\mathcal{B}=
\{\x \in \RR^e: |\x-2^{-1/e}\x^{(\RR)}|<\kappa \},
$$
where
$\x^{(\RR)}=(1,0,\dots,0)\in \mathbb{R}^e$. In particular $\nf_K(\x^{(\RR)})=N_{K/\QQ}(\omega_1)=1$.
We write $B\mathcal{B}=\{B\x:\x\in \mathcal{B}\}$ for the dilation by a parameter $B\geq 1$.
We shall be interested in the counting function
\begin{equation}\label{eq:define-R}
R_K(n;B)=\#\left\{\x\in \ZZ^e\cap B\mathcal{B}: \nf_K(\x)=n\right\},
\end{equation}
for $n\in \ZZ$.
The parameter $B$ is allowed to vary, but everything else is to be considered fixed. In particular, all implied constants in our work are allowed to depend on the number field $K$.
Finally, we note that  $R_K(n;B)=0$ unless $|n|\ll B^e$.

Our principal concern is with the function
$$
N_{\c}(x)=\sum_{n\leq x}  R_K(f_\c(n);B),
$$
as $x\to \infty$, where $f_\c(t)=c_0t^d+\cdots +c_d$ is a degree $d$ polynomial,
with coefficient vector $\c\in \ZZ^{d+1}$.
We  shall  always assume that 
\begin{equation}\label{eq:suff-small}
x\leq H^{\Delta_{d,e}}, \quad \text{ with }\, \Delta_{d,e}=\frac{1}{2de(e+3)}.
\end{equation}
Simply put,  our goal is then to show that 
$
N_\c(x)\gg x,
$ 
for sufficiently large $x$ and suitable $B$,  for $100\%$ of 
coefficient vectors $\c\in S_d^{\text{loc}}(H)$.
We  shall ultimately take $B$ to be such that $B^e$ has the same order of magnitude as $f_\c(n)$ for typical $n\leq x$ and typical $\c\in S_d(H)$ with $c_0>0$, but we delay  specifying its value until later.
Instead, our plan is to 
define an approximate counting function 
$\hat N_\c(x)$. This will be chosen to  approximate $N_\c(x)$ very well (on average) and, furthermore,  so that 
it is  $\gg x$ (on average).
The approximate counting function $\hat N_\c(x)$ 
relies on the construction of a localised function $\hat R_K(n;B)$
that is designed to   approximate $R_K(n;B)$. 

Let
\begin{equation}\label{eq:k-x}
k(x)=1000dA \lceil \log\log x\rceil
\end{equation}
and let 
\begin{equation}\label{eq:def_W}
W=\prod_{p\leq w} p^{k(x)}\quad \text{ for }\quad 
w=\exp(\sqrt{\log x}).
\end{equation}
In particular, we have $\log W =(1+o(1)) wk(x)$ by the prime number theorem, and so
$W$ exceeds any power of $x$.
(In contrast to the choice of $W$ in Section~\ref{sec:bateman-horn}, it will be important to allow prime powers with exponents that   grow with $x$.)

For any $a\in \ZZ$ and  $q\in \NN$, recall the definition
\begin{equation}\label{eq:def-gamma}
\gamma(q,a)=
\frac{1}{q^{e-1}} 
\#\left\{ \mathbf{s}\in (\ZZ/q\ZZ)^e : 
\nf_K(\mathbf{s})\equiv a \bmod{q}
\right\}
\end{equation}
from Lemma \ref{lem:sun}.
This  function is multiplicative in $q$.
With this notation, we put
\begin{equation}\label{eq:def-hatR}
\hat R_K(n;B)= \gamma(W,n) \cdot \omega(n;B),
\end{equation}
where 
$\omega(n;B)$ is an  archimedean factor that we proceed to define. 
(Note that   $\gamma(W,n)$
 converges to the usual product of non-archimedean densities  as $x\to \infty$.)

Consider the map $\CC^e\to \CC$, given by  
$\mathbf{x}\mapsto \nf_K(\mathbf{x}).$ 
This is non-singular at any point for which 
$\nf_K(\mathbf{x})\neq 0$, by which we mean that 
$\nabla \nf_K(\mathbf{x})\neq \0$  if $\nf_K(\mathbf{x})\neq 0$.
Noting that 
$$
\frac{\partial \nf_K}{\partial x_1}(\x^{(\RR)})=eN_{K/\QQ}(\omega_1)=e,
$$
it follows from  continuity that 
\begin{equation}\label{eq:wave}
\frac{\partial \nf_K}{\partial x_1}(\x)\gg B^{e-1}
\end{equation}
throughout $B\mathcal{B}$, if $\kappa$ is small enough.
Suppose now that $y=\nf_K(\x)$ with $\x\in B\mathcal{B}$. Then, by the implicit function theorem, we can express $x_1$ in terms of $\x'=(x_2,\dots,x_e)$ and $y$ as $x_1=x_1(\x',y)$.  
We now define 
\begin{equation}\label{eq:define_omega}
\omega(y)=\omega(y;B)=
\int_{\mathcal{B}(y)} \left(\frac{\partial \nf_K}{\partial x_1}(x_1(\x',y),\x')\right)^{-1} \d\x',
\end{equation}
where 
$
\mathcal{B}(y)=\left\{\x'\in \RR^{e-1}: (x_1(\x',y),\x')\in B\mathcal{B}\right\}.
$
This expression represents the real density of 
solutions to the equation $y=\nf_K(\x)$ and 
should be familiar to practitioners of the circle method as an 
 incarnation of the 
{\em singular integral} using the 
Leray measure.

The following result 
is a 1-dimensional analogue of~\cite[Lemma 9]{gafa}, 
collecting together some basic facts about $\omega(y)$.

\begin{lemma}\label{lem:props-om}
The
function $\omega(y)=\omega(y;B)$ 
in~\eqref{eq:define_omega} is non-negative,  continuously differentiable and satisfies the following properties:
\begin{enumerate}
\item[(i)] we have 
$$
\omega(y+h)-\omega(y)\ll B^{-e}|h|,
$$ 
for all $y,h\in \RR$;
\item[(ii)]
we have 
$$
\int_I\omega(y)\d y=
\vol\left\{
\x\in \mathcal{B}: \nf_K(\x)\in I
\right\},
$$
for any interval $I\subset \RR$;
\item[(iii)]
$\omega$ 
 is supported on an interval of length $O(B^e)$ centred on the origin
 and
satisfies $\omega(y)\ll 1$ throughout its support;
\item[(iv)]
there exits an interval of length  $\gg B^{e}$ around the point 
$\frac{1}{2}B^{e}$  on which we have $\omega(y)\gg 1$.
\end{enumerate}
\end{lemma}

\begin{proof}
By~\eqref{eq:wave} and~\eqref{eq:define_omega}, we have $\omega(y)\geq 0$. The map $\omega:\RR\to [0,\infty)$ is 
continuously differentiable, since the function $x_1(\x',y)$ is 
continuously differentiable with respect to $\x'$ and $y$.
Part (ii) follows from making the substitution $y=\nf_K(\x)$ in 
$\vol\{
\x\in \mathcal{B}: \nf_K(\x)\in I\}$.
It is clear  that  $\omega$ is supported on the set of values
of $\nf_K(\x)$ as $\x$ runs over $\mathcal{B}$, whence 
$\omega$ is supported on an interval of length $O(B^e)$ centred on the origin, as claimed in 
part (iii). Moreover, in view of 
\eqref{eq:wave}, we clearly have 
$$\omega(y)\ll B^{e-1}\cdot (B^{e-1})^{-1}\ll 1
$$ throughout this interval.
Part (iv) follows from the implicit function theorem and the observation that 
$ \nf_{K}(2^{-1/e}\x^{(\RR)})=\frac{1}{2}$.

It remains to prove part (i),
for which we need to investigate $\omega(y+h)-\omega(y)$.  
Define 
$$
J(\x)=
\frac{\partial \nf_K}{\partial x_1}(\x),
$$ 
a form of degree $e-1$. 
Suppose now that $y=\nf_K(\x)$ with $\x\in \mathcal{B}$. We saw, via the implicit function theorem, that this  yields a function $x_1(\x',y)$ expressing $x_1$ in terms of 
$\x'$ and $y$. We clearly have 
$
\partial y/\partial x_1=J(\x),
$ whence
\begin{equation}\label{eq:horse}
\frac{\partial x_1(\x',y)}{\partial y}\ll J(\x)^{-1}\ll B^{1-e},
\end{equation}
by~\eqref{eq:wave}.
On $\mathcal{B}(y+h)\cap
\mathcal{B}(y)$ 
we may now conclude that 
\begin{align*}
J(x_1(\x',y+h),\x')^{-1}-&J(x_1(\x',y),\x')^{-1}\\ 
&\ll B^{2-2e}
|J(x_1(\x',y+h),\x')-J(x_1(\x',y),\x')| \\
&\ll B^{2-2e} \cdot B^{e-2}
|x_1(\x',y+h) -x_1(\x',y)|\\
&\ll B^{1-2e}|h|,
\end{align*}
by~\eqref{eq:wave}, 
\eqref{eq:horse}
 and Taylor expansion.
Since $\vol(\mathcal{B}(y))\ll B^{e-1}$ this shows that 
the contribution to $\omega(y+h)-\omega(y)$ from the set
 $\mathcal{B}(y+h)\cap
\mathcal{B}(y)$ is $O(B^{-e}|h|)$.

The remaining contribution is 
$O(B^{1-e}\vol(\mathcal{B}'))$, 
on invoking the lower bound~\eqref{eq:wave}, 
where
$\mathcal{B}'$ is the set of $\x'\in \RR^{e-1}$ for which precisely one of 
$(x_1(\x',y),\x')$ or $(x_1(\x',y+h),\x')$ lies in $ \mathcal{B}$.
We may assume without loss of generality that 
$(x_1(\x',y),\x')\in \mathcal{B}$ and  $(x_1(\x',y+h),\x')\not\in \mathcal{B}$.
According to~\eqref{eq:horse}, this means that 
$(x_1(\x',y),\x')$ is a distance of $O(B^{1-e}|h|)$ from the boundary of 
$\mathcal{B}$. It follows  that 
$$
\vol(\mathcal{B}') \ll B^{e-2}\cdot B^{1-e}|h|=B^{-1}|h|.
$$
Thus the remaining contribution is $O(B^{-e}|h|)$, which thereby completes the proof of the lemma.
\end{proof}

Bearing~\eqref{eq:def-hatR} in mind, 
we may now define the localised  counting function to be
$$
\hat N_\c(x)=\sum_{n\leq x} \hat R_K(f_\c(n);B),
$$
for $\c\in \ZZ^{d+1}$.
The proof of 
Theorem~\ref{t:1} comes in two steps. First we shall show that $\hat N_\c(x)$ is usually a good approximation to $N_\c(x)$, for suitable ranges of $x$ and $H$. 
The remaining task will then be to show that 
$\hat N_\c(x)$ is rarely smaller than it should be.

\subsection{First moment of the localised counting function}

An important ingredient in our work is an asymptotic formula for the first moment of $\hat R_K(n;B)$.
We set
$$
S(I;q,u)=\sum_{\substack{n\in I \\n\equiv u\bmod q}}\hat R_K(n;B)= \sum_{\substack{n\in I\\n\equiv u\bmod q}}\gamma(W,n) \cdot \omega(n;B),
$$
for any interval $I\subset [1,X]$, with 
 $X\geq 1$,  and any   $u,q\in \NN$ such that $1\leq u\leq q$.
The most obvious way to proceed would be to break the sum into residue classes modulo $W$. However, this ultimately leads to error terms involving powers of $W$. In our application we shall apply this result when $X$ is a power of $x$, rendering such error terms unacceptable.
Instead, we shall 
adopt an alternative  approach that  resembles the fundamental lemma in sieve theory, in order to prove the following result. 

\begin{proposition}\label{prop:1moment} Recall the notation~\eqref{eq:k-x} and~\eqref{eq:def_W}.
Let $X$ satisfy  $B^e\geq X\geq x$ and $W\geq X^2$. Let $I\subset [1,X]$ such that $|I|>X^{9/10}$. 
Let $u,q\in \NN$ such that $1\leq u\leq q\leq X^{1/5}$. Assume that $p^{k(x)}\nmid q$ for any $p\leq w$, and that
 \begin{equation}\label{eq:assume-gcd}
\gcd(u,q)\leq \exp \left(\frac{2\sqrt{\log X}}{\log\log X}\right).
\end{equation}
Then 
\begin{align*}
S(I;q,u)=~& \frac{\gamma(\gcd(q,W),u)}{q}\int_{I} \omega(t;B)\d t  +
 O\left(\frac{|I|}{q}
 \exp\left(-\frac{\log X}{23\sqrt{\log x}} \right)
\right).
\end{align*}
\end{proposition}

It turns out that the very small prime factors of $W$  cause difficulties in our argument and so we reduce  their influence by breaking into residue classes. 
We henceforth let 
$
M_{d,K}
$
be a suitably large positive constant that only depends on $d$ and the number field $K$, and whose value will become apparent during the course of the proof. (We shall certainly want to assume that $M_{d,K}>d$ and $M_{d,K}>eD_K$, for example.)
Let  
\begin{equation}\label{eq:WWW}
W_0=\prod_{\substack{p\leq M_{d,K} }} p^{k(x)} \quad \text{ and } \quad 
W_1=\prod_{\substack{M_{d,K}<p\leq w }}p^{k(x)} .
\end{equation}
Clearly $\log W_0\leq k(x)\sum_{p\leq M_{d,K}}\log p\ll  k(x)\ll \log\log x$, by~\eqref{eq:k-x}.
Hence, since $x\leq X$, we have  
\begin{equation}\label{eq:W0}
W_0\ll (\log x)^{C'}\leq (\log X)^{C'},
\end{equation}
for an appropriate constant $C'>0$ that depends on $e,K$ and the constant $C$ in the definition of $k(x)$.
We have   $\gamma(W,n)=\gamma(W_0,n)\gamma(W_1,n)$, by the Chinese remainder theorem. But then 
\begin{equation}\label{eq:S-to-T}
S(I;q,u)=
\sum_{\substack{\nu \bmod{W_0}\\
\nu\equiv u \bmod{\gcd(W_0,q)}
}} \gamma(W_0,\nu)T(I;\nu),
\end{equation}
where
$$
T(I;\nu)=
\sum_{\substack{n\in I 
\\ n\equiv \nu\bmod{W_0}\\
n\equiv u\bmod q}}\gamma(W_1,n) \cdot \omega(n;B),
$$
and we have observed that $T(I;\nu)=0$  unless $\nu\equiv u \bmod{\gcd(W_0,q)}$.

It will  be convenient to put
$$
P(w)=\prod_{\substack{M_{d,K}<p\leq w}} p.
$$
Moreover, for any prime $p$ and any $k\in \NN$, we set
\begin{equation}\label{eq:alpha-p}
\alpha_p=\left(1-\frac{1}{p}\right)^{-1}\prod_{\fp\mid p} \left(1-\frac{1}{\n \fp}\right)
\end{equation}
and 
\begin{equation}\label{eq:beta-p}
\beta_{p^k}=p^k\sum_{j\geq k} \frac{r_K(p^j)}{p^j}\prod_{\fp\mid p} \left(1-\frac{1}{\n \fp}\right).
\end{equation}
It follows from part (i) of Lemmas~\ref{lem:LM} and~\ref{lem:sun'} that $\beta_{p^k}\leq (1+k)^e$.
It is clear that $\alpha_p$ is positive. 
We can get a better lower  bound from~\eqref{eq:clip}, which shows that 
$\alpha_p\geq \frac{1}{2}$ if $p\mid P(w)$ and $M_{d,K}$ is chosen to be sufficiently large in  the definition
\eqref{eq:WWW}. 
This implies that 
\begin{equation}\label{eq:beta-bound'}
\alpha_p^{-1}\beta_{p^k}
\leq (1+k)^{e+1} \quad \text{ if }\, p\mid P(w),
\end{equation}
since $2(1+k)^e\leq (1+k)^{e+1}$.

Recalling that $M_{d,K}>eD_K$, 
parts (i) and (iii) of Lemma~\ref{lem:LM} yield
\begin{equation}\label{eq:calc}
\begin{split}
\gamma(W_1,n)
&=\prod_{\substack{p\mid P(w)}} \gamma(p^{k(x)},n) \\
&=\prod_{\substack{p\mid P(w)\\
p^{k(x)}\nmid n}}
r_K(p^{v_p(n)}) \alpha_p
\prod_{\substack{p\mid P(w)\\
p^{k(x)}\mid n
}}
\beta_{p^{k(x)}}\\
&=c(w)
\prod_{\substack{p\mid P(w)\\
p^{k(x)}\nmid n}}
r_K(p^{v_p(n)}) 
\prod_{\substack{p\mid P(w)\\
p^{k(x)}\mid n
}}
\alpha_p^{-1}\beta_{p^{k(x)}},
\end{split}
\end{equation}
where 
\begin{equation}\label{eq:c(w)}
c(w)=\prod_{p\mid P(w)} \alpha_p.
\end{equation}
Note that $c(w)\leq \varphi^*(W)$. 

Let $\tilde r_K (n)$ be the multiplicative arithmetic function which on prime powers is defined to be
\begin{equation}\label{eq:r-tilde}
\tilde r_K (p^j) =\begin{cases}
r_K(p^j) &\text{ if $j<k(x)$,}\\
\alpha_p^{-1}\beta_{p^{k(x)}} &\text{ if $j\geq k(x)$.}
\end{cases}
\end{equation}
Then, in summary, we have
\begin{equation}\label{eq:berkley-hotel}
T(I;\nu)=
c(w)
\sum_{\substack{n\in I 
\\ n\equiv \nu\bmod{W_0}\\
n\equiv u\bmod q}} \tilde r_K(n_W) \omega(n;B),
\end{equation}
where
$$
n_W=\prod_{\substack{p^\nu\| n\\ p\mid P(w)}}p^\nu.
$$

It remains to obtain an asymptotic formula for the inner sum 
\begin{equation}\label{eq:Unu}
U(I;\nu)=
\sum_{\substack{n\in I
\\ n\equiv \nu\bmod{W_0}\\
n\equiv u\bmod q}} \tilde r_K(n_W) \omega(n;B).
\end{equation}

This will rely crucially on an estimate
 for the number of $w$-smooth numbers which lie in  a given  interval and a given  residue class.
 
\subsection{Some smooth number estimates} 

To begin with, let  $\Psi(N,w)$ be the number of $w$-smooth numbers $n\leq N$.

\begin{lemma}\label{lem:friable'}
Let $\varepsilon>0$. For $(\log N)^{2}\leq w\leq N$, we have
$$
\Psi(N,w)\ll Ns^{-(1+o_{s\to \infty}(1))s},
$$
where  $s=(\log N)/(\log w)$.
\end{lemma}

\begin{proof}
This follows from 
\cite[Cor.~1.3]{adolf} 
in the survey of Hildebrand and Tenenbaum on smooth numbers.
\end{proof}

Using this we may deduce the following result, which deals with short intervals and also allows for a restriction to a given congruence class.

 \begin{lemma}\label{lem:friable}
Let $r\in \NN$ and let  $a\in \ZZ$. Assume that 
$\sqrt{N-H}\leq H\leq N/2$ and $r\leq \sqrt{H}$. Suppose, furthermore, that 
$$\log N\leq w\leq \min\{\sqrt{H},\exp((\log N)/(\log \log N))\}.
$$ 
Then, 
$$	
\max_{a\bmod r}\#\{N-H\leq n\leq N: n\equiv a\bmod{r}, ~p\mid n \Rightarrow p\leq w\} \ll \frac{H}{r}\cdot s^{-s/5}, 
$$
where $s=(\log N)/(\log w)$.
\end{lemma}

\begin{proof}
Let $\varphi(N,H)$ denote the quantity that is to be estimated.
We shall apply Harper's Smooth Numbers Result 3 in~\cite{harper}, with $y=w$, $X=N-H$, $Z=H$ and $q=r$.
The hypotheses of this result are met if 
 $\log (N-H)\leq w\leq N-H$ and 
$rw\leq H\leq N-H$, all of which  are implied by the hypotheses of the lemma. 
Let $\tilde s=(\log (N-H))/(\log w)$. 
Then, in the light of 
\cite[Eq.~(2.1)]{harper},  we deduce that 
$$
\varphi(N,H) \ll \frac{H}{r(N-H)}\cdot \Psi(N-H,w)\cdot \left(\frac{(N-H)r}{H}\right)^{1-\alpha}\log N,
$$
with 
$$
\alpha=
1-\frac{\log(\tilde s\log(\tilde s+1))}{\log w}+O\left(\frac{1}{\log w}\right)=
1-\frac{(1+o_{\tilde s \to \infty}(1))\log \tilde s}{\log w}.
$$

Moreover, Lemma~\ref{lem:friable'} implies that 
$\Psi(N-H,w)\ll (N-H){\tilde s}^{-(1+o(1))\tilde s}.$ Hence
\begin{align*}
\varphi(N,H) 
&\ll \frac{H}{r}\cdot {\tilde s}^{-(1+o_{\tilde s\to \infty}(1))\tilde s} \cdot \left(\frac{(N-H)r}{H}\right)^{\frac{(1+o_{\tilde s\to \infty}(1))\log \tilde s}{\log w}}\log N\\
&\ll \frac{H}{r}\cdot {\tilde s}^{-(\frac{1}{4}+o_{\tilde s\to \infty}(1))\tilde s} \log N,
\end{align*}
since the hypotheses of the lemma imply that
$(N-H)r/H\leq (N-H)^{3/4}$. 
Finally, let 
$s=(\log N)/(\log w)$.
Then 
$
{\tilde s}^{-(\frac{1}{4}+o_{\tilde s\to \infty}(1))\tilde s} \log N\ll 
{s}^{-s/5},
$
 since $N-H\geq N/2$ and $\tilde s\geq (\log \log N)/2$.  
The statement of the lemma follows.
\end{proof}

We may also  use Lemma~\ref{lem:friable'} to study the sum
\begin{equation}\label{eq:definition-sig0}
\Sigma_B(z)=\sum_{\substack{p\mid k \Rightarrow p\leq w\\ k>z}} \frac{\tau(k)^{B}}{k},
\end{equation}
for $B\geq 1$ and $z>0$, which will appear later in the course of proving Proposition~\ref{prop:1moment}.

\begin{lemma}\label{lem:gouda}
Let $B\geq 1$, $w\geq 2$ and $z\geq w^{C_B}$, where $C_B$ is a large enough constant. Then 
$$
\Sigma_{B}(z)\ll_B
\exp\left(-\frac{\log z}{4\log w} \right).
$$
\end{lemma}

\begin{proof} 
Let $s=(\log 2^j)/(\log w)$. 
If $(j\log 2)^2\leq w$ then it follows from 
Lemma 
\ref{lem:friable'}  that
$\Psi(2^j,w)\ll 2^{j}s^{-9s/10}$.
If 
$w<(j\log 2)^2$, we can still crudely estimate
$\Psi(2^j,w)\leq \Psi(2^j,(j\log 2)^2)\ll 2^{9j/10}$. Hence it follows that 
$$
\Psi(2^j,w)\ll 2^{9j/10}+2^{j}s^{-9s/10},
$$
for any $j\geq 1$. Furthermore, we can write
$$
s^{-s}=\exp\left(-s\log s\right)=\exp\left(-\frac{j\log 2(\log j +\log\log 2-\log\log w)}{\log w}\right).
$$

Breaking into dyadic intervals and then applying the Cauchy--Schwarz inequality, it follows that 
\begin{align*}
\Sigma_{B}(z) &\ll
 \sum_{\substack{j\in \ZZ \\ 2^j>z/2}} 2^{-j}\sum_{\substack{2^j\leq n<2^{j+1}\\p\mid n\Rightarrow p\leq w}}\tau(n)^{B}\\
&\leq  \sum_{\substack{j\in \ZZ \\ 2^j>z/2}} 2^{-j}\left(\sum_{\substack{2^j\leq n<2^{j+1}\\p\mid n\Rightarrow p\leq w}}1\right)^{1/2}\left(\sum_{2^j\leq n<2^{j+1}}\tau(n)^{2B}\right)^{1/2}.
\end{align*}
Now 
$$
\sum_{2^j\leq n<2^{j+1}}\tau(n)^{2B}\ll_B 2^j j^{2^{2B-1}}.
$$
Hence it follows that 
\begin{align*}
\Sigma_{B}(z) 
&\ll_B  
\sum_{\substack{j\in \ZZ \\ 2^j> z/2}}2^{-j/30} +
\sum_{\substack{j\in \ZZ \\ 2^j> z/2}} j^{2^{2B-1}}
\exp\left(-\frac{(\log 2)j\log j}{3\log w} \right)\\
&\ll \sum_{\substack{j\in \ZZ \\ 2^j> z/2}} 
\exp\left(-\frac{(\log 2)j\log j}{4\log w} \right)+z^{-1/30},
 \end{align*}
on taking $C_B$ large enough. The statement of the lemma easily follows. 
\end{proof}

\subsection{Proof of Proposition~\ref{prop:1moment}}

We may now return to~\eqref{eq:Unu}. We give an asymptotic formula for $U(I;\nu)$ in the next lemma, and later in Lemma~\ref{lem:5.5} we simplify the main term.

\begin{lemma}\label{lem:AAA}
Let $X$ satisfy $B^e\gg X\geq x\geq 3$. Let $I\subset [1,X]$ be an interval such that $|I|>X^{9/10}$.
Let $1\leq u\leq q\leq X^{1/5}$, and assume that $u$ and $q$ satisfy~\eqref{eq:assume-gcd} holds.
Then, for any $\nu\in \ZZ/W_0\ZZ$, we have
$$
U(I;\nu)=
\frac{\sigma(I)}{[q,W_0]} 
\sum_{\substack{k\in \NN\\
p\mid k \Rightarrow p \mid P(w)\\
\gcd(k,q)\mid u}} 
\frac{g(k)\gcd(k,q) }{k}
+
O\left(\frac{|I|}{q}
\exp\left(-\frac{\log X}{21\sqrt{\log x}} \right)\right),
$$
where $g=\tilde r_K* \mu$ and 
\begin{equation}\label{eq:def-I(x)}
\sigma(I)=\int_{I} \omega(t;B)\d t.
\end{equation}
\end{lemma}

\begin{proof} Recall the expression~\eqref{eq:Unu} for $U(I;\nu)$. We may assume without loss of generality that $I= (N-H,N]$ for some $H,N$ with $2H\leq N\leq X$. In particular $|I|=H\geq X^{9/10}$ by assumption.

Let us  write $\mathcal{C}_w$ for the set of positive integers whose only prime divisors are divisors of $P(w)$.
Any $n\in \NN$ with $n_W=d$ admits a unique  factorisation $n=dm$, in which $m$ is coprime to $ P(w)$. Hence
\begin{align*}
U(I;\nu)
&=\sum_{d\in \mathcal{C}_w} \tilde r_K(d)
\sum_{\substack{n\in I\\ n\equiv \nu\bmod{W_0}\\n\equiv u\bmod{q}\\
n_W=d
}}  \omega(n;B)\\ &=
\sum_{\substack{k\in \mathcal{C}_w}} g(k) 
\sum_{\substack{n\in I
\\ n\equiv \nu\bmod{W_0}\\ n\equiv u\bmod{q}\\
k\mid n}} \omega(n;B),
\end{align*}
where
$
g=\tilde r_K* \mu$. The  inner sum vanishes unless 
$\gcd(k,q)\mid u$ and $k\leq X$.   Moreover, we have $\gcd(k,W_0)=1$ since $k\in \mathcal{C}_w$.
Hence 
$$
U(I;\nu)
=
\sum_{\substack{k\in \mathcal{C}_w\\ \gcd(k,q)\mid u\\ k\leq X}} g(k) 
\sum_{\substack{n\in I\\ n\equiv \nu\bmod{W_0}\\
 n\equiv u\bmod{q}\\
k\mid n}}  \omega(n;B).
$$
We break the outer  sum into three intervals 
$$
I_1=(0,\sqrt{X}], \quad I_2=(\sqrt{X},|I|/q], \quad I_3=(|I|/q,X].
$$
Then 
\begin{equation}\label{eq:sum-sig123}
U(I;\nu)=\Sigma_1+\Sigma_2+\Sigma_3,
\end{equation}
where
$$
\Sigma_i=
\sum_{\substack{k\in \mathcal{C}_w\\ \gcd(k,q)\mid u\\ k\in I_i}} g(k) 	
\sum_{\substack{n\in I
\\ n\equiv \nu\bmod{W_0}\\n\equiv u\bmod{q}\\
k\mid n}}  \omega(n;B),
$$
for $1\leq i\leq 3$. 

Note that we have the simple estimate
\begin{equation}\label{eq:sum-int}
\sum_{a\leq m\leq  b}g(m) =\int_a^b g(t)\d t +O\left(\int_a^b |g'(t)|\d t +|g(a)|+g(b)|\right),
\end{equation}
for any continuously  differentiable function $g:[a,b]\to \RR$. Applying ~\eqref{eq:sum-int}, it follows from parts (i) and (iii) of Lemma~\ref{lem:props-om} that
$$
\sum_{\substack{n\in I
\\ n\equiv \nu\bmod{W_0}\\n\equiv u\bmod{q}\\
k\mid n}}  \omega(n;B)=
\frac{1}{[k,q,W_0]}\int_{I} \omega(t;B)\d t+O(B^{-e}X+1),
$$
where $[k,q,W_0]$ is the least common multiple of $k,q$ and $W_0$.
Since $k$ and $W_0$ are coprime,
it follows that 
$$
[k,q,W_0]=\frac{qkW_0}{\gcd(q,k)\gcd(q,W_0)}=\frac{[q,W_0]k}{\gcd(q,k)}.
$$
Recalling the definition~\eqref{eq:def-I(x)} of $\sigma(I)$, we therefore obtain
 $$
 \Sigma_1
= \sum_{\substack{k\in \mathcal{C}_w\\ \gcd(k,q)\mid u\\ k\leq \sqrt{X}}} g(k) 
\left(\frac{\sigma(I)\gcd(k,q)}{[q,W_0]k}  +O(1)\right),
$$
since we are assuming that $B^e\gg X$.

Next, we note that  $g(p^j)=\tilde r_K(p^j)-\tilde r_K(p^{j-1})$ for any prime power $p^j>1$,
where $\tilde r_K$ is given by~\eqref{eq:r-tilde}. For any  $p\mid P(w)$, we  claim that 
\begin{equation}\label{eq:nature-g}
|g(p^j)|
\begin{cases}
=0 &\text{ if $j>k(x)$,}\\
\leq  \tau(p^j)^{e+1} &\text{ if $j\leq k(x)$.}
\end{cases}
\end{equation}
This is obvious if $j>k(x)$. 
If $j=k(x)$,  then 
$$
|g(p^j)|=|\alpha_p^{-1}\beta_{p^{k(x)}}-r_K(p^{k(x)-1})|\leq (1+k(x))^{e+1}
$$
by~\eqref{eq:beta-bound'} and~\eqref{eq:moon}, from which the claim follows.  Finally, if $j<k(x)$ then 
$|g(p^j)| \leq 
2\tau(p^j)^e \leq \tau(p^j)^{e+1}$.
Hence it follows from multiplicativity and~\eqref{eq:moon} that  
\begin{equation}\label{eq:size-g}
|g(k)|=\prod_{p^j\| k} |g(p^j)| \leq 
 \tau(k)^{e+1} \quad \text{ if }\, k\in \mathcal{C}_w.
\end{equation}
In particular, we have  $g(k)=O_\ve(k^\ve)$ for any $\ve>0$.
It follows that 
 \begin{align*}
 \Sigma_1
&= \frac{\sigma(I)}{[q,W_0]}
\sum_{\substack{k\in \mathcal{C}_w\\ \gcd(k,q)\mid u\\ k\leq \sqrt{X}}} \frac{g(k)\gcd(k,q) }{k}
+O_\ve(X^{1/2+\ve}).
\end{align*}
But then, on noting that $\sigma(I)=O(|I|)$,  we deduce  that 
 \begin{align*}
 \Sigma_1
=~& \frac{\sigma(I)}{[q,W_0]} 
\sum_{\substack{k\in \mathcal{C}_w\\ \gcd(k,q)\mid u}} \frac{g(k)\gcd(k,q) }{k}\\
&\quad 
+O_\ve\left(X^{1/2+\ve} +\frac{|I|\gcd(u,q)}{[q,W_0]}\Sigma_{e+1}(\sqrt{X})\right),
\end{align*}
in the notation of ~\eqref{eq:definition-sig0}. Lemma~\ref{lem:gouda} yields
$$
\Sigma_{e+1}(\sqrt{X})
\ll \exp\left(-\frac{\log X}{8\log w} \right)=\exp\left(-\frac{\log X}{8\sqrt{\log x}} \right),
$$
since  $w=\exp(\sqrt{\log x})$. 
Recalling the assumption ~\eqref{eq:assume-gcd}, 
it  therefore follows that 
$$
\Sigma_1=
\frac{\sigma(I)}{[q,W_0]} 
\sum_{\substack{k\in \mathcal{C}_w\\ \gcd(k,q)\mid u}} \frac{g(k)\gcd(k,q) }{k}
+
O\left(\frac{|I|}{q}
\exp\left(-\frac{\log X}{16\sqrt{\log x}} \right)\right).
$$
Here, we have observed that the error term $X^{1/2+\ve}$ is dominated by the stated  error term, since 
$|I|\geq X^{9/10}$ and 
$q\leq X^{1/5}$.

Turning to $\Sigma_2$, the bound $\omega(n;B)\ll 1$ gives us 
$$
\Sigma_2\ll 
\sum_{\substack{k\in \mathcal{C}_w\\ \gcd(k,q)\mid u\\ k\in I_2}} |g(k)| 
\frac{|I|}{[q,k]}\ll_\ve \frac{|I|\gcd(u,q)}{q} \Sigma_{e+1}(\sqrt{X}),
$$
since $kq\leq |I|$.
Hence, Lemma~\ref{lem:gouda} yields
$$
\Sigma_2\ll 
\frac{|I|}{q}
\exp\left(-\frac{\log X}{16\sqrt{\log x}} \right),
$$
which is satisfactory. 

We now turn to $\Sigma_3$, for which we use an approach based on the hyperbola method. 
Estimating $\omega(n;B)\ll 1$,  we 
make the change of variables $n=km$ and rearrange the sum, in order to obtain
\begin{align*}
\Sigma_3
&\ll \sum_{\substack{k\in \mathcal{C}_w\\ \gcd(k,q)\mid u\\ |I|/q\leq k\leq X}} |g(k)|
\sum_{\substack{n\in I\\n\equiv u\bmod q\\ k\mid n}}1 \ll 
\sum_{\substack{m\leq qX/|I|}}
\sum_{\substack{k\in \mathcal{C}_w\cap m^{-1}I\\ \gcd(k,q)\mid u\\ 
km\equiv u\bmod{q}
}} |g(k)|.
\end{align*}
Note that $qX/|I|\leq q X^{1/10}\leq X^{1/3}$. Moreover, 
the inner sum vanishes unless $\gcd(m,q)$ divides $u$, in which case there 
exists $u'\in \ZZ$ such that the inner congruence is equivalent to $k\equiv u'\bmod{q'}$, where 
$q'=q/\gcd(m,q)$.
Hence it follows from the Cauchy--Schwarz inequality that 
\begin{equation}\label{eq:foot}
\Sigma_3
\ll
\sum_{\substack{m\leq X^{1/3}\\ \gcd(m,q)\mid u }}
\sum_{\substack{k\in \mathcal{C}_w\cap m^{-1}I\\ 
k\equiv u'\bmod{q'}
}} |g(k)|
\leq 
\sum_{\substack{m\leq X^{1/3}\\ \gcd(m,q)\mid u }}
\sqrt{U(m)V(m)},
\end{equation}
where
$$
U(m)=\#\left\{k\in \mathcal{C}_w\cap m^{-1}I: 
k\equiv u'\bmod{q'}\right\}
$$
and
$$
V(m)=
\sum_{\substack{k\in m^{-1}I\\ 
k\equiv u'\bmod{q'}
}} |g(k)|^2.
$$

Recall now that by the discussion at the beginning of the proof we have $I= (N-H,N]$ for some $H,N$ with $2H\leq N\leq X$ and $H>X^{9/10}$. Then 
 $m^{-1}I=(N'-H',N']$, with $H'=H/m$ and $N'=N/m$. 
 It  follows from Lemma 
\ref{lem:friable} that 
$$
U(m)\ll \frac{H'}{q'} \cdot {s}^{-s/5}, 
$$
where $s=(\log N')/(\log w)$. Now $N'=N/m\geq 2H/m\gg \sqrt{X}$, since $H\geq X^{9/10}$ and $m\leq X^{1/3}$.
 We readily conclude that 
\begin{equation}\label{eq:ankle}
U(m)\ll \frac{H }{mq} \exp\left(-\frac{\log X}{10\sqrt{\log x}}\right),
\end{equation}
since $\gcd(m,q)\leq \gcd(u,q)\leq  \exp(\frac{2\sqrt{\log X}}{\log\log X})$ in~\eqref{eq:foot}.

Turning to the remaining sum, it follows from~\eqref{eq:size-g} that 
$$
V(m)\leq 
\sum_{\substack{N'-H'<k\leq N'\\
k\equiv u'\bmod{q'}
}} 
h(k),
$$
where
$h(k)=\tau(k)^{2e+2}$ is a non-negative multiplicative arithmetic function that lies in the class of 
arithmetic functions handled in work of  Shiu ~\cite[Thm.~1]{shiu}. 
We first need to take care of the possibility that $u'$ and $q'$ are not coprime. Let $b=\gcd(u',q')$. If $k=bk'$ then $h(k)\leq h(b)h(k')$. Writing
$u''=u/b$ and $q''=q'/b$, 
we thereby obtain
$$
V(m)=
h(b)\sum_{\substack{N''-H''<k'\leq N''\\
k'\equiv u''\bmod{q''}
}} h(k'),
$$
with $H''=H'/b$, $N''=N'/b$, and where now $\gcd(u'',q'')=1$. It follows from 
\eqref{eq:assume-gcd} that 
$
b\leq \exp(\frac{2\sqrt{\log X}}{\log\log X}).$
Hence Shiu's theorem yields
\begin{align*}
V(m)
&\ll h(b)
 \frac{H''}{q''} \frac{\log\log q''}{\log N''} \exp\left(\sum_{\substack{p\leq N''}} \frac{h(p)}{p}\right)\\
&\ll \frac{h(b)}{m}
 \frac{H\gcd(m,q)}{q} (\log\log X) (\log X)^{2^{2e+2}-1},
\end{align*}
by Mertens' theorem. 
Now $\gcd(m,q)\leq \gcd(u,q)
\leq \exp(\frac{2\sqrt{\log X}}{\log\log X})$, by 
\eqref{eq:assume-gcd}.
Using $h(b)\ll \sqrt{b}$, it easily follows that 
$$
V(m)\ll 
  \frac{H}{mq}  \exp\left(\frac{4\sqrt{\log X}}{\log\log X}\right).
$$

Finally, 
combining our estimate for $V(m)$ with~\eqref{eq:ankle} in~\eqref{eq:foot}, we finally conclude that
\begin{align*}
\Sigma_3
&\ll \frac{|I|}{q} \exp\left(\frac{2\sqrt{\log X}}{\log\log X}\right)
  \exp\left(-\frac{\log X}{20\sqrt{\log x}}\right)\sum_{\substack{m\leq X^{1/3}}}
  \frac{1}{m}\\
&\ll   \frac{|I|}{q} 
  \exp\left(-\frac{\log X}{21\sqrt{\log x}}\right).
  \end{align*}
The  lemma follows on combining our estimates for $\Sigma_1, \Sigma_2$ and $\Sigma_3$ in~\eqref{eq:sum-sig123}.
\end{proof}

Recalling the definition~\eqref{eq:Unu} of $U(I;\nu)$, 
we now want to insert  Lemma~\ref{lem:AAA} 
into~\eqref{eq:berkley-hotel}. Note that 
$$
1\leq \frac{X^{1/5}}{q }\ll \frac{|I|}{q} 
\exp\left(-\frac{\log X}{21\sqrt{\log x}} \right),
$$
since $q\leq X^{1/5}$. 
Moreover, 
$$
c(w)\exp\left(-\frac{\log X}{21\sqrt{\log x}} \right)
\ll
\exp\left(-\frac{\log X}{22\sqrt{\log x}} \right),
$$
since $c(w)\leq \varphi^*(W)\ll \log\log  W\ll \sqrt{\log x}$. 
It now follows that 
\begin{equation}\label{eq:sirius*}
T(I;\nu)= 
\frac{c(w)\sigma(I)}{[q,W_0]} 
\sum_{\substack{k\in \NN\\
p\mid k \Rightarrow p \mid P(w)\\
\gcd(k,q)\mid u}} 
\frac{g(k)\gcd(k,q) }{k}
+
O\left(\frac{|I|}{q}
\exp\left(-\frac{\log X}{22\sqrt{\log x}} \right)\right),
\end{equation}
where $g=\tilde r_K*\mu$, $\sigma(I)$ is given by~\eqref{eq:def-I(x)}
and $c(w)$ is given by~\eqref{eq:c(w)}.  
The following result provides a precise  calculation of the 
$k$-sum. 

\begin{lemma}\label{lem:5.5}
Let $u,q\in \NN$ such that $1\leq u\leq q$. 
Assume that $p^{k(x)}\nmid q$
for any $p\leq w$.
Then 
$$
c(w)
\sum_{\substack{k\in \NN\\ 
p\mid k \Rightarrow p\mid P(w)
\\ \gcd(k,q)\mid u}} \frac{g(k)\gcd(k,
q) }{k}
= 
\gamma(\gcd(q,W_1),u),
$$
where  $W_1$ is given by~\eqref{eq:WWW}.
\end{lemma}

\begin{proof}
We start by noting that 
$$
\gamma(\gcd(q,W_1),u)=\prod_{p\mid P(w)} \gamma(p^{\ell},u),
$$
where $\ell=v_p(q)<k(x)=v_p(W_1)$.
Appealing to parts (i) and (iii) of Lemma~\ref{lem:LM}, we deduce that 
\begin{equation}\label{eq:HHH}
\gamma(p^{\ell},u)=\begin{cases}
\alpha_p r_K(p^{v_p(u)}) & \text{ if $v_p(u)<\ell$,}\\
\beta_{p^\ell}
& 
\text{ if $v_p(u)\geq \ell$,}
\end{cases}
\end{equation}
in the notation of~\eqref{eq:alpha-p} and 
\eqref{eq:beta-p}.

Next, 
it follows from multiplicativity that 
$$
c(w)
\sum_{\substack{k\in \NN\\ 
p\mid k \Rightarrow p\mid P(w)
\\ \gcd(k,q)\mid u}} \frac{g(k)\gcd(k,q) }{k}
= 
\prod_{p\mid P(w)} \alpha_p \lambda_p,
$$
where $\alpha_p$ is given by 
\eqref{eq:alpha-p} and
$$
\lambda_p=
\sum_{
\substack{ j\geq 0\\
\min(j,v_p(q))\leq v_p(u)
}}\frac{g(p^j)
p^{\min(j,v_p(q))}
}{p^j}.
$$
We proceed by studying $\alpha_p\lambda_p$ for a given prime $p\mid P(w)$.

Note that   $g(p^j)=\tilde r_K(p^j)-\tilde r_K(p^{j-1})$, if $j\geq 1$, where
$\tilde r_K$ is given by~\eqref{eq:r-tilde}. In particular, it follows that $g(p^j)=0$ if $j>k(x)$.  
Let $\ell=v_p(q)$, as before. Then  we may write
\begin{align*}
\lambda_p
&=
\sum_{
\substack{
0\leq j\leq \min(\ell, v_p(u))
}}g(p^j) + 
p^{\ell}
\sum_{
\substack{
\ell<j\leq k(x) \\
\min(j,\ell)\leq v_p(u)
}}\frac{g(p^j)}{p^j}.
\end{align*}
The first sum is  $r_K(p^{\min(\ell, v_p(u))})$, since $\ell<k(x)$. 
If $\ell>v_p(u)$ then the second sum vanishes and it follows that 
$
\alpha_p\lambda_p=\alpha_p r_K(p^{v_p(u)})=\gamma(p^\ell,u),
$
by~\eqref{eq:HHH}.
Alternatively, if  $\ell\leq v_p(u)$ then we deduce that 
\begin{align*}
\lambda_p
&=r_K(p^\ell) +
p^{\ell}\sum_{
\substack{\ell<j\leq k(x)}}
\frac{\tilde r_K(p^j)-\tilde r_K(p^{j-1})}{p^j}
\\
&=r_K(p^\ell)+p^\ell\left(-\frac{r_K(p^{\ell})}{p^\ell}+
\left(1-\frac{1}{p}\right)\sum_{
\substack{\ell\leq j< k(x)}}
\frac{r_K(p^j)}{p^j} +\frac{\alpha_p^{-1}\beta_{p^{k(x)}}}{p^{k(x)}}\right).
\end{align*}
But then, recalling~\eqref{eq:alpha-p} and~\eqref{eq:beta-p}, we deduce that
\begin{align*}
\alpha_p\lambda_p
&=
p^\ell\sum_{
\substack{\ell\leq j< k(x)}}
\frac{r_K(p^j)}{p^j} 
\prod_{\fp\mid p}\left(1-\frac{1}{\n\fp}\right)
+\frac{p^\ell\beta_{p^{k(x)}}}{p^{k(x)}}\\
&=
p^\ell\sum_{
\substack{j\geq \ell}}
\frac{r_K(p^j)}{p^j} 
\prod_{\fp\mid p}\left(1-\frac{1}{\n\fp}\right)\\
&=\gamma(p^\ell,u),
\end{align*}
by ~\eqref{eq:HHH}.

Putting everything together we finally deduce that 
\begin{align*}
c(w)
\sum_{\substack{k\in \NN\\ 
p\mid k \Rightarrow p\mid P(w)
\\ \gcd(k,q)\mid u}} \frac{g(k)\gcd(k,q) }{k}
&= 
\prod_{p\mid P(w)} \gamma(p^\ell,u) 
= 
 \gamma(\gcd(q,W_1),u),
\end{align*}
which thereby completes the proof.
\end{proof}

We now have everything in place to complete the proof of Proposition~\ref{prop:1moment}.

\begin{proof}[Proof of Proposition~\ref{prop:1moment}]
Combining Lemma~\ref{lem:5.5} with~\eqref{eq:sirius*}, we obtain
$$
T(I;\nu)= 
\frac{\gamma(\gcd(q,W_1),u)}{[q,W_0]} \cdot \sigma(I)
+
O\left(\frac{|I|}{q}
\exp\left(-\frac{\log X}{22\sqrt{\log x}} \right)
\right).
$$
We now wish to insert this into~\eqref{eq:S-to-T}. 
In view of Lemma~\ref{lem:sun} and ~\eqref{eq:W0}, 
it easily follows that 
$$
\exp\left(-\frac{\log X}{22\sqrt{\log x}} \right)
 \sum_{\nu\bmod{W_0}} \gamma(W_0,\nu)\ll 
\exp\left(-\frac{\log X}{23\sqrt{\log x}} \right).
$$
Moreover, for any $d\mid W_0$ we have
\begin{align*}
\sum_{\substack{\nu \bmod{W_0}\\
\nu\equiv u \bmod{d}
}} 
\gamma(W_0,\nu)&=
\frac{1}{W_0^{e-1}} 
\hspace{-0.4cm}
\sum_{\substack{\x \in (\ZZ/W_0\ZZ)^e\\
\nf_K(\x)\equiv u\bmod{d}
}}
\hspace{-0.5cm}
\#\left\{
\nu \bmod{W_0}: 
\begin{array}{l}
\nu\equiv u \bmod{d}\\
 \nu\equiv \nf_K(\x) \bmod{W_0}
 \end{array}
 \hspace{-0.2cm}
 \right\}\\
&= \frac{1}{W_0^{e-1}} \cdot \frac{W_0^e\rho(d,u)}{d^e} \cdot 1\\
&=\frac{W_0\gamma(d,u)}{d}.
\end{align*}
Taking $d=\gcd(q,W_0)$, we deduce that 
\begin{align*}
\sum_{\substack{\nu \bmod{W_0}\\
\nu\equiv u \bmod{\gcd(W_0,q)}
}} 
\hspace{-0.3cm}
\gamma(W_0,\nu) 
\cdot \frac{\gamma(\gcd(q,W_1),u)}{[q,W_0]}
%&=
%\frac{W_0 \gamma(\gcd(q,W_0),u)}{\gcd(q,W_0)}\\
%\cdot\frac{\gamma(\gcd(q,W_1),u)}{[q,W_0]}\\
&=
\frac{\gamma(\gcd(q,W),u)}{q}.
\end{align*}
Hence 
$$
S(I;q,u)=
\frac{\gamma(\gcd(q,W),u)}{q} \cdot \sigma(I)
+
O\left(\frac{|I|}{q}
\exp\left(-\frac{\log X}{23\sqrt{\log x}} \right)
\right),
$$
which completes the proof.
\end{proof}

\section{Norm forms: approximation by the localised counting function}\label{sec:localised}

\subsection{Checking the hypotheses}
Our goal in this section is to apply 
Theorem 
~\ref{prop_general} to the sequence
$
F(n)=R_K(n;B)-\hat R_K(n;B).
$
We begin by checking that assumption (1) holds.
Any vector $\x\in \ZZ^e\cap B\mathcal{B}$ counted by 
$R_K(n;B)$ produces an element $x=x_1\omega_1+\cdots+x_e\omega_e\in \fo_K$ such that 
$N_{K/\QQ}(x)=n$. On accounting for units,  we  find that  
$
R_K(n;B)\ll r_K(n),
$
where the implied constant depends on $K$ and 
$r_K(n)$ is the number of integral ideals $\fa\subset \fo_K$ of norm $n$. It follows from 
\eqref{eq:moon} that $r_K(n)\leq \tau(n)^e$. 
Hence
\begin{equation}\label{eq:cloth}
R_K(n;B)\ll \tau(n)^e,
\end{equation}
for an implied constant depending only on $K$.
Next, it follows from 
\eqref{eq:def-hatR},  part (iii) of Lemma~\ref{lem:props-om} and  Lemma~\ref{lem:sun} that 
\begin{align*}
\hat R_K(n;B)
&= \gamma(W,n) \cdot \omega(n;B)
\ll \gamma(W,n)
\ll \varphi^*(n)\tau(n)^e\leq \tau(n)^{e+1}.
\end{align*}
Hence 
$$
|F(n)|\leq R_K(n;B)+\hat R_K(n;B)\ll \tau(n)^{e+1},
$$
so that $|F(n)|\leq \tau_J(n)$ if $J$ is large enough.
This establishes  that   assumption (1) holds in Theorem 
\ref{prop_general} for any large enough constant $A$.

Turning to assumption (2), having addressed the first moment of  $\hat R_K(n;B)$ in 
Proposition~\ref{prop:1moment},
we must next examine the first moment of 
$R_K(n;B)$.

\begin{lemma}\label{lem:1stmom-1}
Let $B\geq 1$, let $I\subset \RR$ be an interval,  let $q\in \NN$ and let $1\leq u\leq q$. Then 
$$
\sum_{\substack{n\in I\\n\equiv u\bmod q}}R_K(n;B)=
\frac{\gamma(q,u)  }{q}
\int_{I}\omega(t;B)\d t
+O(q^{e+1} B^{e-1/2}),
$$
where $\gamma(q,u)$ is given by ~\eqref{eq:def-gamma}.
\end{lemma}

\begin{proof}
Recalling
\eqref{eq:define-R},
we have 
$$
\sum_{\substack{n\in I\\n\equiv u\bmod q}}R_K(n;B)=
\sum_{
\substack{
\s\in (\ZZ/q\ZZ)^e\\
\nf_K(\s)\equiv u\bmod q
}} N(q,\s),
$$
where
$$
N(q,\s)=\#\left\{\x\in \ZZ^e\cap B\mathcal{B}: 
\begin{array}{l}
\nf_K(\x)\in I\\
 \x\equiv \s\bmod{q}
 \end{array}
 \right\}.
$$
Note that the outer sum features at most $q^e$ choices for $\s$.
We would like to approximate the inner sum by an integral. 
Our goal is to prove that 
\begin{equation}\label{eq:towel}
N(q,\s)=\frac{1}{q^e} \vol\left\{\x\in B\mathcal{B}: \nf_K(\x)\in I\right\} +O_\ve(q B^{e-1+\ve}),
\end{equation}
for any $\ve>0$. 
Note that 
$$
\vol\left\{\x\in B\mathcal{B}: \nf_K(\x)\in I\right\} =\int_{I}\omega(t;B)\d t,
$$
by part (ii) of Lemma~\ref{lem:props-om}. Hence,
on taking  $\ve=1/2$,  it follows from~\eqref{eq:towel} that 
$$
\sum_{\substack{n\in I\\n\equiv u\bmod q}}R_K(n;B)=
\frac{\gamma(q,u)}{q^e} \int_{I}\omega(t;B)\d t 
+O(q^{e+1}B^{e-1/2}).
$$
The statement of the lemma follows, assuming~\eqref{eq:towel}, on recalling the definition 
\eqref{eq:def-gamma} of $\gamma(q,u)$.

To prove~\eqref{eq:towel}, we claim that 
\begin{equation}\label{eq:current}
N(1,\0)=
\vol\left\{\x\in B\mathcal{B}: \nf_K(\x)\in I\right\}+O(B^{e-1}).
\end{equation}
But 
\begin{align*}
&\vol\left\{\x\in B\mathcal{B}: \nf_K(\x)\in I\right\}\\
&\qquad\qquad=
\sum_{\x\in \ZZ^n} 
\vol\left\{\u\in [0,1]^e: \text{$\x+\u\in B\mathcal{B}$ and $\nf_K(\x+\u)\in I$}\right\}.
\end{align*}
Now $\nf_K(\x+\u)=\nf_K(\x)+O(|\x|^{e-1})$ for any $\x\in \RR^e$ and 
 $\u\in [0,1]^e$. Hence there  are  $O(B^{e-1})$ pairs of vectors $(\x,\u)$ that contribute to the right-hand side, but for which 
$\x\not \in B\mathcal{B}$ or  $\nf_K(\x)\not \in I$. The claim~\eqref{eq:current} follows. 

Next, we turn to the  estimation of $N(q,\s)$, which we achieve by comparing it with $N(q,\0)$. 
For $\x\in \ZZ^e$ let $\x'=(q\lfloor x_1/q\rfloor,\dots,q\lfloor x_e/q\rfloor)$.
Then if $\x$ is counted by $N(q,\s)$ it will follow that 
$\x'$ is counted by $N(q,\0)$ unless either $\x$ is within a distance $O(q)$ of the boundary 
of $B\mathcal{B}$, or unless $\nf_K(\x)$ is within a distance of $O(qB^{e-1})$ of the end points of $I$. 
A  given  integer  arises at most $O_\ve(B^\ve)$ times as a value of $\nf_K(\x)$ with
$\x\in \ZZ^e\cap B\mathcal{B}$, by~\eqref{eq:cloth}.
Hence
$$
N(q,\0)=N(q,\s)+O_\ve(qB^{e-1+\ve}).
$$
Summing over all $\s \bmod{q}$ we deduce that 
$$
q^e N(q,\0) = N(1,\0)+O_\ve(q^{e+1}B^{e-1+\ve}).
$$
The claimed bound~\eqref{eq:towel} follows on 
combining this with~\eqref{eq:current}.
\end{proof}

The  following result is the key to dealing with assumption  (2) in Theorem 
\ref{prop_general}.

\begin{lemma}\label{lem:for_(2)}
Let $\ve>0$ be fixed. Let $B\geq 1$. Define
$
F(n)= R_K(n;B)-\hat R_K(n;B)
$ 
and
\begin{equation}\label{eq:defQ}
\mathcal{Q}=\left\{p^{k(x)}: p\leq w\right\}.
\end{equation}
Let  $I\subset \RR$ be an interval such that 
$I\subset [-2H x^d, 2 H x^d]$, $|I|>H^{1-\ve}$ and $x\leq H^{\Delta_{d,e}}$ with $\Delta_{d,e}$ as in~\eqref{eq:suff-small}.
Let $q\in \NN$ and let $1\leq u\leq q$. 
Assume that  $q$ is not divisible by any  element of $\mathcal{Q}$, that 
$q\leq x^{d}$ and that 
\begin{equation}\label{eq:window}
\gcd(u,q)\leq \exp \left(\frac{2\sqrt{\log x}}{\log\log x}\right).
\end{equation}
Then
$$
\left|\sum_{\substack{n\in I\\n\equiv u\bmod q}}F(n)\right|
\ll
\frac{|I|}{q} \left(
 \exp\left(-\tfrac{1}{23}\sqrt{\log x} \right) 
+\frac{x^{d(e+2)}B^{e-1/2}}{|I|} \right).
$$
\end{lemma}

\begin{proof}
To begin with,  we break the interval 
$I=I_+\sqcup I_-$ into a non-negative piece and a negative piece. On  $I_+$ 
we shall apply  
Proposition~\ref{prop:1moment} 
with $H^{1-\ve}\ll  X\ll B^e$. 
Since $H\geq x^2$ it follows that 
\eqref{eq:assume-gcd} holds with $X=2Hx^d$ under the hypotheses of the lemma. 
Moreover, since $x\leq H^{1/2}$, we can crudely estimate 
$q\leq x^{d}\leq H^{1/6}\ll X^{1/5}$.
Thus the hypotheses of 
Proposition~\ref{prop:1moment}  are met and it follows that 
$$
\sum_{\substack{n\in I_+\\n\equiv u\bmod q}}
\hspace{-0.2cm}
\hat R_K(n;B)=
\frac{\gamma(\gcd(q,W),u)}{q}\int_{I_+} \omega(t;B)\d t  +
 O\left(
E
\right),
$$
where $E= \frac{|I|}{q}
 \exp(-\tfrac{1}{23}\sqrt{\log x})$.
To handle $I_-$ we make a change of $n\mapsto -n$ and then reapply 
Proposition~\ref{prop:1moment}. This leads to the conclusion that 
$$
\sum_{\substack{n\in I\\n\equiv u\bmod q}}
\hspace{-0.2cm}
\hat R_K(n;B)=
\frac{\gamma(\gcd(q,W),u)}{q}\int_{I} \omega(t;B)\d t  +
 O(E).
$$

It follows from part (iii) of Lemma~\ref{lem:props-om} that 
$$
\int_{I}\omega(t;B)\d t\ll |I|.
$$
Combining the previous first moment bound 
with Lemma~\ref{lem:1stmom-1}, we deduce that 
\begin{align*}
\sum_{\substack{n\in I\\n\equiv u\bmod q}}F(n)
\ll~&
\frac{|I|}{q}\left|
\gamma(q,u)
-\gamma(\gcd(q,W),u)\right|
\\
&\quad +
\frac{|I|}{q} \left(
 \exp\left(-\tfrac{1}{23}\sqrt{\log x} \right) 
+\frac{x^{d(e+2)} B^{e-1/2}}{|I|} \right),
\end{align*}
since $q\leq x^{d}$.

Let us put 
$$
D_q=\gamma(q,u)
-\gamma(\gcd(q,W),u),
$$
for any $q\in \NN$. Assuming that $q\leq x^{d}$ and that $q$ is not divisible by any element of $\mathcal{Q}$, in the notation of 
\eqref{eq:defQ}, we shall prove that 
\begin{equation}\label{eq:shoed}
D_q\ll \exp\left(-\tfrac{1}{2} \sqrt{\log x}\right).
\end{equation}
This will suffice to complete the proof of the lemma.
In order to prove the claim, 
we recall the definition~\eqref{eq:def_W} of $W$. We shall factorise $q$ as $q_0q_1$, where 
$$
q_0=\prod_{\substack{p^j\| q\\ p\leq w}} p^j \quad \text{ and } \quad 
q_1=\prod_{\substack{p^j\| q\\ p> w}} p^j.
$$
Our assumption on $q$ implies that  $q_0\mid W$. Hence
 we have $\gamma(q,u)=\gamma(q_0,u)\gamma(q_1,u)$ and 
$\gamma(\gcd(q,W),u)=
\gamma(q_0,u)$.
 It follows that 
\begin{equation}\label{eq:shoe1}
D_q=\gamma(q_0,u)\left(\gamma(q_1,u)-1\right).
\end{equation}

We appeal to our work in Section~\ref{sec:local}
to estimate these quantities. On the one hand, 
Lemma~\ref{lem:sun}
yields
$
\gamma(q_0,u)\ll  \varphi^*(q_0) \tau(\gcd(u,q_0))^e.
$
We now appeal to the bound 
$$
\tau(n)^e\ll n^{1/2}.
$$
Since $\gcd(u,q_0)\leq \gcd(u,q)$ and 
$\gcd(u,q)$ satisfies~\eqref{eq:window}, 
we deduce that 
\begin{equation}\label{eq:shoe2}
\gamma(q_0,u)
\ll \exp\left(\frac{2 \sqrt{\log x}}{\log\log x}\right).
\end{equation}

On the other hand, the Chinese remainder theorem yields
\begin{align*}
\gamma(q_1,u)=
\prod_{\substack{p^j \| q \\ p>w }} 
p^{-j(e-1)}\rho(p^j,u).
\end{align*}
If $v_p(u)\geq 1$ for some prime $p$ in this product, then 
there exists a prime $p>w$ such that $p\mid u$. But this is impossible, since we are assuming that $\gcd(u,q)\leq w$, where $w$ is given by~\eqref{eq:def_W}.  Hence we can assume that $v_p(u)=0$ for every $p>w$, whence
\begin{align*}
p^{-j(e-1)}\rho(p^j,u)&=  
\left(1-\frac{1}{p}\right)^{-1}\prod_{\fp\mid p} \left(1-\frac{1}{\n \fp}\right)
=1-\frac{r_K(p)-1}{p}+O\left(\frac{1}{p^2}\right),
\end{align*}
by part (iii) of Lemma ~\ref{lem:LM} and~\eqref{eq:clip}.
But then it follows that 
$$
\left|\gamma(q_1,u)-1\right|=\left|1-
\prod_{\substack{p^j \| q \\ p>w }}  
\left(1-\frac{r_K(p)-1}{p}+\frac{\theta_p}{p^2}\right)\right|
$$
for some $\theta_p=O(1)$.
Let 
$
c_p=r_K(p)-1-\theta_p/p$ and
extend $c_d$ multiplicatively on square-free integers $d$ via $c_d=\prod_{p\mid d}c_{p}$.
Then we have 
\begin{align*}
\prod_{\substack{p^j \| q \\ p>w }}  
\left(1-\frac{r_K(p)-1}{p}+\frac{\theta_p}{p^2}\right)
&= \prod_{p\mid q_1} \left(1-\frac{c_p}{p}\right)
= 1+\sum_{\substack{d\mid q_1\\ d>1}}\frac{\mu(d)c_d}{d}.
\end{align*}
Note that $d>w^{\omega(d)}$, for any $d\mid q_1$ such that $d>1$.
Moreover, there is a constant $C>0$, depending only on $K$, such that  $|c_d|\leq C^{\omega(d)}$. Hence
\begin{align*}
\left|\sum_{\substack{d\mid q_1\\ d>1}}\frac{\mu(d)c_d}{d}\right|
&\leq \sum_{t=1}^{\omega(q_1)} \frac{C^t}{w^t}\binom{\omega(q_1)}{t}\ll\left(1+\frac{C}{w}\right)^{\omega(q_1)}-1\ll\frac{ \omega(q_1) }{w},
\end{align*}
as $w>\log x\geq \omega(q_1)$.
Thus
\begin{equation}\label{eq:shoe3}
\left|\gamma(q_1,u)-1\right|\ll \frac{\omega(q_1)}{ w}\leq \frac{\log q_1}{w}.
\end{equation}
We now combine~\eqref{eq:shoe2} and~\eqref{eq:shoe3} in~\eqref{eq:shoe1} to conclude
that 
$$
D_q\ll  \frac{\log q}{ w}\exp\left(\frac{2e \sqrt{\log x}}{\log\log x}\right).
$$
Finally, we recall that 
$\log q=O(\log x)$ and 
$w=\exp(\sqrt{\log x})$ in~\eqref{eq:def_W}.
The stated bound ~\eqref{eq:shoed} easily follows.
\end{proof}

\subsection{Next steps}

We   summarise the next stage in the proof of Theorem~\ref{t:1}.
Recall the definitions \eqref{eq:stripe} and  
\eqref{eq:def-Sd-l} of  
$S_d(H)$ and $S_d^{\text{loc}}(H)$, respectively.  
Rather than working with 
$S_d(H)$ and $S_d^{\text{loc}}(H)$, it will be convenient to study a refinement in which 
the leading constant $c_0$ is placed in a dyadic interval. Let  $\tilde H$ be such that 
\begin{equation}\label{eq:tilde-H}
H \exp\left(-\sqrt{\log H}\right)
\leq \tilde H \leq H.
\end{equation}
We define 
\begin{equation}\label{eq:Sd+}
S_d(\tilde H, H)=
\left\{\c\in S_d(H): \frac{1}{2}\tilde H\leq c_0\leq \tilde H \right\}
\end{equation}
and 
\begin{equation}\label{eq:Sdloc}
S_d^{\text{loc}}(\tilde H, H)=S_d(\tilde H, H)\cap 
S_d^{\text{loc}}(H).
\end{equation}
We proceed by establishing the following result.

\begin{proposition}\label{p:STEP1}
Let $A\geq 1$ be fixed, and let $\delta>0$. Take $B=\tilde H^{1/e}x^{d/e}$ and assume that~\eqref{eq:tilde-H} holds. 
Recall the definition~\eqref{eq:suff-small} of $\Delta_{d,e}$. 
Then, for any 
$
\frac{1}{2}H^{\Delta_{d,e}}< x\leq H^{\Delta_{d,e}},
$
we have 
$$
\#\left\{ \c\in S_d(\tilde H, H): |N_\c(x)-\hat N_\c(x)|>\delta x\right\} \ll
\frac{ H^{d+1}}{\delta^2  (\log H)^{2A}}.
$$
\end{proposition}

\begin{proof}
It follows from 
\eqref{eq:tilde-H}, together with the fact that $H\ll x^{1/\Delta_{d,e}}$,   that 
$$
|f_{\mathbf{c}}(n)|\leq (d+1)\tilde H x^d,
$$ 
for any $n\leq x$ and any $\c\in S_d(\tilde H, H)$.
We begin by orientating ourselves for an application of   Theorem 
~\ref{prop_general}. 
We have already seen that  hypothesis (1)
 of Theorem 
\ref{prop_general} is satisfied. 
Let us proceed by checking that the  set $\mathcal{Q}$ 
defined in Lemma~\ref{lem:for_(2)} satisfies the property in hypothesis (2) of 
Theorem 
\ref{prop_general}. 
But we clearly have 
$\mathcal{Q}\subset [1,x^{2d}]\cap \mathbb{Z}$ and
\begin{align*}
 \sum_{q\in \mathcal{Q}}
q^{-1/(8d)}
&\leq \sum_{p\leq w} p^{-k(x)/(8d)}\ll
2^{-k(x)/(8d)}.
\end{align*}
This is 
$O((\log x)^{-100A})$, by our choice of $k(x)$
in~\eqref{eq:k-x}. 

Let
$I\subset [-2\tilde H x^d, 2\tilde H x^d]$
be any interval  such that $|I|> H^{1-\varepsilon}$.
Let 
 $q\leq x^{d}$ not be  a multiple of an element of $\mathcal{Q}$, and let $1\leq u\leq q$ such that 
~\eqref{eq:window} holds. Then it  it follows from
Lemma~\ref{lem:for_(2)}  that 
\begin{align}\label{eq:FIu}
\frac{q}{|I|}
\left|\sum_{\substack{n\in I\\n\equiv u\bmod q}}F(n)\right|
\ll 
\exp\left(-\tfrac{1}{23}\sqrt{\log x}\right)
+
x^{d(e+2)} \frac{(\tilde H^{1/e}x^{d/e})^{e-1/2}}{H^{1-\ve}}.
\end{align}
Clearly 
\begin{align*}
x^{d(e+2)} \frac{(\tilde H^{1/e}x^{d/e})^{e-1/2}}{H^{1-\ve}}
\leq 
x^{d(e+2)+d-d/(2e)}{H^{-1/(2e)+\ve}},
\end{align*}
by~\eqref{eq:tilde-H}. 
This is at most $H^{-\varepsilon}$,  
under our assumption that $x\leq H^{\Delta_{d,e}}$, 
where $\Delta_{d,e}$ is given by~\eqref{eq:suff-small}. 
Thus this term is subsumed by the first error term in~\eqref{eq:FIu}.
Since  we are also  assuming that 
$x>\frac{1}{2}H^{\Delta_{d,e}}$,
the overall error term is $O(\exp(-C\sqrt{\log H}))$, 
for an appropriate constant $C>0$ depending only on $d$ and $e$.
This 
 is certainly $O((\log H)^{-30Ad})$ and so  hypothesis (2) of 
Theorem 
\ref{prop_general} is fully met.

We are now in a position to complete the proof.
Define
$$
E_\delta(x,H)=\#\left\{ \c\in S_d(\tilde H, H): |N_\c(x)-\hat N_\c(x)|>\delta x\right\} ,
$$
for any  $\delta>0$,  where 
$B=\tilde H^{1/e}x^{d/e}$ in the definitions of $N_\c(x)$ and $\hat N_\c(x)$, as above.
Then
\begin{align*}
E_\delta(x,H) &<
\frac{1}{\delta^2 x^2} \sum_{\substack{\c\in S_d(\tilde H, H)}} |N_\c(x)-\hat N_\c(x)|^2 \\
&=  \frac{1}{\delta^2 x^2} 
\sum_{\substack{\c\in S_d(\tilde H, H)}}\left|\sum_{n\leq x} F(c_0n^d+c_1n^{d-1}+\cdots+c_d)\right|^2,
\end{align*}
where 
$F(n)= R_K(n;B)-\hat R_K(n;B)$. 
We use Theorem 
~\ref{prop_general} 
to estimate this sum, with $\alpha_n=1$ and $2A$ in place of $A$. It therefore follows from  
\eqref{equ1} that 
\begin{align*}
E_\delta(x,H) 
&\ll 
\frac{1}{\delta^2 x^2} \cdot  H^{d+1} x^2 (\log x)^{-2A}\ll 
\frac{ H^{d+1}}{\delta^2  (\log x)^{2A}},
\end{align*}
as required.
\end{proof}

\subsection{Final deduction of Theorem~\ref{t:1}}

Recall the definition of 
$S_d^{\text{loc}}(\tilde H, H)$ from 
\eqref{eq:def-Sd-l} and~\eqref{eq:Sdloc}.
In the next section we shall prove the following result, which shows that the localised counting function is rarely smaller than it should be.

\begin{proposition}\label{p:STEP2'}
Let $A\geq 1$ be fixed. Take $B=\tilde H^{1/e}x^{d/e}$ and assume that~\eqref{eq:tilde-H} holds.
Let $\delta>  (\log x)^{-A}$.
Then 
\begin{align*}
&\#\left\{ \c\in S_d^{\textnormal{loc}}(\tilde H, H): 
\hat N_\c(x) <\delta x\right\} 
 \ll 
 \frac{H^{d+1}}{(\log x)^{dA}}+
 (\delta (\log x)^{c_{d,e}})^{\frac{1}{d^2e}}{\tilde H H^{d}},
\end{align*}
where $c_{d,e}=e+d^2(d+1)^{e+1}$.
\end{proposition}

In the rest of this section, we apply Proposition~\ref{p:STEP2'} to finish the proof of Theorem~\ref{t:1}.
Let $A\geq 1$ be fixed and 
recall the definition
$$
N_{\c}(x)=\sum_{n\leq x}  R_K(f_\c(n);B),
$$
where $R_K$ is given by~\eqref{eq:define-R}.
We now have all the ingredients with which to show that 
$N_{\c}(x)$  almost always has the expected order of magnitude (up to powers of $\log x$), as one runs over coefficient vectors 
$\c\in S_d^{\textnormal{loc}}(H)$. This is summarised in the following result, in the light of 
which  Theorem~\ref{t:1} is a trivial consequence. 

\begin{theorem}\label{t:1'}
Let $d\geq 1$ and 
let  $K/\mathbb{Q}$ be a number field of degree $e$.
Let  $A>0$, let  $H\geq 1$ and let $x$ lie in the interval
$
\frac{1}{2}H^{\Delta_{d,e}}< x\leq H^{\Delta_{d,e}}
$
for the parameter 
$\Delta_{d,e}$ given in~\eqref{eq:suff-small}.
Then there exists $B_A>0$ such that 
$$
\frac{1}{\#S_d^{\textnormal{loc}}(H) }\# \left\{ \c\in S_d^{\textnormal{loc}}(H): 
N_{\c}(x)> \frac{x}{(\log x)^{B_A}}\right\}= 1+O_{d,K,A}\left(\frac{1}{(\log H)^A}\right),
$$
where the implied constant depends only on $d$, on the number field $K$ and on the choice of $A$.
\end{theorem}

\begin{proof}[Proof of Theorem~\ref{t:1'} assuming Proposition~\ref{p:STEP2'}]
We will find it convenient to break the range of $c_0$ into dyadic intervals.
Recall  from~\eqref{eq:Sdloc} that 
 $S_d^{\text{loc}}(\tilde H, H)$ is 
the set of $\c\in S_d^{\text{loc}}(H)$ such that $\frac{1}{2}\tilde H\leq c_0\leq \tilde H$. 
Let $A'\geq 1$. We note that 
\begin{align*}
&\# \left\{ \c\in S_d^{\text{loc}}(H): 
N_{\c}(x)> \frac{x}{(\log x)^{A'}}
\right\} 
=\#S_d^{\text{loc}}(H)-\sum_{\substack{1\leq \tilde H\leq H\\\tilde H=2^j}} E(\tilde H, H),
\end{align*}
where
$$
E(\tilde H, H)=
\# \left\{ \c\in S_d^{\text{loc}}(\tilde H, H): 
N_{\c}(x)\leq  \frac{x}{(\log x)^{A'}}
\right\}.
$$
Our aim is to show that the contribution from $E(\tilde H, H)$ is negligible. 

It will be convenient to put 
$$
\delta=\frac{1}{(\log x)^{A'}}.
$$ 
Let $\c\in S_d^{\text{loc}}(\tilde H, H)$ and 
let $B=\tilde H^{1/e}x^{d/e}$. Suppose that  $\c$ is counted by 
$E(\tilde H, H)$. We claim that 
$|N_\c(x)-\hat N_\c(x)|>\delta x$ or $\hat N_\c(x)<2\delta x.$ 
Indeed, if the opposite of both inequalities held, then it would follow that
$$
N_\c(x)=N_\c(x)-\hat N_\c(x)+\hat N_\c(x)>\delta x =\frac{x}{(\log x)^{A'}},
$$
which is impossible. Propositions~\ref{p:STEP1} and ~\ref{p:STEP2'} can therefore be combined to show that 
 $$
E(\tilde H , H) \ll \frac{H^{d+1}}{(\log H)^{A+1}},
 $$
provided that $\tilde H\geq H\exp(-\sqrt{\log H})$ and $A'$ is chosen to be sufficiently large in terms of $A$.
It now follows that
\begin{align*}
\sum_{\substack{1\leq \tilde H\leq 2H\\\tilde H=2^j}} E(\tilde H, H)
&\leq 
\sum_{\substack{\tilde H\leq 2H\exp(-\sqrt{\log H})\\\tilde H=2^j}} 
\hspace{-0.3cm}
E(\tilde H, H)+
\sum_{\substack{H\exp(-\sqrt{\log H})\leq \tilde H\leq H\\\tilde H=2^j}} 
\hspace{-0.3cm}
E(\tilde H, H)\\
&\ll
\sum_{\substack{\tilde H\leq H\exp(-\sqrt{\log H})\\\tilde H=2^j}} \tilde H H^{d}
+\frac{H^{d+1}}{(\log H)^{A}}\\
&\ll
\frac{H^{d+1}}{\exp(\sqrt{\log H})}+
\frac{H^{d+1}}{(\log H)^{A}},
\end{align*}
since there are $O(\log H)$ choices for $\tilde H$ in the second sum. 
\end{proof}

\subsection{Proof of Corollary~\ref{cor:zariski}}

Any $(\x,t)\in \ZZ^{e+1}$ counted by 
$N_{\c}(x)$ is constrained to lie in the box $[-P,P]^{e+1}$ for 
$$
P\ll B\leq H^{1/e}x^{d/e}\leq H^{1/e+d\Delta_{d,e}/e}.
$$
  Let $\ve>0$ be fixed.
According to 
Theorem~\ref{t:1'}, with probability 1, 
as admissible  coefficient vectors are ordered by height, 
there are $\gg H^{\Delta_{d,e}-\ve}$ integer solutions to the equation 
$\nf_K(\x)=f(t)$.
We now appeal to a result of Bombieri and Pila~\cite{BP}, which shows that an irreducible degree $k$ curve in $\AA^{e+1}$ has $O_{e,k,\ve}(P^{1/k+\ve})$ points in $(\ZZ\cap [-P,P])^{e+1}$, where the implied constant depends only on $e, k$ and $\ve$. Thus, if all the solutions were covered by $\leq L$ irreducible curves of degree $k\leq D$ we would have 
$$
H^{\Delta_{d,e}-\ve}\ll_{d,D,L,\ve} P^{1/k+\ve}\ll H^{(1/e+d\Delta_{d,e}/e)/k+\ve}, 
$$ 
which is a contradiction for $$
k>\frac{1}{e\Delta_{d,e}}+\frac{d}{e}. 
$$
This is equivalent to $k>d(2e^2+6e+1)/e$, on 
recalling 
the definition~\eqref{eq:suff-small} of
 $\Delta_{d,e}$. Noting that $2e^2+6e+1< 2e(e+7/2)$ for $e\geq 2$, the desired statement follows.

\section{Norm forms: the localised counting function is  rarely small}\label{sec:rarely-small}

The aim of this section is to prove 
Proposition~\ref{p:STEP2'}.
Recall the definition of 
$S_d^{\text{loc}}(\tilde H, H)$ from 
\eqref{eq:def-Sd-l} and~\eqref{eq:Sdloc}.
We will find it convenient to 
write
$$
h_\c=\gcd(c_0,\dots,c_d)
$$
for the content of the polynomial 
$f_\c(t)=c_0t^d+\cdots+c_d\in \ZZ[t]$ that is associated to any coefficient vector 
$\c\in \ZZ^{d+1}$.
Our aim in this section is to show that 
$\hat N_\c(x)$ is rarely small for 
$\c\in S_d^{\text{loc}}(\tilde H, H)$. Much as in the proof of Proposition~\ref{prop:1moment}, we need to separate out 
the small factors of $W$, but this time we shall also need to separate out the factors of $W$ that share a common factor with $h_\c$. Accordingly, for any $\c\in S_d^{\text{loc}}(\tilde H, H)$, we set
\begin{equation}\label{eq:WWW'}
W_0=\prod_{\substack{\text{$p\leq M_{d,K}$ or $p\mid h_\c$}}} p^{k(x)} \quad \text{ and } \quad 
W_1=\prod_{\substack{M_{d,K}<p\leq w \\ p\nmid h_\c}}p^{k(x)} .
\end{equation}

We assume throughout this section that  $\tilde H$ satisfies
$H \exp(-\sqrt{\log H})
\leq \tilde H \leq H$, as in ~\eqref{eq:tilde-H}.
We need to study the size of
$$
\hat N_\c(x)=\sum_{n\leq x} \hat R_K(f_\c(n);B),
$$
for
 $B=\tilde H^{1/e}x^{d/e}$ and  $\c\in 
   S_d^{\text{loc}}(\tilde H, H)$.
   We shall prove the  following result in  the next subsection.

\begin{proposition}\label{p:STEP2}
Let $A\geq 1$ be fixed. Take $B=\tilde H^{1/e}x^{d/e}$ and assume that $\tilde H$ lies in the interval~\eqref{eq:tilde-H}.
Let $\delta>  (\log x)^{-A}$.
Then there exists a constant $M=M_A\geq 1$ such that
\begin{align*}
&\#\left\{ \c\in S_d^{\textnormal{loc}}(\tilde H, H): 
\hat N_\c(x) <\delta x\right\} 
 \ll \frac{H^{d+1}}{(\log x)^{dA}}\\
& \qquad +\#
\left\{ \c\in S_d^{\textnormal{loc}}(\tilde H, H):  \sigma_{\c}(W_0) \leq M_A
\delta (\log x)^{c_{d,e}}\right\},
\end{align*}
where $c_{d,e}=e+d^2(d+1)^{e+1}$ and 
\begin{equation}\label{eq:sig-c}
\sigma_\c(W_0)=W_0^{-e}\#\left\{
(\x,t)\in (\ZZ/W_0\ZZ)^{e+1}: \nf_K(\x)\equiv f_\c(t)\bmod{W_0}
\right\}.
\end{equation}
\end{proposition}

\subsection{A transition to non-archimedean densities}

   Recall the definition~\eqref{eq:define_omega} of $\omega(f_\c(n);B)$.  Then we have the   following lower bound. 

\begin{lemma}\label{lem:apple}
Suppose that $x\geq H^\delta$, for some $\delta>0$, and that $H$ is large enough. There exists some small $\eta>0$ 
such that 
$$
\omega(f_\c(n);B)\gg 1, 
$$ 
if $\c\in 
   S_d^{+}(\tilde H, H)$ and 
$
|n/x-(\tfrac{1}{2}\tilde H/c_0)^{1/d}|< \eta^2.
$
\end{lemma}

\begin{proof}
According to part (iv) of Lemma~\ref{lem:props-om}, there exists $\eta>0$ 
such that 
$\omega(y;B)\gg 1$ if $|y/B^e-\frac{1}{2}|<\eta$. But if  $|n/x-(\tfrac{1}{2}\tilde H/c_0)^{1/d}|< \eta^2$, then 
$$
(n/x)^d=\tfrac{1}{2}\tilde H/c_0 +O(\eta^2)
$$
since $c_0>\frac{1}{2}\tilde H$. But then, since $f_\c(n)=c_0n^d+O(Hx^{d-1})$ and $B^e=\tilde H x^d$, it follows that
$$
|f_\c(n)/B^e-\tfrac{1}{2}|\ll \eta^2 +\frac{H}{x\tilde H}.
$$
The right-hand side  is $\ll \eta$ if $\tilde H$ satisfies~\eqref{eq:tilde-H} and $x\geq H^\delta$.
The statement of the lemma follows by choosing $\eta>0$ small enough. 
\end{proof}

Define the interval
$$
I=\left\{t\in \RR: 0<t\leq x, ~|t/x-(\tfrac{1}{2}\tilde H/c_0)^{1/d}|< \eta^2\right\}.
$$
Note that 
$1\geq (\tfrac{1}{2}\tilde H/c_0)^{1/d}\geq 2^{-1/d}$
if $\c\in S_d^{\text{loc}}(\tilde H, H)$. 	
Assuming that $\eta$ is sufficiently small, it follows that 
$I=(\xi_0,\xi_1)$, with 
$$
\xi_0=(\tfrac{1}{2}\tilde H/c_0)^{1/d}x-\eta^2 x \quad\text{ and } \quad
\xi_1= \min\left\{x, (\tfrac{1}{2}\tilde H/c_0)^{1/d}x+\eta^2 x\right\}.
$$
In particular, 
$x\ll \xi_0<\xi_1\leq x$ and
$\xi_1-\xi_0\gg x$, for absolute implied constants. In particular $x\ll |I|\leq x$.

Lemma~\ref{lem:apple} implies that 
$\hat N_\c(x)\gg \tilde N_\c(I)$, 
where
$$
\tilde N_\c(I)= \sum_{n\in I} \gamma(W,f_\c(n)).
$$
Hence $\hat N_\c(x) <\delta x$ implies that $\tilde N_\c(I) \ll \delta x$ and we may conclude that 
\begin{equation}\label{eq:dog}
\#\left\{ \c\in S_d^{\text{loc}}(\tilde H, H): 
\hat N_\c(x) <\delta x\right\} \leq
\#\left\{ \c\in S_d^{\text{loc}}(\tilde H, H): 
\tilde N_\c(I) \ll \delta x\right\}.
\end{equation}
Recall 
the definition 
\eqref{eq:sig-c} of 
$\sigma_\c(W_0)$ and put 
$$
P_\c(w)=\prod_{\substack{M_{d,K}<p\leq w\\ p\nmid h_\c}} p,
$$
for a suitably large constant  $M_{d,K}>0$ that only depends on $d$ and $K$. 
The following result is the key step in the proof of Proposition~\ref{p:STEP2}.

\begin{lemma}\label{lem:tilde-N}
Let $A\geq 1$.
Let $\c\in S_d(\tilde H, H)$, in the notation of~\eqref{eq:Sd+}. 
Recall the definition ~\eqref{eq:r-tilde} of 
$\tilde r_K$.
Suppose that $f_\c$ is separable, has no integer zero, and has content $h_\c\leq (\log x)^A$.
Then
$$
\tilde N_\c(I) = 
\sigma_\c(W_0)|I| 
\sum_{\substack{k\in \NN\\ p\mid k \Rightarrow p\mid P_\c(w) }}
\frac{c_\c(w) g(k) 	
\lambda_{f_\c}(k)}{k}
+O_A\left(
x\exp\left(-(\log x)^{1/8}\right)\right),
$$
where $g=\tilde r_K*\mu$, $\lambda_{f_\c}$ is given by~\eqref{eq:lambda}, 
 $W_0$ is given by ~\eqref{eq:WWW'} and 
$$
c_\c(w)=\prod_{p\mid P_\c(w)} \alpha_p,
$$
with  $\alpha_p$  given by~\eqref{eq:alpha-p}.
\end{lemma}

\begin{proof}
We shall use 
similar techniques to those that we developed in the proof of 
Proposition~\ref{prop:1moment}, writing $W=W_0W_1$ in the notation of~\eqref{eq:WWW'}.
It follows from~\eqref{eq:W0} that 
$$
\prod_{p\leq M_{d,K}} p^{k(x)}\ll (\log x)^{C_1},
$$ 
where $C_1$  depends on the constant $C$ in the definition of $k(x)$, and on both $d$ and the number field $K$.
Similarly, 
$$
\prod_{p\mid h_\c} p^{k(x)}\leq h_\c^{k(x)\omega(h_\c)}\leq 
\exp\left(k(x)(\log h_\c)^2\right)\leq \exp\left(C_2(\log\log x)^3\right),
$$ 
since $h_\c\leq (\log x)^A$, 
for a constant $C_2$ that depends on $A$ and  
on the constant $C$ in the definition of $k(x)$. Hence it follows that 
\begin{equation}\label{eq:oat}
W_0\leq \exp\left(C'(\log\log x)^3\right),
\end{equation}
for a further constant $C'$ that depends on $d,K$ and $A,C$.

Since   $\gamma(W,f_\c(n))=\gamma(W_0,f_\c(n))\gamma(W_1,f_\c(n))$, we obtain
\begin{equation}\label{eq:S-to-T'}
\tilde N_{\c}(I)=
\sum_{\nu \bmod{W_0}} \gamma(W_0,f_\c(\nu))T(I;\nu), 
\end{equation}
where now
$$
T(I;\nu)=
\sum_{\substack{n\in I\\ 
n\equiv \nu\bmod{W_0}}}
\gamma(W_1,f_\c(n)).
$$
It follows from~\eqref{eq:calc} 
that 
$$
\gamma(W_1,f_\c(n))
=\prod_{p\mid P_\c(w)} \alpha_p
\prod_{\substack{p\mid P_\c(w)\\
p^{k(x)}\nmid f_\c(n)}}
r_K(p^{v_p(f_\c(n))}) 
\prod_{\substack{p\mid P_\c(w)\\
p^{k(x)}\mid f_\c(n)
}}
\alpha_p^{-1}\beta_{p^{k(x)}},
$$
where $\alpha_p$ and $ \beta_{p^{k(x)}}$ are given by~\eqref{eq:alpha-p} and ~\eqref{eq:beta-p}, 
respectively.
Hence
\begin{equation}\label{eq:pigeon}
T(I;\nu)=
U(I;\nu)\prod_{p\mid P_\c(w)} \alpha_p
\end{equation}
where
$$
U(I;\nu)=\sum_{\substack{n\in I
\\ n\equiv \nu\bmod{W_0}}} \tilde r_K(f_\c(n)_W),
$$
with the notation~\eqref{eq:r-tilde} for $\tilde r_K$ and where
$
n_W=\prod_{p^\nu \| n,~ p\mid P_\c(w)}p^\nu.
$
It remains to prove an asymptotic formula for $U(I;\nu)$.

 Note that $f_\c(n)\neq 0$ in $U(I,\nu)$, by hypothesis.
 Write 
 $$\mathcal{C}_w=\{n\in \mathbb{N}:\,\ p\mid n\implies p\mid P_\c(w)\}.$$
Let $d=f_\c(n)_W$.  Then there exist an integer $d'$ such that $f_\c(n)=dd'$, with 
$d\in \mathcal{C}_w$ and 
$\gcd(d',P_\c(w))=1$. 
Using M\"obius inversion to detect the coprimality condition, we obtain
\begin{align*}
U(I;\nu)
&=\sum_{\substack{k\in \mathcal{C}_w
}} g(k)
\sum_{\substack{n\in I
\\ n\equiv \nu\bmod{W_0}\\ 
k\mid f_\c(n)}} 1,
\end{align*}
where 
$
g=\tilde r_K* \mu$.
We will  break the outer  sum into two intervals, leading to 
\begin{equation}\label{eq:sum-sig123'}
U(I;\nu)=\Sigma_1+\Sigma_2,
\end{equation}
where
\begin{equation}\label{eq:Sigma1'}
\Sigma_1=
\sum_{\substack{k\in \mathcal{C}_w 
\\ k\leq x^{1/(2d)}}} g(k) 	
\sum_{\substack{n\in I
\\ n\equiv \nu\bmod{W_0}\\
k\mid f_\c(n)}} 1
\end{equation}
and 
$$
\Sigma_2=
\sum_{\substack{k\in \mathcal{C}_w 
\\ 
k>x^{1/(2d)}}} 
g(k) 	
\sum_{\substack{n\in I
\\ n\equiv \nu\bmod{W_0}\\
k\mid f_\c(n)}} 1.
$$

Beginning with the second sum, it follows from~\eqref{eq:size-g} and the inclusion  $I\subset (0,x]$, that 
\begin{align*}
\Sigma_2
&\ll
\sum_{\substack{k\in \mathcal{C}_w 
\\ 
k>x^{1/(2d)}}}
\tau(k)^{e+1}
\sum_{\substack{n\leq x\\ k\mid f_\c(n)}} 1.
\end{align*}
Suppose that $\omega(k)=r$. Then any $k$ in the sum satisfies
$$
x^{1/(2d)}<k=p_1^{j_1}\cdots p_{r}^{j_r}\leq w^{k(x)r},
$$
whence $\omega(f_\c(n))\geq \omega(k)=r > K_0$, with 
$$
K_0=\frac{\log x}{2 k(x)d\log w}=\frac{\sqrt{\log x}}{2k(x)d}.
$$
But then it follows that $1\leq 2^{-K_0} 2^{\omega(f_\c(n))}\leq 2^{-K_0} \tau(|f_\c(n)|)$, whence
$$
\Sigma_2
\ll
\frac{1}{2^{K_0}} 
\sum_{\substack{n\leq x}} \tau(|f_\c(n)|)^{e+2}.
$$
Since $f_\c$ has content  $h_\c\leq (\log x)^A$, 
it now follows from Lemma~\ref{lem:henriot} that 
\begin{equation}\label{eq:Sigma2-''}
\Sigma_2
\ll_A
\frac{x(\log x)^{K_1}}{2^{K_0}} 
\ll x\exp\left(-(\log x)^{1/4}\right),
\end{equation}
for a suitable constant $K_1$ only depending on $d,e$ and $A$.

Turning to the sum~\eqref{eq:Sigma1'}, 
it follows from~\eqref{eq:nature-g} that 
 $\gcd(k,W_0)=1$ 
and $v_p(k)\leq k(x)$, 
for any $k\in \mathcal{C}_w$ such that $g(k)\neq 0$. 
 Breaking into residue classes modulo $k$, we deduce that 
\begin{align*}
\Sigma_1
&=
\sum_{\substack{k\in \mathcal{C}_w\\ k\leq x^{1/(2d)}}} g(k) 	
\sum_{\substack{\alpha \bmod{k}\\ f_\c(\alpha)\equiv 0 \bmod{k}}}
\sum_{\substack{n\in I
\\ n\equiv \nu\bmod{W_0}\\
n\equiv \alpha\bmod{k}}}1\\
&=
\sum_{\substack{k\in \mathcal{C}_w\\ k\leq x^{1/(2d)}}} g(k) 	
\lambda_{f_\c}(k)
\left(\frac{|I|}{W_0k} +O(1)\right),
\end{align*}
where $\lambda_{f_\c}$ is given by~\eqref{eq:lambda}.
Since $f_\c$ has content  $h_\c\leq (\log x)^A$, 
 Lemma~\ref{lem:stewart} implies that $\lambda_{f_\c}(k)=O_\ve(k^{1-1/d+\ve}(\log x)^{A/d})$, for any $\ve>0$.
We have already seen in~\eqref{eq:size-g} that  $|g(k)|\leq \tau(k)^{e+1}=O_\ve(k^{\ve})$.
Hence
\begin{equation}\label{eq:Sigma1''}
\Sigma_1
=
\frac{|I|}{W_0}
\sum_{\substack{k\in \mathcal{C}_w\\ k\leq x^{1/(2d)}}} \frac{g(k) 	
\lambda_{f_\c}(k)}{k}
+O_{A,\ve}\left(x^{1-1/(2d)+\ve}\right).
\end{equation}
Since $|I|\leq x$, we see that 
$$
\frac{|I|}{W_0}
\sum_{\substack{k\in \mathcal{C}_w\\ k\leq x^{1/(2d)}}} \frac{g(k) 	
\lambda_{f_\c}(k)}{k}=
\frac{|I|}{W_0}
\sum_{\substack{k\in \mathcal{C}_w}} \frac{g(k) 	
\lambda_{f_\c}(k)}{k}+O\left( x\Upsilon(x)\right),
$$
where
\begin{equation}\label{eq:robin}
\Upsilon(x)=\sum_{\substack{k\in \mathcal{C}_w\\  k> x^{1/(2d)}}}  
 \frac{|g(k)| 
\lambda_{f_\c}(k)}{k}.
\end{equation}

We will show momentarily that 
\begin{equation}\label{eq:dove}
\Upsilon(x)\ll_A \exp\left(-\frac{\sqrt{\log x}}{16d}\right).
\end{equation}
From~\eqref{eq:dove} it will follow that 
$$
\Sigma_1
=\frac{|I|}{W_0}
\sum_{\substack{k\in \mathcal{C}_w}} \frac{g(k) 	
\lambda_{f_\c}(k)}{k}+O_A\left( x\exp\left(-\frac{\sqrt{\log x}}{16d}\right)
\right),
$$
in~\eqref{eq:Sigma1''}. Combining this with~\eqref{eq:Sigma2-''} in~\eqref{eq:sum-sig123'}, we obtain 
$$
U(I;\nu)
=\frac{|I|}{W_0}
\sum_{\substack{k\in \NN\\
p \mid  k \Rightarrow p\mid P_\c(w) }}
\hspace{-0.3cm}
 \frac{g(k) 	
\lambda_{f_\c}(k)}{k}+O_A\left( x\exp\left(-(\log x)^{1/4}\right)\right).
$$ 
We now insert this into~\eqref{eq:pigeon} and~\eqref{eq:S-to-T'}, noting that 
$c_\c(w)\leq \varphi^*(W)\ll \sqrt{\log x}$ and 
$$
\sum_{\nu \bmod{W_0}} \frac{\gamma(W_0,f_\c(\nu))}{W_0}=\sigma_{\mathbf{c}}(W_0),
$$
in the notation of ~\eqref{eq:sig-c}. 
The main term therefore agrees with the statement of the lemma.  In view of the upper bound 
\eqref{eq:oat} for $W_0$, we easily see that 
$$
c_\c(w)x\exp\left(-(\log x)^{1/4}\right)
\sum_{\nu \bmod{W_0}} \frac{\gamma(W_0,f_\c(\nu))}{W_0} 
\ll x\exp\left(-(\log x)^{1/8}\right).
$$
This 
thereby completes the proof of the lemma, subject to the verification of the  upper bound~\eqref{eq:dove} for $\Upsilon(x)$.

Returning to~\eqref{eq:robin}, we write $k=k_1k_2$ for coprime $k_1,k_2$, where $k_1$ is square-free and $k_2$ is square-full. 
Moreover, we may assume that $v_p(k_2)\leq k(x)$ for any prime $p\mid k_2$, since otherwise $g(k)=0$.
Furthermore any $p\mid k_1k_2$ is coprime to the content of $f_\c$, since $k\in \mathcal{C}_w$.
Note that  $\lambda_{f_\c}(k)=\lambda_{f_\c}(k_1)\lambda_{f_\c}(k_2)$ and 
$$
\lambda_{f_\c}(k_1)
\leq \prod_{p\| k_1}d
\leq d^{\omega(k_1)}\leq \tau(k_1)^d,
$$ 
by  Lemma~\ref{lem:stewart}.
Hence, on 
applying~\eqref{eq:size-g} to get 
 $|g(k)|\leq \tau(k)^{e+1}$, 
we deduce that 
$$
\Upsilon(x)\ll  (\log x)^A
\sum_{\substack{k_2\in \mathcal{C}_w\\
p^\nu\| k_2\Rightarrow 2\leq \nu\leq k(x)}}
 \frac{\tau(k_2)^{e+1}
\lambda_{f_\c}(k_2)}{k_2} 
\Sigma_{d+e+1}\left(\frac{x^{1/(2d)}}{k_2}\right),
$$
in the notation of~\eqref{eq:definition-sig0}.  Appealing to Lemma~\ref{lem:gouda} , we deduce that 
\begin{align*}
\Sigma_{d+e+1}\left(\frac{x^{1/(2d)}}{k_2}\right)
&\ll \exp\left(-\frac{\log x}{8d\log w}\right)
\exp\left(\frac{\log k_2}{4\log w}\right)\\
&
\leq \exp\left(-\frac{\sqrt{\log x}}{8d}\right) 
k_2^{1/(4\log w)},
\end{align*}
since 
$\log w=\sqrt{\log x}$.
Hence we obtain 
$$
\Upsilon(x)\ll 
 \exp\left(-\frac{\sqrt{\log x}}{8d}\right) (\log x)^{A} 
 \sum_{\substack{k_2\in \mathcal{C}_w\\
p^\nu\| k_2\Rightarrow 2\leq \nu\leq k(x)}}
 \frac{\tau(k_2)^{e+1}
\lambda_{f_\c}(k_2)}{k_2^{1-1/(4\log w)}
}.
$$
If $p\leq w$ then we have $p^{-1+1/(4\log w)}\leq \frac{3}{2}p^{-1}$. Hence it  follows from multiplicativity and Lemma~\ref{lem:stewart} that 
\begin{align*}
\sum_{\substack{k_2\in \mathcal{C}_w\\
p^\nu\| k_2\Rightarrow 2\leq \nu\leq k(x)}}
\hspace{-0.8cm}
& \frac{\tau(k_2)^{e+1}
\lambda_{f_\c}(k_2)}{k_2^{1-1/(4\log w)}}\\
&\leq \prod_{p\leq w}\left(1+\sum_{2\leq j\leq d} \frac{3d(j+1)^{e+1}}{2p} +\sum_{j>d} \frac{d(j+1)^{e+1}(3/2)^j}{p^{j/d}} \right)\\
&=\prod_{p\leq w}\left(1+O\left(\frac{1}{p}\right) \right)\\
&\ll (\log x)^B,
\end{align*}
for a constant $B$ depending only on $d$ and $e$. The desired upper bound~\eqref{eq:dove} for $\Upsilon(x) $ easily follows, which thereby completes the proof of the lemma. 
 \end{proof}
  
It turns out that we can directly lower bound the sum over  $k$ in 
the previous lemma. This is the object of the following result.

\begin{lemma}\label{lem:lower-ksum}
Let $f\in \mathbb{Z}[t]$ be a  
 separable polynomial of degree $d$ and  content $h$, and let
$$
P=\prod_{\substack{M_{d,K}<p\leq w\\ p\nmid h}} p.
$$
Then 
we have 
$$
\prod_{p\mid P} \alpha_p
\sum_{\substack{k\in \NN\\
p\mid  k \Rightarrow p\mid 
P}}
\frac{g(k) 	
\lambda_{f}(k)}{k}\gg (\log x)^{-c_{d,e}},
$$
where $\alpha_p$ is given by~\eqref{eq:alpha-p}
and 
$c_{d,e}=e+d^2(d+1)^{e+1}$.
\end{lemma}

\begin{proof}
It  follows from  multiplicativity that 
$$\prod_{p\mid P} \alpha_p
\sum_{\substack{k\in \NN\\
p\mid  k \Rightarrow p\mid P}}
\frac{g(k) 	
\lambda_{f}(k)}{k}
=
\prod_{p\mid P} \alpha_p\xi_p,
$$
where $\alpha_p$ is given by~\eqref{eq:alpha-p} and 
\begin{align*}
\xi_p&=1+\sum_{j\geq 1} \frac{g(p^j)\lambda_{f}(p^j)}{p^j}\\
&\geq 1-\sum_{j\geq 1}\frac{d\tau(p^j)^{e+1}\min\{p^{j-1}, p^{j(1-1/d)}\}
}{p^j}\\
&= 1-\sum_{j=1}^{d} \frac{d(j+1)^{e+1}}{p}
-
\sum_{j\geq d+1}\frac{d(j+1)^{e+1}}{p^{j/d}},
\end{align*}
by~\eqref{eq:size-g} and Lemma~\ref{lem:stewart}.
But then, it follows from~\eqref{eq:clip} that 
$$
\alpha_p\xi_p\geq 1- \frac{e+d^2(d+1)^{e+1}}{p} +O\left(\frac{1}{p^{1+1/d}}\right).
$$
On assuming that $M_{d,K}$ is sufficiently large, the statement of the lemma readily follows from Mertens' theorem. 
\end{proof}

\begin{proof}[Proof of Proposition~\ref{p:STEP2}]
Let $\c\in S_d(\tilde H, H)$ and 
suppose that $f_\c$ is separable, has no integer zero, and has content $\leq (\log x)^{A}$. Recall 
that $|I|\gg x$.
Combining Lemmas~\ref{lem:tilde-N} and~\ref{lem:lower-ksum}, we  deduce that 
\begin{equation}\label{eq:cap}
\tilde N_\c(I) \gg 
\sigma_\c(W_0) \frac{ x}{(\log x)^{c_{d,e}}}
+
O\left( x\exp\left(-(\log x)^{1/8}\right)\right),
\end{equation}
where $\sigma_\c(W_0)$ is given by~\eqref{eq:sig-c}.
We would now like to apply this in ~\eqref{eq:dog}. To do so we must  handle separately the contribution from those $\c \in S_d^{\text{loc}}(\tilde H, H)$ which our argument fails to handle; viz.\ those $\c$ for which 
$f_\c$ either fails to be separable, or has an integer root, or has content greater than $(\log x)^A$. 

Let $D=D(c_0,\dots,c_d)$ denote the discriminant of $f_\c$, viewed as a form of degree $d(d-1)$ in the coefficients of $f_\c$. There 
 are $O(H^{d})$ choices of 
$\c \in S_d^{\text{loc}}(\tilde H, H)$  for which $D(\c)=0$. Similarly, 
on appealing to work of Kuba~\cite{kuba},   
there  are $O(H^d\log H)$ choices of 
$\c \in S_d^{\text{loc}}(\tilde H, H)$ for which $f_\c$ is reducible over $\QQ$. It follows that there are 
 $O(H^d\log H)$ choices of 
$\c \in S_d^{\text{loc}}(\tilde H, H)$ 
for which $f_\c$ has an integer root, since these are the polynomials which admit a linear factor over $\ZZ$.  
Finally, the number of vectors in 
$S_d^{\text{loc}}(\tilde H, H)$  with common divisor exceeding  $(\log x)^A$ is easily seen to be
\begin{align*}
&\leq \sum_{h>(\log x)^A} \#\{\c\in \ZZ^{d+1}: c_0>0, ~|\c|\leq H, h\mid \c\}\\ 
&\ll \sum_{h>(\log x)^A} \left(\frac{H}{h}\right)^{d+1}\\
&\ll \frac{H^{d+1}}{(\log x)^{dA}}.
\end{align*}
The statement of Proposition~\ref{p:STEP2} now  follows
from ~\eqref{eq:cap} easily.
\end{proof}

\subsection{The non-archimedean densities are rarely small  }

Let $\Delta>0$. (In our application we shall take 
$\Delta$ to be of order $\delta (\log x)^{c_{d,e}}$.)
In this section we seek to provide an  upper bound for 
\begin{equation}\label{eq:Mdelta}
M(\tilde H, H;\Delta)=\#\left\{ \c\in S_d^{\text{loc}}(\tilde H, H):  \sigma_\c(W_0) <\Delta
\right\}, 
\end{equation}
where $W_0$ is given by~\eqref{eq:WWW}, 
$S_d^{\text{loc}}(\tilde H, H)$ is defined in~\eqref{eq:Sdloc} and 
$\sigma_\c(W_0)$ is given by~\eqref{eq:sig-c}.
We shall assume that $\tilde H$ satisfies~\eqref{eq:tilde-H}.
Broadly speaking, our strategy for bounding~\eqref{eq:Mdelta} is inspired by  the work in~\cite[\S 5]{fano}.

Since 
$\sigma_\c(W_0)$ only depends on the 
residue class of $\c$ modulo $W_0$, we may break  the $\c$ into residue classes modulo $W_0$.
Let $U(W_0)\subset (\ZZ/W_0\ZZ)^{d+1}$ be the 
image of the set $S_d^{\text{loc}}(\tilde H, H)$ under reduction modulo $W_0$. 
Then 
\begin{align*}
M(\tilde H,H;\Delta)
&\leq
\sum_{
\substack{ 
 \u\in U(W_0) \\ 
 \sigma_\u(W_0) <\Delta}} \#\left\{\c\in \ZZ^{d+1}: |\c|\leq H, ~\c\equiv \u\bmod{W_0}
\right\}\\
&\ll
\sum_{
\substack{ 
 \u\in U(W_0) \\ 
 \sigma_\u(W_0) <\Delta}}\frac{\tilde H}{W_0} \left(\frac{H}{W_0}\right)^{d},
\end{align*} 
since 
\eqref{eq:W0} 
and our assumption
\eqref{eq:tilde-H} 
together ensure that 
$H\geq \tilde H \gg W_0$.
It follows from Rankin's trick that 
\begin{align*}
M(\tilde H,H;\Delta)
&\ll\frac{\Delta^{\kappa}\tilde H H^d}{W_0^{d+1} }
\sum_{\u\in U(W_0)} \frac{1}{ \sigma_\u(W_0)^{\kappa}},
\end{align*}
for any $\kappa>0$.

For any prime power $p^k$, let  $U(p^k)$ denote the set of 
$\u\in (\ZZ/p^k\ZZ)^{d+1}$ such that there 
exists  $(\x_0,t_0)\in \ZZ_p^{e+1}$ for which $\nf_K(\x_0)\equiv f_\u(t_0) \bmod{p^k}$.
Then 
it follows from multiplicativity that 
\begin{equation}\label{eq:st-post}
M(\tilde H,H;\Delta)
\ll\frac{\Delta^{\kappa}\tilde H H^d}{W_0^{d+1} }
\prod_{p\mid W_0}
\sum_{\u\in U(p^{k(x)})} \frac{1}{ \sigma_\u(p^{k(x)})^{\kappa}},
\end{equation}
Moreover, 
$$
\sigma_\u(p^k)=p^{-ke}\#\left\{
(\x,t)\in (\ZZ/p^k\ZZ)^{e+1}: \nf_K(\x)\equiv f_\u(t)\bmod{p^k}
\right\}.
$$

Our next result provides a good lower bound for 
$\sigma_\u(p^k)$ whenever  the equation $\nf_K(\x)=f_\u(t)$ admits a solution over $\ZZ_p$ which is not too singular. 

\begin{lemma}\label{lem:0.3}
Let $p^k$ be a prime power and let 
  $\u\in (\ZZ/p^k\ZZ)^{d+1}$.
Assume that there exists 
$\alpha\in \ZZ_{\geq 0}$ and 
$ (\x_0,t_0)\in \ZZ_p^{e+1}$ such that 
\begin{equation}\label{eq:goat}
\nf_K(\x_0)\equiv f_\u(t_0) \bmod{p^k}\quad \text{ and }\quad
p^\alpha\| (\nabla \nf_K(\x_0), f_\u'(t_0)).
\end{equation}
Then
$
\sigma_\u(p^k)\geq p^{-(\alpha+1)e}.
$
\end{lemma}

\begin{proof}
We take the lower bound
$$
\sigma_\u(p^k)\geq p^{-ke}\#\left\{
(\x,t)\in (\ZZ/p^k\ZZ)^{e+1}: 
\begin{array}{l}
\nf_K(\x)\equiv f_\u(t)\bmod{p^k}\\
(\x,t)\equiv (\x_0,t_0) \bmod{p^{\alpha+1}}
\end{array}
\right\}.
$$
If $k\leq \alpha$ then $\sigma_\u(p^k)\geq p^{-ke}\geq p^{-\alpha e}$, which is satisfactory.
If $\alpha< k\leq 2\alpha$ then 
we write  $(\x,t)= (\x_0,t_0)+p^{\alpha+1}(\x_1,t_1)$, where $(\x_1,t_1)$ runs modulo 
$p^{k-\alpha-1}$.  An application of Taylor's theorem implies that 
any choice of $(\x_1,t_1)$ contributes to the right-hand side, since 
 $p^\alpha\mid   (\nabla \nf_K(\x_0), f_\u'(t_0))$. Hence
$$
\sigma_\u(p^k)\geq p^{-ke} \cdot p^{(k-\alpha-1)(e+1)}\geq  p^{-(\alpha+1) e},
$$
since $k\geq \alpha+1$.
Finally, if $k\geq 2\alpha+1$ then 
we can estimate
$\sigma_\u(p^k)$ from below using Hensel's lemma, in the form of~\cite[Lemma 3.3]{nf}. 
It easily follows that $\sigma_\u(p^k)$ is 
\begin{align*}
&\geq p^{-(2\alpha+1)e}
\#\left\{
(\x,t)\in (\ZZ/p^{2\alpha+1}\ZZ)^{e+1}: 
\begin{array}{l}
\nf_K(\x)\equiv f_\u(t)\bmod{p^{2\alpha+1}}\\
(\x,t)\equiv (\x_0,t_0) \bmod{p^{\alpha+1}}
\end{array}
\right\}\\
&\geq p^{-(2\alpha+1)e}\cdot p^{\alpha (e+1)}\\
&\geq  p^{-(\alpha+1)e}.
\end{align*}
Taking these estimates together, the proof of the lemma is complete.
\end{proof}

We are now ready to record our final bound for
$M(\tilde H,H;\Delta)$ in~\eqref{eq:Mdelta}.

\begin{proposition}\label{prop2}
We have 
$
M(\tilde H,H;\Delta)
\ll  \Delta^{\frac{1}{d^2e}}{\tilde H H^{d}}.
$
\end{proposition}

\begin{proof}
Our 
starting point is~\eqref{eq:st-post}. 
Let $p\mid W_0$, so that $p\leq M_{d,K}=O(1)$.
We stratify the sets 
$U(p^{k(x)})$ according to the $p$-adic valuation of the vector 
$ (\nabla \nf_K(\x_0), f_\u'(t_0))$. 
For $0\leq \alpha\leq k(x)$, let $U_\alpha(p^{k(x)})$ be the set of $\u\in (\ZZ/p^{k(x)}\ZZ)^{d+1}$ such that  $p\nmid \u$, and for which there exists
$ (\x_0,t_0)\in \ZZ_p^{e+1}$ such that~\eqref{eq:goat} holds with $k=k(x)$.
(Note that $U_{k(x)}(p^{k(x)})$ should actually be defined with the condition that 
$p^{k(x)}\mid  (\nabla \nf_K(\x_0), f_\u'(t_0))$ in~\eqref{eq:goat}.)
Then we have  
$$
U(p^{k(x)})=U_0(p^{k(x)})\sqcup U_1(p^{k(x)})\sqcup \cdots 
\sqcup
U_{{k(x)}}(p^{k(x)}).
$$
It follows from Lemma~\ref{lem:0.3} that  
that 
\begin{align*}
\sum_{\u\in U(p^{k(x)})} \frac{1}{\sigma_\u(p^{k(x)})^\kappa}
\leq \sum_{0\leq \alpha\leq {k(x)}} 
p^{(\alpha+1)e\kappa}\#U_\alpha(p^{k(x)}).
\end{align*}

If $U_\alpha(p^{k(x)})$ is non-empty, then there exists 
$ (\x_0,t_0)\in \ZZ_p^{e+1}$ such that
$$
\nf_K(\x_0)\equiv f_\u(t_0) \bmod{p^{k(x)}}\quad \text{ and }\quad
p^\alpha\mid  (\nabla \nf_K(\x_0), f_\u'(t_0)).
$$
Euler's theorem implies that 
$e\nf_K(\x)=\x\cdot\nabla \nf_K(\x)$. Hence it follows that 
$$
v_p(f_\u'(t_0))\geq \alpha \quad \text{ and } \quad 
v_p(f_\u(t_0))\geq \alpha-v_p(e).
$$
But the discriminant of $f_\c$ is a homogeneous polynomial
$D\in \ZZ[c_0,\dots,c_d]$ 
of degree $d(d-1)$, which is  obtained by calculating the resultant of $f_\c$ and $f_\c'$. Thus we may conclude that 
$$
D(u_0,\dots,u_d)\equiv 0 \bmod{p^{\alpha-v_p(e)}},
$$
whenever $\u\in U_\alpha(p^{k(x)})$.
For any $j\geq 1$, it follows from work of Pierce, Schindler and Wood~\cite[Lemma~4.10]{wood} that 
$$
\#\left\{\u\in (\ZZ/p^j\ZZ)^{d+1}: D(\u)\equiv 0\bmod{p^j}\right\}\ll 	p^{j(d+1-\frac{1}{d(d-1)})}.
$$
Hence
$$
\#U_\alpha (p^{k(x)})\ll  p^{{k(x)}(d+1)-\frac{\alpha}{d(d-1)}},
$$
since $p\leq M_{d,K}$.
Finally, we conclude 
that 
\begin{align*}
\sum_{\u\in U(p^{k(x)})} \frac{1}{\sigma_\u(p^{k(x)})^\kappa}
\ll \sum_{0\leq \alpha\leq {k(x)}} 
p^{(\alpha+1)e\kappa}\cdot p^{{k(x)}(d+1)-\frac{\alpha}{d(d-1)}}\ll p^{{k(x)}(d+1)},
\end{align*}
provided that $e\kappa<\frac{1}{d(d-1)}$ and $p\mid W_0=O(1)$.
Clearly 
$$
\kappa=\frac{1}{d^2e}
$$
is  satisfactory. 
Making this choice, we 
return  to~\eqref{eq:st-post} and therefore arrive  at the statement of the  proposition.
\end{proof}

\begin{proof}[Proof of Proposition~\ref{p:STEP2'}]
This follows from
Propositions~\ref{p:STEP2} and~\ref{prop2}.
\end{proof}


\begin{thebibliography}{99}


\bibitem{baier-zhao}
S. Baier and L. Zhao, Primes in quadratic progressions on average. {\em Math. Ann.} {\bf 338} (2007), 963--982. 

\bibitem{rome}
F. Balestrieri and
N. Rome,
Average Bateman--Horn for Kummer polynomials. 
{\em Acta Arith.} {\bf 207} (2023), 315--350.


\bibitem{berg}
J. Berg, 
Obstructions to integral points on affine Ch\^atelet surfaces. {\em Preprint}, 2017.\\
(arXiv:1710.07969)

\bibitem{BP}
E. Bombieri and J. Pila, The number of integral points on 
arcs and ovals. {\em Duke Math. J.} {\bf 59} (1989), 337--357.


\bibitem{BB}
R. de la Bret\`eche and T.D. Browning, 
 Sums of arithmetic functions over values of binary forms.
{\em Acta Arith.} {\bf 125} (2006), 291--304.

\bibitem{wa}
M. Bright, T.D. Browning and D. Loughran,
 Failures of weak approximation in families.
{\em Compositio Math.} {\bf 152} (2016), 1435--1475.






\bibitem{gafa}  T.D. Browning and D.R. Heath-Brown,
{Quadratic polynomials represented by norm forms.} {\em Geom. Funct. Anal.} {\bf 22} (2012), 1124--1190.

\bibitem{fano} 
T.D. Browning, P. Le Boudec and W. Sawin, 
The Hasse principle for random Fano hypersurfaces. {\em Annals of Math.}
{\bf 197} (2023), 1115--1203.


\bibitem{nf} T.D Browning and L. Matthiesen,
Norm forms for arbitrary number fields as products of linear polynomials.
{\em Ann. Sci. \'Ecole Norm. Sup.} {\bf 50} (2017), 1375--1438. 


\bibitem{chowla}
S. Chowla, {\em The {R}iemann hypothesis and {H}ilbert's tenth problem},
Mathematics and Its Applications, Vol. 4, Gordon and Breach Science Publishers, New York--London--Paris,
1965.

\bibitem{CTH}
J.-L. Colliot-Th\'el\`ene and D. Harari,
Approximation forte en famille. {\em J. reine angew. Math.} 
{\bf 710} (2016), 173--198.

\bibitem{sansuc}
J.-L. Colliot-Th\'el\`ene and J.-L. Sansuc, 
Sur le principe de Hasse et l'approximation faible, et sur une hypoth\`ese de Schinzel. 
{\em Acta Arith.} {\bf 41} (1982), 33–53. 


\bibitem{crelle1}
J.-L. Colliot-Th\'el\`ene, J.-J. Sansuc and P. Swinnerton-Dyer, 
Intersections of two quadrics and Ch\^atelet surfaces, I.  
{\em J. reine angew. Math.}  {\bf 373}  (1987), 37--107. 


\bibitem{crelle2}
J.-L. Colliot-Th\'el\`ene, J.-J. Sansuc and P. Swinnerton-Dyer, 
Intersections of two quadrics and Ch\^atelet surfaces, II.  
{\em J. reine angew. Math.}  {\bf 374}  (1987), 72--168. 

\bibitem{lead}
J.-L. Colliot-Th\'el\`ene, A.N. Skorobogatov and P. Swinnerton-Dyer, Rational
points and zero-cycles on fibred varieties: Schinzel's hypothesis and Salberger's device.
{\em J.\ reine angew.\ Math.} {\bf 495} (1998), 1--28.

\bibitem{xu}
J.-L. Colliot-Th\'el\`ene and F. Xu, Brauer--Manin obstruction for integral points of homogeneous spaces
and representation by integral quadratic forms. {\em Compos. Math.} {\bf 145} (2009), 309--363.



\bibitem{DW}
U. Derenthal and D. Wei, Strong approximation and descent. {\em J. reine angew. Math.} {\bf 731} (2017), 
235--258.

\bibitem{foo-zhao}
T. Foo and L. Zhao, On primes represented by cubic polynomials. {\em Math. Z.} {\bf 274} (2013), 323--340.


\bibitem{granville-mollin}
A. Granville and  R.A. Mollin,
Rabinowitsch revisited.
{\em Acta Arith.} {\bf 96} (2000), 139--153.

\bibitem{gundlach}
F. Gundlach, Integral Brauer--Manin obstructions for sums of two squares and a power. {\em J.\ Lond.\ Math.\ Soc.}  {\bf 88} (2013), 599--618. 

\bibitem{harari}
D. Harari, 
Le d\'efaut d'approximation forte pour les groupes alg\'ebriques commutatifs. 
{\em Algebra \& Number Theory} {\bf 2} (2008), 595--611.

\bibitem{harper}
A.J. Harper,
Minor arcs, mean values, and restriction theory for exponential sums over smooth numbers. 
{\em Compos. Math.} {\bf 152} (2016), 1121--1158.

\bibitem{adolf}
A. Hildebrand and G. Tenenbaum, 
Integers without large prime factors.
{\em J. Th\'eor. Nombres Bordeaux} {\bf 5} (1993), 411--484.

\bibitem{huxley}
M.N. Huxley, Large values of Dirichlet polynomials, III. {\em Acta Arith.} {\bf 26} (1974), 435--444.

\bibitem{huxley-inventiones} M.N. Huxley, On the difference between consecutive primes. {\em Invent. Math.} {\bf 15} (1972), 164--170. 


\bibitem{iw-kow}
H. Iwaniec and E.  Kowalski,
{\em Analytic number theory}.
American Mathematical Society Colloquium Publications {\bf 53},
AMS, Providence, RI, 2004.

\bibitem{jutila} 
M. Jutila. On Linnik’s constant. {\em Math. Scand.} {\bf 41} (1977), 45--62.

\bibitem{kuba}
G. Kuba, On the distribution of reducible polynomials. {\em Math. Slovaca} {\bf 59} (2009), 349--356.

\bibitem{landreau}
B. Landreau, 
A new proof of a theorem of van der Corput.
{\em Bull. London Math. Soc.} {\bf 21} (1989), 
366--368.

\bibitem{kaisa}
K. Matom\"aki, M. Radziwiłł, T. Tao,
An averaged form of Chowla's conjecture. {\em 
Algebra \& Number Theory} {\bf 9} (2015), 2167--2196.


\bibitem{vlad}
V. Mitankin, 
Integral points on generalised affine Ch\^atelet surfaces. 
{\em Bull. Sci. Math.} {\bf 159} (2020), 102830, 20 pp.

\bibitem{montgomery-vaughan}
H.L. Montgomery and R.C. Vaughan, 
{\em Multiplicative number theory. {I}. {C}lassical theory}.
Cambridge Studies in Advanced Mathematics {\bf 97},
CUP, Cambridge, 2007.

\bibitem{moto} Y. Motohashi, On the sum of the {M}\"{o}bius function in a short segment. {\em Proc. Japan Acad.} {\bf 52} (1976), 477--479.

\bibitem{NT}
M. Nair and G.  Tenenbaum, 
Short sums of certain arithmetic functions. 
{\em Acta Math.} {\bf 180} (1998), 119--144.

\bibitem{wood}
L.B. Pierce,  D. Schindler and M.M. Wood,
Representations of integers by systems of three quadratic forms. {\em J. London Math. Soc.} {\bf 113} (2016), 289--344.

\bibitem{rama}
K. Ramachandra, 
Some problems of analytic number theory.
{\em Acta Arith.} {\bf 31} (1976), 313--324.

\bibitem{schinzel}
A. Schinzel and W. Sierpi\'nski,  Sur certaines hypoth\`eses concernant les nombres premiers. {\em Acta Arith.} {\bf 4}
(1958), 185--208; Errata, ibid. {\bf 5} (1959), 259.

\bibitem{shiu}
P. Shiu, A {B}run--{T}itchmarsh theorem for multiplicative functions.
{\em J.\ reine angew.\ Math.}  {\bf 313} (1980), 161--170.

\bibitem{SS}
A.N. Skorobogatov and E. Sofos,
Schinzel Hypothesis with probability 1 and rational points. {\em Invent. Math.}, {\bf 231} (2023), 673--739.

\bibitem{stewart}
C.L. Stewart, On the number of solutions of polynomial congruences and Thue equations. {\em  J.
Amer. Math. Soc.} {\bf 4}  (1991), 793--835.

\bibitem{tao-chowla}
T. Tao, The logarithmically averaged {C}howla and {E}lliott
              conjectures for two-point correlations.
{\em Forum Math. Pi} {\bf 4} (2016), e8, 36pp.

\bibitem{tt-ant}
T. Tao and J. Ter{\"a}v{\"a}inen, 
The structure of correlations of multiplicative functions at almost all scales, with applications to the Chowla and Elliott conjectures. {\em Algebra \& Number Theory} {\bf 13} (2019), 2103--2150. 


\bibitem{tt-duke}
T. Tao and J. Ter{\"a}v{\"a}inen, 
The structure of logarithmically averaged correlations of multiplicative functions, with applications to the {C}howla and {E}lliott conjectures. {\em Duke Math. J.} {\bf 168} (2019), 
1977--2027.


\bibitem{tera-polynomial} 
J. Ter{\"a}v{\"a}inen, 
On the Liouville function at polynomial arguments. {\em Amer. J. Math.} {\bf 146} (2024), 1115--1167.

\bibitem{zhou}
N.H. Zhou, 
Primes in higher-order progressions on average. 
{\em Int. J. Number Theory} {\bf 14} (2018), 1943--1959.


\end{thebibliography}
\end{document}